\documentclass[10pt,a4paper]{article}     
\usepackage{setspace,graphicx,amsmath,amssymb,graphics,float,fullpage,multirow,fancyhdr,amsthm,color,url,enumerate,subfig,pgf,tikz,mathrsfs,setspace}                       
\usepackage[sort&compress]{natbib}                                       
\usepackage[toc,page]{appendix}

\begin{document}             
\bibpunct{[}{]}{,}{n}{,}{;}
\def\Pb{\mathbb{P}}
\def\Eb{\mathbb{E}}
\def\Rb{\mathbb{R}}
\def\Qb{\mathbb{Q}}
\def\Nb{\mathbb{N}}
\def\Zb{\mathbb{Z}}
\def\Tc{\mathcal{T}}
\def\Rc{\mathcal{R}}
\def\Sc{\mathcal{S}}
\def\Dc{\mathcal{D}}
\def\Jc{\mathcal{J}}
\def\Fc{\mathcal{F}}
\def\Lc{\mathcal{L}}
\def\Mc{\mathcal{M}}
\def\Hc{\mathcal{H}}
\def\Bc{\mathcal{B}}
\def\Yc{\mathcal{Y}}
\def\Pc{\mathcal{P}}
\def\Ic{\mathcal{I}}
\def\Kc{\mathcal{K}}
\def\Pr{\mathbf{P}}
\def\Er{\mathbf{E}} 
\def\Gk{\mathscr{G}}
\def\Ek{\mathscr{E}}
\def\Fs{\mathscr{F}}
\def\Pt{\mathit{P}}
\def\Et{\mathit{E}}
\def\d{\mathrm{d}}
\def\ind{\mathbf{1}}
\def\nin{n \rightarrow \infty}
\def\minf{m \rightarrow \infty}
\def\limn{\lim_{\nin}}
\def\limt{\lim_{t\rightarrow \infty}}
\def\limk{\lim_{k \rightarrow \infty}}
\def\limsn{\limsup_{\nin}}
\def\limin{\limsup_{\nin}}
\def\lime{\lim_{\varepsilon \rightarrow 0}}
\def\iid{\stackrel{\text{\tiny{i.i.d.}}}{\sim}}
\def\cd{\stackrel{\text{\tiny{d}}}{\rightarrow}}
\def\cl{\stackrel{\Lc}{\rightarrow}}
\def\ed{\stackrel{\text{\tiny{d}}}{=}}
\def\th{^{\text{th}}}
\def\qqquad{\qquad \quad}
\def\qqqquad{\qqquad \quad}
\def\qqqqquad{\qqqquad \quad}
\def\qqqqqquad{\qqqqquad \quad}
\def\qqqqqqquad{\qqqqqquad \quad}
\def\qqqqqqqquad{\qqqqqqquad \quad}
\def\qqqqqqqqquad{\qqqqqqqquad \quad}
\usetikzlibrary{arrows}
\newtheorem{dfn}{Definition}
\newtheorem{thm}{Theorem}
\newtheorem{lem}{Lemma}[section]
\newtheorem{prp}[lem]{Proposition}
\newtheorem{cly}[lem]{Corollary}
\newtheorem{conj}[lem]{Conjecture}
\newtheorem{fig}{Figure}
\pagenumbering{arabic} 
\singlespacing
\allowdisplaybreaks
\title{Escape regimes of biased random walks on Galton-Watson trees}
\author{Adam Bowditch, University of Warwick}
\date{}
\maketitle
\numberwithin{equation}{section}
\begin{abstract}We study biased random walk on subcritical and supercritical Galton-Watson trees conditioned to survive in the transient, sub-ballistic regime. By considering offspring laws with infinite variance, we extend previously known results for the walk on the supercritical tree and observe new trapping phenomena for the walk on the subcritical tree which, in this case, always yield sub-ballisticity. This is contrary to the walk on the supercritical tree which always has some ballistic phase.\end{abstract}

\let\thefootnote\relax\footnote{\textit{MSC2010 subject classifications:} Primary 60K37, 60F05; secondary 60E07, 60J80. }

\section{Introduction}
In this paper, we investigate biased random walks on subcritical and supercritical Galton-Watson trees. These are a natural setting for studying trapping phenomena as dead-ends, caused by leaves in the trees, slow the walk. For supercritical GW-trees with leaves, it has been shown in \cite{lypepe} that, for a suitably large bias away from the root, the dead-ends in the environment create a sub-ballistic regime. In this case, it has further been observed in \cite{arfrgaha}, that the walker follows a polynomial escape regime but cannot be rescaled properly due to a certain lattice effect. Here we show that, when the offspring law has finite variance, the walk on the subcritical GW-tree conditioned to survive experiences similar trapping behaviour to the walk on the supercritical GW-tree shown in \cite{arfrgaha}. However, the main focus of the article concerns offspring laws belonging to the domain of attraction of some stable law with index $\alpha\in(1,2)$. In this setting, although the distribution of time spent in individual traps has polynomial tail decay in both cases, the exponent varies with $\alpha$ in the subcritical case and not in the supercritical case. This results in a polynomial escape of the walk which is always sub-ballistic in the subcritical case unlike the supercritical case which always has some ballistic phase. 

We now describe the model of a biased random walk on a subcritical GW-tree conditioned to survive which will be the main focus of the article. Let $f(s)=\sum_{k=0}^\infty p_ks^k$ denote the probability generating function of the offspring law of a GW-process with mean $\mu>0$ and variance $\sigma^2>0$ (possibly infinite) and let $Z_n$ denote the $n\th$ generation size of a process with this law started from a single individual, i.e.\ $Z_0=1$. Such a process gives rise to a random tree $\Tc$, where individuals in the process are represented by vertices and undirected edges connect individuals with their offspring. 

A $\beta$-biased random walk on a fixed, rooted tree $\Tc$ is a random walk $(X_n)_{n \geq 0}$ on $\Tc$ which is $\beta$-times more likely to make a transition to a given child of the current vertex than the parent (which are the only options). More specifically, let $\rho$ denote the root of the $\Tc$, $\overleftarrow{x}$ the parent of $x \in \Tc$ and $c(x)$ the set of children of $x$, then the random walk is the Markov chain started from $X_0=z$ defined by the transition probabilities 
\[\Pt^\Tc_z(X_{n+1}=y|X_n=x)=\begin{cases} \frac{1}{1+\beta d_x} & \text{if } y=\overleftarrow{x}, \\  \frac{\beta}{1+\beta d_x}, & \text{if } y \in c(x), \; x \neq \rho, \\ \frac{1}{
d_\rho}, & \text{if } y \in c(x), \; x =\rho, \\ 0, & \text{otherwise.} \\ \end{cases} \]
We use $\Pb_\rho(\cdot)=\int \Pt^\Tc_\rho(\cdot)\Pr(\text{d}\Tc)$ for the annealed law obtained by averaging the quenched law $\Pt^\Tc_\rho$ over a law $\Pr$ on random trees with a fixed root $\rho$. In general we will drop the superscript $\Tc$ and subscript $\rho$ when it is clear to which tree we are referring and we start the walk at the root. 

We will mainly be interested in trees $\Tc$ which survive, that is $\Hc(\Tc):=\sup\{n\geq 0: Z_n>0\}=\infty$. It is classical (e.g.\ \cite{atne}) that when $\mu>1$ there is some strictly positive probability $1-q$ that $\Hc(\Tc)=\infty$ whereas when $\mu\leq 1$ we have that $\Hc(\Tc)$ is almost surely finite. However, it has been shown in \cite{ke} that there is some well defined probability measure $\Pr$ over $f$-GW trees conditioned to survive for infinitely many generations which arises as a limit of probability measures over $f$-GW trees conditioned to survive at least $n$ generations. 

For $x \in \Tc$ let $|x|=d(\rho,x)$ denote the graph distance between $x$ and the root of the tree and write $\Tc_x$ to be the descendent tree of $x$. The main object of interest is $|X_n|$, that is, how the distance from the root changes over time. Due to the typical size of finite branches in the tree being small and the walk not backtracking too far we shall see that $|X_n|$ has a strong inverse relationship with the first hitting times $\Delta_n:=\inf\{m\geq 0:X_m \in \Yc, \; |X_m|=n\}$ of levels along the backbone $\Yc:=\{x \in \Tc:\Hc(\Tc_x)=\infty\}$ so for much of the paper we will consider this instead. It will be convenient to consider the walk as a trapping model. To this end we define the underlying walk $(Y_k)_{k\geq 0}$ defined by $Y_k=X_{\eta_k}$ where $\eta_0=0$ and $\eta_k=\inf\{m>\eta_{k-1}: \; X_m,X_{m-1} \in \Yc\}$ for $k\geq1$.

When $X_n$ is a walk on an $f$-GW tree conditioned to survive for $f$ supercritical ($\mu>1$), it has been shown in \cite{lypepe} that $\nu(\beta):=\lim_n|X_n|/n$ exists $\Pb$-a.s.\ and is positive if and only if $\mu^{-1}<\beta<f'(q)^{-1}$ in which case we call the walk ballistic. If $\beta\leq \mu^{-1}$ then the walk is recurrent because the average drift of $Y$ acts towards the root. When $\beta\geq f'(q)^{-1}$ the walker expects to spend an infinite amount of time in the finite trees which hang off $\Yc$ (see Figure \ref{suptreediag} in Section \ref{suptree}) thus causing a slowing effect which results in the walk being sub-ballistic. In this case, the correct scaling for some non-trivial limit is $n^{\gamma}$ where $\gamma$ will be defined later in (\ref{gamma}). In particular it has been shown in \cite{arfrgaha} that, when $\sigma^2<\infty$, the laws of $|X_n|n^{-\gamma}$ are tight and, although $|X_n|n^{-\gamma}$ doesn't converge in distribution, we have that $\Delta_nn^{-1/\gamma}$ converges in distribution under $\Pb$ along certain subsequences to some infinitely divisible law. In Section \ref{suptree} we prove several lemmas which extend this result by relaxing the condition that the offspring law has finite variance and instead requiring only that it belongs to the domain of attraction of some stable law of index $\alpha >1$. 

Recall that the offspring law of the process is given by $\Pr(\xi=k)=p_k$, then we define the size-biased distribution by the probabilities $\Pr(\xi^*=k)=kp_k\mu^{-1}$. It can be seen (e.g.\ \cite{ja}) that the subcritical ($\mu<1$) GW-tree conditioned to survive coincides with the following construction: Starting with a single special vertex, at each generation let every normal vertex give birth onto normal vertices according to independent copies of the original offspring distribution and every special vertex give birth onto vertices according to independent copies of the size-biased distribution, one of which is chosen uniformly at random to be special. Unlike the supercritical tree which has infinitely many infinite paths, the backbone of the subcritical tree conditioned to survive consists of a unique infinite path from the initial vertex $\rho$. We call the vertices not on $\Yc$ which are children of vertices on $\Yc$ buds and the finite trees rooted at the buds traps (see Figure \ref{treediag} in Section \ref{numtrp}).

Briefly, the phenomena that can occur in the subcritical case are as follows. When $\Er[\xi\log^+(\xi)]<\infty$ and $\mu<1$ there exists a limiting speed $\nu(\beta)$ such that $|X_n|/n$ converges almost surely to $\nu(\beta)$ under $\Pb$; moreover, the walk is ballistic ($\nu(\beta)>0$) if and only if $1<\beta<\mu^{-1}$ and $\sigma^2<\infty$. This essentially follows from the argument used in \cite{lypepe} (to show the corresponding result on the supercritical tree) with the fact that, by (\ref{expind}) and (\ref{expexc}), the conditions given are precisely the assumptions needed so that the expected time spent in a branch is finite (see \cite{bo}). The sub-ballistic regime has four distinct phases. When $\beta\leq1$ the walk is recurrent and we are not concerned with this case here. When $1<\beta<\mu^{-1}$ and $\sigma^2=\infty$ the expected time spent in a trap is finite and the slowing of the walk is due to the large number of buds. When $\beta\mu>1$ and $\sigma^2<\infty$, the expected time spent in a subcritical GW-tree forming a trap is infinite because the strong bias forces the walk deep into traps and long sequences of movements against the bias are required to escape. In the final case for the subcritical tree ($\beta\mu>1$, $\sigma^2=\infty$) slowing effects are caused by both strong bias and the large number of buds. 

Figure \ref{phase} is the phase diagram for the almost sure limit of $\log(|X_n|)/\log(n)$ (which is the first order scaling of $|X_n|$ relative to $\beta$ and $\mu$) where the offspring law has stability index $\alpha$, which is $2$ when $\sigma^2<\infty$, and we define 
\begin{flalign}\label{gamma}
\gamma=\begin{cases}\frac{\log(f'(q)^{-1})}{\log(\beta)}, & \mu> 1,\\ \frac{\log(\mu^{-1})}{\log(\beta)}, & \mu< 1, \end{cases}
\end{flalign}
where we note that $f'(q)$ and $\mu$ are the mean number of offspring from vertices in traps of the supercritical and subcritical trees respectively. Strictly, $f'(q)$ isn't a function of $\mu$ therefore the line $\beta=f'(q)^{-1}$ is not well defined; Figure \ref{phase} shows the particular case when the offspring distribution belongs to the geometric family. It is always the case that $f'(q)<1$ therefore some such region always exists however the parametrisation depends on the family of distributions.    

When the offspring law has finite variance, the limiting behaviour of $|X_n|$ on the supercritical and subcritical trees is very similar. Both have a regime with linear scaling (which is, in fact, almost sure convergence of $|X_n|/n$) and a regime with polynomial scaling caused by the same phenomenon of deep traps (which results in $|X_n|n^{-\gamma}$ not converging). When the offspring law has infinite variance, the bud distribution of the subcritical tree has infinite mean which causes an extra slowing effect which isn't seen by the supercritical tree. This equates for the different exponents observed in the two models as shown in Figure \ref{phase}. The walk on the critical ($\mu=1$) tree experiences a similar trapping mechanism to the subcritical tree; however, the slowing is more extreme and belongs to a different universality class which had been shown in \cite{crfrku} to yield a logarithmic escape rate.

\begin{figure}[H]
\centering
 \includegraphics[scale=0.3]{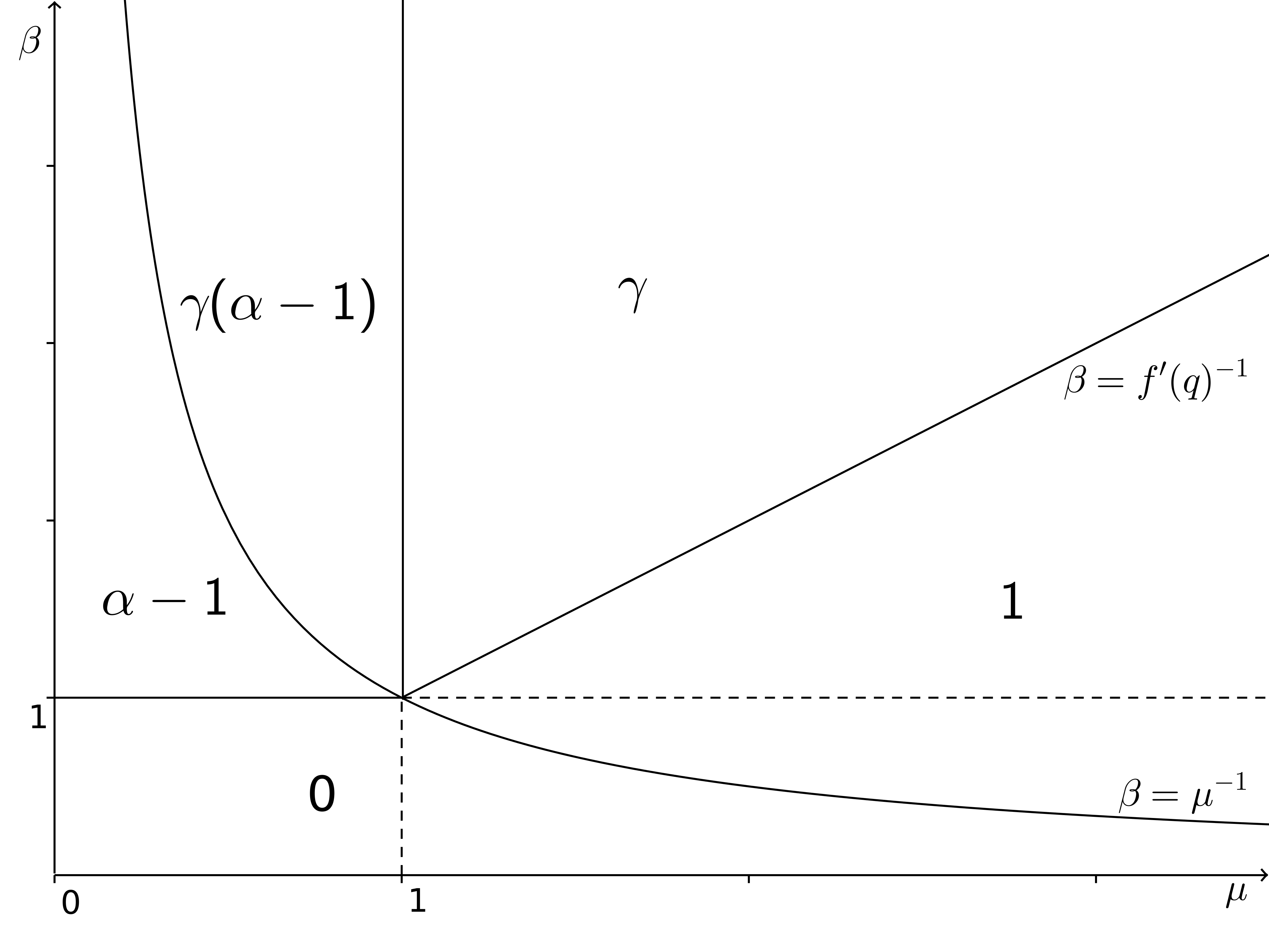} 
\caption{Escape regimes.}\label{phase}
\end{figure}

\section{Statement of main theorems and proof outline}\label{stmnt}
In this section we introduce the three sub-ballistic regimes in the subcritical case and the one further regime for the infinite variance supercritical case that we consider here. We then state the main theorems of the paper. 

The subcritical tree has bud distribution $\xi^*-1$ where $\Pr(\xi^*=k)=kp_k\mu^{-1}$ which yields the following important property relating the size biased and offspring distributions 
\begin{flalign}\label{expind}
 \Er[\varphi(\xi^*)]=\sum_{k=1}^\infty \varphi(k)\frac{kp_k}{\mu}=\Er[\varphi(\xi)\xi]\mu^{-1}.
\end{flalign}
In particular, choosing $\varphi$ to be the identity we have finite mean of the size-biased distribution if and only if the variance of the offspring distribution is finite. This causes a phase transition for the walk that isn't seen in the supercritical tree. The reason for this is that in the corresponding decomposition for the supercritical tree we have subcritical GW-trees as leaves but the number of buds is exponentially tilted and therefore maintains moment properties. 

If the offspring law belongs to the domain of attraction of some stable law of index $\alpha \in (1,2)$ then taking $\varphi(x)=x\ind_{\{x\leq t\}}$ shows that the size biased distribution belongs to the domain of attraction of some stable law with index $\alpha-1$ and allows us to attain properties of the scaling sequences (see for example \cite{fe} IX.8). 

The first case we consider is when $\beta\mu<1$ but $\sigma^2=\infty$; we refer to this as the infinite variance and finite excursion case: 
\begin{dfn}\label{finexc}(IVFE)
The offspring distribution has mean $\mu$ satisfying $1<\beta<\mu^{-1}$ and belongs to the domain of attraction of a stable law of index $\alpha \in (1,2)$. 
\end{dfn}
Under this assumption we let $L$ vary slowly at $\infty$ such that as $x\rightarrow \infty$ 
\begin{flalign}\label{trsm}
 \Eb[\xi^2\ind_{\{\xi\leq x\}}] & \sim x^{2-\alpha}L(x)
\end{flalign}
and choose $(a_n)_{n\geq 1}$ to be some scaling sequence for the size-biased law such that for any $x>0$, as $\nin$ we have $\Pr(\xi^*\geq xa_n) \sim x^{-(\alpha-1)}n^{-1}$. Moreover for some slowly varying function $\tilde{L}$ we have that $a_n=n^{\frac{1}{\alpha-1}}\tilde{L}(n)$. 

In this case we have that the heavy trapping is caused by the number of excursions in traps. Since $\beta$ is small we have that the expected time spent in a trap is finite but, because the size-biased law has infinite mean, the expected time spent in a branch is infinite. The main result for IVFE is Theorem \ref{finexcthm} which reflects that $\Delta_n$ scales similarly to the sum of independent copies of $\xi^*$.
\begin{thm}\label{finexcthm}
 For IVFE, the laws of the process \[\left(\frac{\Delta_{nt}}{a_n}\right)_{t\geq 0}\]
 converge weakly as $\nin$ under $\Pb$ with respect to the Skorohod $J_1$ topology on $D([0,\infty),\Rb)$ to the law of an $\alpha-1$ stable subordinator $R_t$ with Laplace transform \[\varphi_t(s)=\Eb[e^{-sR_t}] =e^{-C_{\alpha,\beta,\mu}ts^{\alpha-1}}\] where $C_{\alpha,\beta,\mu}$ is a constant which we shall determine during the proof (see (\ref{Cval})).
\end{thm}

We refer to the second $(\sigma^2<\infty, \;\beta\mu>1)$ and third $(\sigma^2=\infty, \;\beta\mu>1)$ cases as the finite variance and infinite excursion and infinite variance and infinite excursion cases respectively.
\begin{dfn}\label{finvar}(FVIE)
 The offspring distribution has mean $\mu$ satisfying $1<\mu^{-1}<\beta$ and variance $\sigma^2<\infty$.
\end{dfn}

\begin{dfn}\label{infall}(IVIE)
 The offspring distribution has mean $\mu$ satisfying $1<\mu^{-1}<\beta$ and belongs to the domain of attraction of a stable law of index $\alpha \in (1,2)$. 
\end{dfn}
As for IVFE, in IVIE we let $L$ vary slowly at $\infty$ such that (\ref{trsm}) holds and $(a_n)_{n\geq 1}$ be some scaling sequence for the size-biased law such that for any $x>0$, as $\nin$ we have $\Pr(\xi^*\geq xa_n) \sim x^{-(\alpha-1)}n^{-1}$. It then follows that $a_n=n^{\frac{1}{\alpha-1}}\tilde{L}(n)$ for some slowly varying function $\tilde{L}$. In FVIE, $a_n=n$ will suffice.

In FVIE and IVIE the heavy trappings are caused by excursions in deep traps because the walk is required to make long sequences of movements against the bias in order to escape. The times spent in large traps tend to cluster around $(\mu\beta)^\Hc$ where $\Hc$ is the height of the branch. Since $\Hc$ is approximately geometric we have that, for $\beta$ large, $(\mu\beta)^\Hc$ won't belong to the domain of attraction of any stable law. For this reason, as in \cite{arfrgaha}, we only see convergence along specific increasing subsequences $n_k(t)=\lfloor t \mu^{-k}\rfloor$ for $t>0$ in FVIE and $n_k(t)$ such that $a_{n_k(t)} \sim t\mu^{-k}$ for IVIE. Such a sequence exists for any $t>0$ since by choosing $n_k(t):=\sup\{m\geq 0:a_m< t\mu^{-k}\}$ we have that $a_{n_k}< t\mu^{-k} \leq a_{n_k+1}$ and therefore 
\begin{flalign*}
 1 \geq \frac{a_{n_k}}{t\mu^{-k}}  \geq \frac{a_{n_k}}{a_{n_k+1}} \rightarrow 1.
\end{flalign*}
 
Recalling (\ref{gamma}), the main results for FVIE and IVIE are Theorems \ref{finvarthm} and \ref{infallthm}, which reflect heavy trappings due to deep excursions (and also the large number of traps in IVIE). 
\begin{thm}\label{finvarthm}
 In FVIE, for any $t>0$ we have that as $k\rightarrow \infty$  \[\frac{\Delta_{n_k(t)}}{n_k(t)^\frac{1}{\gamma}} \rightarrow R_t\] in distribution under $\Pb$, where $R_t$ is a random variable with an infinitely divisible law.  
\end{thm}

\begin{thm}\label{infallthm}
 In IVIE, for any $t>0$ we have that as $k\rightarrow \infty$ \[\frac{\Delta_{n_k(t)}}{a_{n_k(t)}^\frac{1}{\gamma}} \rightarrow R_t\] in distribution under $\Pb$, where $R_t$ is a random variable with an infinitely divisible law.  
\end{thm}

We write $r_n$ to be $a_n$ in IVFE, $n^{1/\gamma}$ in FVIE and $a_n^{1/\gamma}$ in IVIE; then, letting $b_n:=\max\{m\geq 0:r_m\leq n\}$ we will also prove Theorem \ref{tghtthm}. This shows that, although the laws of $X_n/b_n$ don't converge in general (for FVIE and IVIE), the suitably scaled sequence is tight and we can determine the leading order polynomial exponent explicitly.
\begin{thm}\label{tghtthm}
 In IVFE, FVIE or IVIE we have that
 \begin{enumerate}
  \item\label{deltght} The laws of $(\Delta_n/r_n)_{n\geq 0}$ under $\Pb$ are tight on $(0,\infty)$;
  \item\label{xtght} The laws of $(|X_n|/b_n)_{n\geq 0}$ under $\Pb$ are tight on $(0,\infty)$.
 \end{enumerate}
Moreover, in IVFE, FVIE and IVIE respectively, we have that $\Pb$-a.s.\
\begin{flalign*}
 \limn \frac{\log|X_n|}{\log(n)} & =\alpha-1; \\
 \limn \frac{\log|X_n|}{\log(n)} & =\gamma; \\
 \limn \frac{\log|X_n|}{\log(n)} & =\gamma(\alpha-1).
\end{flalign*}
\end{thm}

The final case we consider is an extension of a result of \cite{arfrgaha} for the walk on the supercritical tree which we put aside until Section \ref{suptree} since it only requires several technical lemmas and the argument is of a different structure to the subcritical tree. For the same reason as in FVIE, we only see convergence along specific subsequences $n_k(t)=\lfloor t f'(q)^{-k}\rfloor$ for $t>0$. 
\begin{thm}\label{suptreethm}(Infinite variance supercritical case)
Suppose the offspring law belongs to the domain of attraction of some stable law of index $\alpha \in (1,2)$, has mean $\mu>1$ and the derivative of the generating function at the extinction probability satisfies $\beta>f'(q)^{-1}$. Then, \[\frac{\Delta_{n_k(t)}}{n_k(t)^\frac{1}{\gamma}} \rightarrow R_t\] in distribution as $k\rightarrow \infty$ under $\Pb$, where $R_t$ is a random variable with an infinitely divisible law whose parameters are given in \cite{arfrgaha}. Moreover, the laws of $(\Delta_nn^{-\frac{1}{\gamma}})_{n\geq 0}$ and $(|X_n|n^{-\gamma})_{n\geq 0}$ under $\Pb$ are tight on $(0,\infty)$ and $\Pb$-a.s.\
\begin{flalign*}
 \limn \frac{\log|X_n|}{\log(n)} & =\gamma.
\end{flalign*}
\end{thm}

The proofs of Theorems \ref{finexcthm}, \ref{finvarthm} and \ref{infallthm} follow a similar structure to the corresponding proof of \cite{arfrgaha} which, for the walk on the supercritical tree, only considers the case in which the variance of the offspring distribution is finite. However, for the latter reason, the proofs of Theorems \ref{finexcthm} and \ref{infallthm} become more technical in some places, specifically with regards to the number of traps in a large branch. The proof can be broken down in to a sequence of stages which investigate different aspects of the walk and the tree. This is ideal for extending the result onto the supercritical tree because many of these behavioural properties will be very similar for the walk on the subcritical tree due to the similarity of the traps.

In all cases it will be important to decompose large branches. In Section \ref{numtrp} we show a decomposition of the number of deep traps in any deep branch. This is only important for FVIE and IVIE since the depth of the branch plays a key role in decomposing the time spent in large branches. In Section \ref{trpapr} we determine conditions for labelling a branch as large in each of the regimes so that large branches are sufficiently far apart so that, with high probability, the underlying walk won't backtrack from one large branch to the previous one. In Section \ref{spntlrg} we justify the choice of label by showing that time spent outside these large branches is negligible. From this we then have that $\Delta_n$ can be approximated by a sum of i.i.d.\ random variables whose distribution depends on $n$. In Section \ref{siid} we only consider IVFE and show that, under a suitable scaling, these variables converge in distribution which allows us to show the convergence of their sum. Similarly, in Section \ref{wlkindp} we show that the random variables, suitably scaled, converge in distribution for FVIE and IVIE. We then show convergence of their sum in Section \ref{convsub}. In Section \ref{tght} we prove Theorem \ref{tghtthm} which is standard following Theorems \ref{finexcthm}, \ref{finvarthm} and \ref{infallthm}. Finally, in Section \ref{suptree}, we prove three short lemmas which extend the main result of \cite{arfrgaha} to prove Theorem \ref{suptreethm}.

\section{Number of traps}\label{numtrp}
In the construction of the subcritical GW-tree conditioned to survive described in the introduction, the special vertices form the infinite backbone $\Yc$ consisting of all vertices with an infinite line of descent. For $i=0,1,...$ we denote the vertex in $\Yc$ in generation $i$ as $\rho_i$. Each vertex on the backbone is connected to buds $\rho_{i,j}$ for $j=1,...,\xi^*_{\rho_i}-1$ (which are the normal vertices that are offspring of special vertices in the construction). Each of these is then the root of an $f$-GW tree $\Tc^*_{\rho_{i,j}}$. We call each $\Tc^*_{\rho_{i,j}}$ a trap and the collection from a single backbone vertex (combined with the backbone vertex) $\Tc^{*-}_{\rho_i}$ a branch. Figure \ref{treediag} shows an example of the first five generations of a tree $\Tc^*$. The solid line represents the backbone and the two dotted ellipses identify a sample branch and trap. The dashed ellipse indicates the children of $\rho_1$ which, since $\rho_1$ is on the backbone, have quantity distributed according to the size-biased law.

\begin{figure}[H]
\centering
 \includegraphics[scale=0.8]{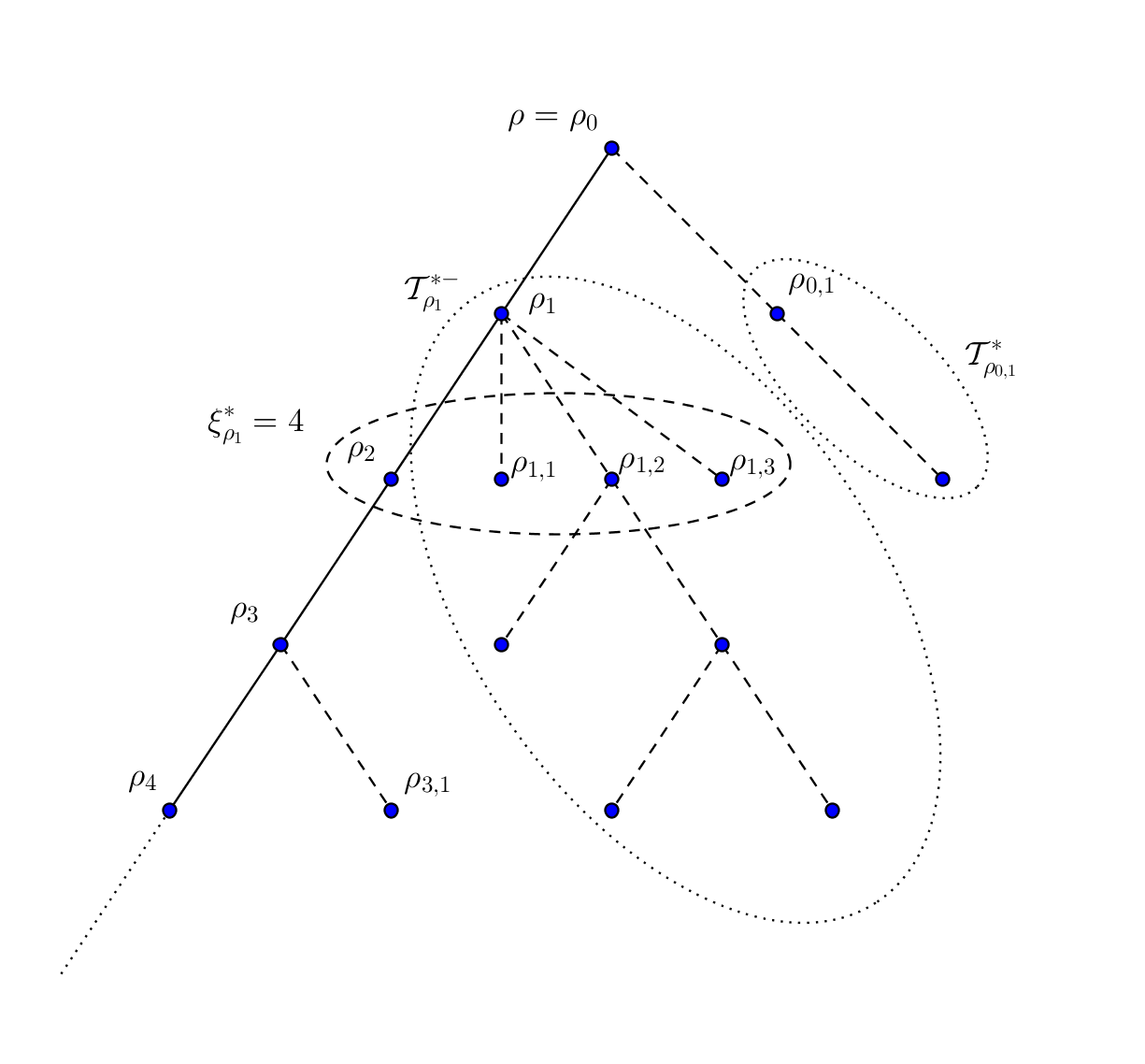} 
\caption{A sample subcritical tree.}\label{treediag}
\end{figure}

The structure of the large traps will have an important role in determining the convergence of the scaled process. In this section we determine the distribution over the number of deep traps rooted at backbone vertices with at least one deep trap. We will show that there is only a single deep trap at any backbone vertex when the offspring law has finite variance whereas, when the offspring law belongs to the domain of attraction of a stable law with index $\alpha <2$ we have that the number of deep traps converges in distribution to a certain heavy tailed law.

A fundamental result for branching processes (see, for example \cite{lypepeot}), is that for $\mu<1$ and $Z_n$ an $f$-GW process, the sequence $\Pr(Z_n>0)/\mu^n$ is decreasing; moreover, $\Er[\xi\log(\xi)]<\infty$ if and only if the limit of $\Pr(Z_n>0)\mu^{-n}$ as $\nin$ exists and is strictly positive. This assumption holds under any of the hypotheses thus for this paper we will always make this assumption and let $c_\mu$ be the constant such that 
\begin{flalign}\label{cmu}
\Pr(Z_n>0)\sim c_\mu\mu^n.
\end{flalign}

For an arbitrary rooted tree $T$ with root $\rho$ write $\Hc(T) = \sup_{x \in T}d(\rho,x)$ to be it's height. Let $(\Tc_i)_{i\geq 1}$ be independent $f$-GW trees then write $N(m)=\sum_{j=1}^{\xi^*-1}\ind_{\{\Hc(\Tc_i)\geq m\}}$ to have the distribution of the number of traps of size at least $m$ rooted at a single backbone vertex. Denote 
\begin{flalign}\label{sm}
s_m=\Pr(\Hc(\Tc_1)< m)=1-c_\mu\mu^m(1+o(1))
\end{flalign} 
the probability that a given trap is of height at most $m-1$ (although in general we shall write $s$ for convenience). We are interested in the limit as $m \rightarrow \infty$ of
\begin{equation}\label{frc}
\Pr(N(m)=l|N(m)\geq 1)=\frac{\Pr(N(m)=l)}{\Pr(N(m)\geq 1)}
\end{equation}
for $l \geq 1$. Recall that $f$ is the p.g.f.\ of the offspring distribution, then we have that
\begin{flalign}
 \Pr(N(m)=l) & = \sum_{k=1}^\infty \Pr(\xi^*=k)\Pr(N(m)=l|\xi^*=k) \notag\\
 & = \sum_{k=l+1}^\infty \frac{kp_k}{\mu}s^{k-1-l}(1-s)^l\binom{k-1}{l} \notag\\
 & = \frac{(1-s)^l}{l!\mu}f^{(l+1)}(s). \label{lcrttrp}
\end{flalign}
 In particular, we have that $\Pr(N(m)\geq 1) = 1-f'(s)/\mu$. Lemma \ref{sintrp} shows that, when $\sigma^2<\infty$, with high probability there will only be a single deep trap in any deep branch.
\begin{lem}\label{sintrp}
 When $\sigma^2<\infty$ \[ \lim_{m\rightarrow \infty} \Pr(N(m)=1|N(m)\geq 1)=1.\]
 \begin{proof}
  Using (\ref{frc}) and (\ref{lcrttrp}) we have that
  \begin{flalign}\label{sumfrc}
   \Pr(N(m)=1|N(m)\geq 1) \; = \; \frac{(1-s)f''(s)/\mu}{1-f'(s)/\mu} \; = \; \frac{\sum_{k=2}^\infty k(k-1)p_ks^{k-2}}{\sum_{k=2}^\infty kp_k\frac{1-s^{k-1}}{1-s}}.
  \end{flalign}
By monotonicity in $s$ we have that \[\lim_{s\rightarrow 1^-}\sum_{k=2}^\infty k(k-1)p_ks^{k-2} = \sum_{k=2}^\infty k(k-1)p_k\] which is finite since $\sigma^2<\infty$. Each summand in the denominator is increasing in $s$ for $s\in (0,1)$ and by L'Hopital's rule $1-s^{k-1} \sim (k-1)(1-s)$ as $s\rightarrow 1^-$ therefore, by monotone convergence, the denominator in the final term of (\ref{sumfrc})  converges to the same limit.
 \end{proof}
\end{lem}

In order to determine the correct threshold for labelling a branch as large we will need to know the asymptotic form of $\Pr(N(m)\geq 1)$. Corollary \ref{hgttl} gives this for the finite variance case.
\begin{cly}\label{hgttl}
 Suppose $\sigma^2<\infty$ then \[\Pr(N(m)\geq 1) \; \sim \; c_\mu \Er[\xi^*-1]\mu^m \; = \; c_\mu\left(\frac{\sigma^2+\mu^2}{\mu}-1\right)\mu^m.\]
 \begin{proof}
 Let $f_*$ denote the p.g.f.\ of $\xi^*$ then $\Pr(N(m)\geq 1)  = 1-s^{-1}f_*(s)$. Since $\sigma^2<\infty$ we have that $f_*'(s)$ exists and is continuous for $s\leq 1$ thus as $s \rightarrow 1^-$ we have that $f_*(1)-f_*(s) \sim (1-s)f_*'(1) = (1-s)\Er[\xi^*]$. It therefore follows that
 \begin{flalign*}
  1-s^{-1}f_*(s) \; = \; f_*(1)-f_*(s) -\frac{f_*(s)(1-s)}{s} \; \sim \; (1-s)(\Er[\xi^*]-1).
 \end{flalign*}
The result then follows by the definitions of $c_\mu$ (\ref{cmu}) and $s$ (\ref{sm}). 
\end{proof}
\end{cly}

\subsection{Infinite variance}
We now consider the case when $\sigma^2=\infty$ but $\xi$ belongs to the domain of attraction of a stable law of index $\alpha \in(1,2)$. The following lemma concerning the form of the probability generating function of the offspring distribution will be fundamental in determining the distribution over the number of large traps rooted at a given backbone vertex. The case $\mu=1$ appears in \cite{bjst}; the proof of Lemma \ref{genform} is a simple extension of this hence the proof is omitted.
\begin{lem}\label{genform}
Suppose the offspring distribution belongs to the domain of attraction of a stable law with index $\alpha \in (1,2)$ and mean $\Er[\xi]=\mu$.
\begin{enumerate}
 \item If $\mu\leq 1$ then as $s\rightarrow 1^-$ \[\Er[s^\xi]-s^\mu\sim \frac{\Gamma(3-\alpha)}{\alpha(\alpha-1)}(1-s)^\alpha L((1-s)^{-1})\]
where $\Gamma(t)=\int_0^\infty x^{t-1}e^{-x}\d x$ is the usual gamma function.
\item If $\mu>1$ then \[1-\Er[s^\xi] = \mu(1-s)+  \frac{\Gamma(3-\alpha)}{\alpha(\alpha-1)}(1-s)^\alpha \overline{L}((1-s)^{-1})\]
where $\overline{L}$ varies slowly at $\infty$.
\end{enumerate}
\end{lem}

When $\mu<1$ it follows that there exists a function $L_1$ (which varies slowly as $s\rightarrow 1^-$) such that $\Er[s^\xi]-s^\mu= (1-s)^\alpha L_1((1-s)^{-1})$ and \[\lim_{s\rightarrow 1^-}\frac{L_1((1-s)^{-1})}{L((1-s)^{-1})}=\frac{\Gamma(3-\alpha)}{\alpha(\alpha-1)}. \]

Write $g(x)=x^\alpha L_1(x^{-1})$ so that $f(s)=s^\mu+g(1-s)$ and thus \[f^{(l)}(s)=s^{\mu-l}(\mu)_l+(-1)^lg^{(l)}(1-s)\]
when this exists where $(\mu)_l=\prod_{j=0}^{l-1}(\mu-j)$ is the Pochhammer symbol. Write $L_2(x)=L_1(x^{-1})$ which is slowly varying at $0$. Using Theorem 2 of \cite{la}, we see that $xg'(x)\sim \alpha g(x)$ as $x \rightarrow 0$. Moreover, using an inductive argument in the proof of this result, it is straightforward to show that for all $l \in \Nb$ we have that $xg^{(l+1)}(x) \sim (\alpha-l)g^{(l)}(x)$ as $x\rightarrow 0$. Therefore, for any integer $l\geq0$  
\begin{flalign}\label{coeff} 
\lim_{x \rightarrow 0^+}\frac{x^lg^{(l)}(x)}{g(x)} \; = \; (\alpha)_l . 
\end{flalign} 

Define $N_m:=N(m)|N(m)\geq 1$ to be the number of traps of height at least $m$ in a branch of greater than $m$. Proposition \ref{numcritt} is the main result of this section and determines the limiting distribution of $N_m$.
\begin{prp}\label{numcritt}
 In IVIE, for $l\geq 1$ as $m \rightarrow \infty$ \[\Pr(N_m=l)\rightarrow \frac{1}{l!}\prod_{j=1}^l|\alpha-j|.\]
 \begin{proof}
Recall that by (\ref{frc}) and (\ref{lcrttrp}) we want to determine the asymptotics of $1-f'(s)/\mu$ and $(1-s)^lf^{(l+1)}(s)/(l!\mu)$ as $s \rightarrow 1^-$. We have that $1-f'(s)/\mu=1-s^{\mu-1}+g'(1-s)/\mu$ and $g'(1-s) \sim \alpha(1-s)^{\alpha-1}L_2(1-s)$ as $s\rightarrow 1$. Since $\alpha<2$, we have that $\lim_{s \rightarrow 1^-}(1-s^{\mu-1})(1-s)^{1-\alpha} =0$ hence 
\begin{flalign}\label{prbhn}
 1-\frac{f'(s)}{\mu} & \sim \frac{\alpha}{\mu}(1-s)^{\alpha-1}L_2(1-s).
\end{flalign}

For derivatives $l \geq 1$ we have that 
\begin{flalign*}
\frac{(1-s)^lf^{(l+1)}(s)}{l!\mu}=\frac{(1-s)^l}{l!\mu}\left(s^{\mu-(l+1)}(\mu)_l+(-1)^{l+1}g^{(l+1)}(1-s) \right).
\end{flalign*}
By (\ref{coeff}) we have that  $(1-s)^lg^{(l+1)}(1-s)\sim (\alpha)_{l+1}(1-s)^{\alpha-1}L_2(1-s)$. For $l\geq 1$ we have that $l+1-\alpha>0$ hence 
\begin{flalign}\label{hgrprb}
 \frac{(1-s)^lf^{(l+1)}(s)}{l!\mu} & \sim \frac{|(\alpha)_{l+1}|}{l!\mu}(1-s)^{\alpha-1}L_2(1-s).
\end{flalign}

Combining (\ref{frc}) with (\ref{prbhn}) and (\ref{hgrprb}) gives the desired result.
\end{proof}
\end{prp}

Proposition \ref{numcritt} will be useful for determining the number of large traps in a large branch but equally important is the asymptotic relation (\ref{prbhn}) which gives the tail behaviour of the height of a branch. Recall that $(\Tc^{*-}_{\rho_i})_{i\geq 0}$ are the finite branches rooted at $(\rho_i)_{i\geq 0}$ which are i.i.d.\ under $\Pr$. By the assumption on $\xi$ that (\ref{trsm}) holds we have that 
\begin{flalign}\label{xistar}
\Pr(\xi^*\geq t)\sim \frac{2-\alpha}{\mu(\alpha-1)}t^{-(\alpha-1)}L(t)
\end{flalign}
as $t\rightarrow \infty$. Therefore, using (\ref{sm}), (\ref{lcrttrp}), (\ref{prbhn}) and the relationship between $L$ and $L_1$ we have that
 \begin{flalign}
\Pr(\Hc(\Tc^{*-}_{\rho_0})>m) \; \sim \; \frac{\Gamma(3-\alpha)c_\mu^{\alpha-1}}{\mu(\alpha-1)} \mu^{m(\alpha-1)}L(\mu^{-m}) \; \sim \; \Gamma(2-\alpha)c_\mu^{\alpha-1}\Pr(\xi^* \geq \mu^{-m}). \label{numcrttrp}
\end{flalign}

\section{Large branches are far apart}\label{trpapr}   
In this section we introduce the conditions for a branch to be large. This will differ in each of the cases however, since many of the proofs will generalise to all three cases, we will use the same notation for some aspects.

In IVFE we will have that the slowing is caused by the large number of traps. In particular, we will be able to show that the time spent outside branches with a large number of buds is negligible.
\begin{dfn}\label{ivfedefn}(IVFE large branch)
 For $\varepsilon\in (-1,1)$ write \[l_n^\varepsilon = a_{\lfloor n^{1-\varepsilon}\rfloor}\] then we have that $\Pr(\xi^*\geq l_n^\varepsilon )   \sim  n^{-(1-\varepsilon)}$. We will call a branch large if the number of buds is at least $l_n^\varepsilon$ and write $\Dc^{(n)}:=\{x \in \Yc:d_x>l_n^\varepsilon\}$ to be the collection of backbone vertices which are the roots of large branches. 
\end{dfn}

In FVIE we will have that the slowing is caused by excursions into deep traps.
\begin{dfn}\label{fviedefn}(FVIE large branch)
For $\varepsilon \in (-1,1)$ write \[h_n^\varepsilon  :=\left\lfloor\frac{(1-\varepsilon)\log(n)}{\log(\mu^{-1})}\right\rfloor \]
and $C_\Dc=c_\mu\Er[\xi^*-1]$ then by Corollary \ref{hgttl} we have that   
\begin{flalign}\label{crtfindep}
 \Pr(\Hc(\Tc^{*-}_{\rho})> h_n^\varepsilon) \; \sim \; C_\Dc \mu^{h_n^\varepsilon} \; \approx \; C_\Dc n^{-(1-\varepsilon)}. 
\end{flalign}
We will call a branch large if there exists a trap within it of height at least $h_n^\varepsilon$ and write $\Dc^{(n)}:=\{x \in \Yc:\Hc(\Tc_x^{*-})> h_n^\varepsilon\} $
to be the collection of backbone vertices which are the roots of large branches. By a large trap we mean any trap of height at least $h_n^\varepsilon$. 
\end{dfn}

In IVIE we will have that the slowing is caused by a combination of the slowing effects of the other two cases. The height and number of buds in branches have a strong link which we show more precisely later; this allows us to label branches as large based on height which will be necessary when decomposing the time spent in large branches. 
\begin{dfn}\label{iviedefn}(IVIE large branch)
For $\varepsilon \in (-1,1)$ write \[h_n^\varepsilon  :=\left\lfloor\frac{\log(a_{n^{1-\varepsilon}})}{\log(\mu^{-1})}\right\rfloor \] 
then by (\ref{numcrttrp}), for $C_\Dc=\Gamma(2-\alpha)c_\mu^{\alpha-1}$, we have that 
\begin{flalign}\label{crtstbdep}
 \Pr(\Hc(\Tc^{*-}_\rho)> h_n^\varepsilon) \; \sim \; C_\Dc \Pr(\xi^*\geq \mu^{-h_n^\varepsilon}) \; \approx \; C_\Dc n^{-(1-\varepsilon)}. 
 \end{flalign}
We will call a branch large if there exists a trap of height at least $h_n^\varepsilon$ and write $\Dc^{(n)}:=\{x \in \Yc:\Hc(\Tc_x^{*-})> h_n^\varepsilon\}$ to be the collection of backbone vertices which are the roots of large branches. By a large trap we mean any trap of height at least $h_n^\varepsilon$. 
\end{dfn}

We want to show that, asymptotically, the large branches are sufficiently far apart to ignore any correlation and therefore approximate $\Delta_n$ by the sum of i.i.d.\ random variables representing the time spent in a large branch. Much of this is very similar to \cite{arfrgaha} so we only give brief details.

Write $\Dc_m^{(n)}:=\{x \in \Dc^{(n)}:|x|\leq m\}$ to be the large roots before level $m$ then let $q_n:=\Pr(\rho \in \Dc^{(n)})$ be the probability that a branch is large and write \[A_1(n,T):=\left\{\sup_{t \in [0,T]}\left||\Dc_{\lfloor tn\rfloor}^{(n)}| -\lfloor tnq_n\rfloor \right| < n^{2\varepsilon/3}\right\} \]
to be the event that the number of large branches by level $Tn$ doesn't differ too much from its expected value. Notice that in all three cases we have that $q_n$ is of the order $n^{-(1-\varepsilon)}$ thus we expect to see $nq_n \approx Cn^\varepsilon$ large branches by level $n$. 
\begin{lem}\label{numcrit}
 For any $T>0$
  \begin{flalign*}
\limn\Pr\left(A_1(n,T)^c\right)=0.
\end{flalign*}
\begin{proof}
 For each $n \in \Nb$ write \[M_m^n:=|\Dc_m^{(n)}| -mq_n \ed \sum_{k=1}^m \left(B_k-q_n\right) \]
 where $B_k$ are independent Bernoulli random variables with success probability $q_n$. Then $\Er[M_m^n]=0$ and $Var_{\Pr}(M_m^n)=mq_n(1-q_n)$ therefore by Kolmogorov's maximal inequality
 \[\Pr\left(\max_{1\leq m\leq\left\lfloor nT\right\rfloor}|M_m^n|>n^{2\varepsilon/3}-2\right) \;\leq\; \frac{cnTq_n}{n^{4\varepsilon/3}}\;\leq\; CTn^{-\varepsilon/3}.\]
 Since $\left|\lfloor nt\rfloor q_n-\lfloor ntq_n\rfloor\right| \leq 1$ we have that \[\sup_{t \in [0,T]}\left||\Dc_{\lfloor tn\rfloor}^{(n)}| -\lfloor tnq_n\rfloor \right| \leq \max_{1\leq m\leq\left\lfloor nT\right\rfloor}|M_m^n|+\sup_{0\leq t\leq T}|tnq_n-\lfloor nt\rfloor q_n| \leq \max_{1\leq m\leq\left\lfloor nT\right\rfloor}|M_m^n|+1 \]
which proves the statement.
\end{proof}
\end{lem}

We want to show that all of the large branches are sufficiently far apart such that the walk doesn't backtrack from one to another. For $t>0$ and $\kappa \in (0,1-2\varepsilon)$ write \[ \Dc(n,t):=\left\{\min_{x \neq y \in \Dc_{\lfloor nt\rfloor}^{(n)}}d(x,y)>n^{\kappa}\right\} \cap\{\rho \notin \Dc^{(n)}\}\] to be the event that all large branches up to level $\lfloor nt\rfloor$ are of distance at least $n^\kappa$ apart and the root of the tree is not the root of a large branch. A union bound shows that $\Pr(\Dc(n,t)^c)\rightarrow 0$ as $\nin$ uniformly over $t$ in compact sets.

We want to show that, with high probability, once the walk reaches a large branch it never backtracks to the previous one. For $t>0$ write \[A^{(0)}_2(n,t):=\bigcap_{i=0}^{\lfloor nt\rfloor}\bigcap_{n \geq \Delta^Y_{\rho_i}}\{|Y_n|>i-\overline{C}\log(n)\}\] to be the event that the walk never backtracks distance $\overline{C}\log(n)$ (where $\Delta_n^Y:=\min\{m\geq 0:Y_m=\rho_n\}$). For $x \in \Tc$ write $\tau_x^+=\inf\{n>0:X_n=x\}$ to be the first return time of $x$. Comparison with a simple random walk on $\Zb$ shows that for $k\geq1$ we have that the escape probability is $\Pt_{\rho_k}\left(\tau_{\rho_{k-1}}<\infty \right)=\beta^{-1}$ hence, using the Strong Markov property, \[\Pt_{\rho_m}\left(\tau_{\rho_0}<\infty\right)\leq C\beta^{-m} \]
 for some constant $C$. Using a union bound we see that 
\begin{flalign}\label{noback}
\Pb(A^{(0)}_2(n,t)^c) \leq Cnt\beta^{-\overline{C}\log(n)} \rightarrow 0
\end{flalign}
for $\overline{C}$ sufficiently large. Combining this with $\Dc(n,t)$ we have that with high probability the walk never backtracks from one large branch to a previous one. 

\section{Time is spent in large branches}\label{spntlrg} 
In this section we show that the time spent up to time $\Delta_n$ outside large branches is negligible. Combined with Section \ref{trpapr} this allows us to approximate $\Delta_n$ by the sum of i.i.d.\ random variables. We begin with some general results concerning the number of excursions into traps and the expected time spent in a trap of height at most $m$.

Recall that $\rho_{i,j}$ are the buds connected to the backbone vertex $\rho_i$. We write $W^{i,j}=|\{m \geq 0: X_{m-1}=\rho_i, \; X_m=\rho_{i,j}\}|$ to be the number of excursions into the $j\th$ trap of the $i\th$ branch where we set $W^{i,j}=0$ if $\rho_{i,j}$ doesn't exist in the tree. Lemma \ref{numexc} shows that, conditional on the number of buds, the number of excursions follows a geometric law.
\begin{lem}\label{numexc}
For any $i,k \in \Nb$ and $A \subset \{1,...,k\}$, when $\beta>1$
 \[ \sum_{j\in A}W^{i,j}\Big| |c(\rho_i)|=k+1 \sim Geo\left(\frac{\beta-1}{(|A|+1)\beta-1}\right) \]
 and in particular for any $j \leq k$ we have that $W^{i,j}\sim Geo(p)$ where $p=(\beta-1)/(2\beta-1)$.
 
 Moreover, conditional on $|c(\rho_i)|=k+1, \; A \subset \{1,...,k\}$, $(W^{i,j})_{j \in A}$ have a negative multinomial distribution with one failure until termination and probabilities 
 \[p_j=\begin{cases} \frac{\beta-1}{(|A|+1)\beta-1} & j=0 \\ \frac{\beta}{(|A|+1)\beta-1} & j\in A \\ \end{cases}\]
that from $\rho_i$ the next excursion will be into the $j\th$ trap (where $j=0$ denotes escaping). 
 \begin{proof}
 From $\rho_{i,j}$ the walk must return to $\rho_i$ before escaping therefore since $\Pt_{\rho_{i,j}}(\tau^+_{\rho_i}<\infty)=1$, any traps not in the set we consider can be ignored so it suffices to assume that $A = \{1,...,k\}$. By comparison with a biased random walk on $\Zb$ we have that $\Pt_{\rho_{i+1}}(\tau^+_{\rho_i}=\infty)=1-\beta^{-1}.$ If $|c(\rho_i)|=k+1$ then $\Pt_{\rho_i}(\tau^+_x=min_{y \in c(\rho_i)}\tau^+_y)=(k+1)^{-1}$ for any $x \in c(\rho_i)$. The probability of never entering a trap in the branch $\Tc^{*-}_{\rho_i}$ is, therefore, 
  \begin{flalign*}
   \Pt_{\rho_i}\left(\bigcap_{j=1}^k\{\tau^+_{\rho_{i,j}}=\infty\}\right) \; = \; \sum_{l=0}^\infty \left(\frac{1}{k+1}\beta^{-1}\right)^l\left(\frac{1-\beta^{-1}}{k+1}\right) \; = \; \frac{\beta-1}{(k+1)\beta-1}.
  \end{flalign*}  
  Each excursion ends with the walker at $\rho_i$ thus the walk takes a geometric number of excursions into traps with escape probability $(\beta-1)/((k+1)\beta-1)$. The second statement then follows from the fact that the walker has equal probability of going into any of the traps. 
 \end{proof}
\end{lem}

For a fixed tree $T$ with $n\th$ generation size $Z_n$ where $Z_1>0$ it is classical (e.g. \cite{lype}) that  
\begin{flalign}\label{return}
\Et^\Tc_\rho[\tau^+_\rho]=2\sum_{n\geq 1}\frac{Z_n\beta^{n-1}}{Z_1}.
\end{flalign}
Denoting $\Tc(i,j)=\Tc_{\rho_{i,j}}\cup\{\rho_i\}$ to be the tree formed by the descendent tree from $\rho_{i,j}$ along with the backbone vertex $\rho_i$ and $Z_n^{\Tc_{\rho_{i,j}}}$ the $n\th$ generation size of the tree $\Tc_{\rho_{i,j}}$, it follows that
\begin{flalign*}
 \Et^{\Tc(i,j)}_{\rho_{i,j}}[\tau^+_{\rho_i}] \; = \; \Et^{\Tc(i,j)}_{\rho_{i}}[\tau^+_{\rho_i}]-1 \; = \; 2\sum_{n \geq 0} Z_n^{\Tc_{\rho_{i,j}}}\beta^n-1. 
 \end{flalign*}
$\Pr(\Hc(\Tc)\leq m)\geq p_0$ therefore, for some constant $C$ and any $m\geq 1$,
\begin{flalign}
\Er\left[\Et^{\Tc(i,j)}_{\rho_{i,j}}[\tau^+_{\rho_i}]|\Hc(\Tc_{\rho_{i,j}})\leq m \right] \; \leq \; \frac{\Er\left[2\sum_{n =0}^{m-1} Z_n^{\Tc_{\rho_{i,j}}}\beta^n-1\right] }{\Pr(\Hc(\Tc)\leq m)} \; \leq \;  \begin{cases}  C(\mu\beta)^m & \beta\mu>1 \\ Cm & \beta\mu=1 \\ C & \beta\mu<1. \end{cases} \label{expexc}
\end{flalign}

Recall that $\Delta^Y_n$ is the first hitting time of $\rho_n$ for the underlying walk $Y$ and write $A_3(n):= \{\Delta^Y_n\leq C_1n\}$ to be the event that level $n$ is reached by time $C_1n$ by the walk on the backbone. Then standard large deviation estimates yield that $\limn\Pb(A_3(n)^c)=0$ for $C_1>(\beta+1)/(\beta-1)$.  

For the remainder of this section we mainly consider the case in which $\xi$ belongs to the domain of attraction of a stable law of index $\alpha \in (1,2)$. The case in which the offspring law has finite variance will proceed similarly however since the corresponding estimates are much simpler in this case we omit the proofs.

In IVIE and IVFE, for $t>0$, let the event that there are at most $\log(n)a_n$ buds by level $\lfloor nt\rfloor$ be
\begin{flalign}\label{A4}
A_4(n,t):=\left\{\sum_{k=1}^{\lfloor nt\rfloor}(\xi^*_k-1) \leq \log(n)a_n \right\}.
\end{flalign} 
Since the laws of $a_{\lfloor nt\rfloor}^{-1}\sum_{k=1}^{\lfloor nt\rfloor}(\xi^*_k-1)$ converge to some stable law $G^*$ and $\limn \overline{G}^*(Ct^{\alpha-1}\log(n))=0$ we clearly have that $\limn\Pr(A_4(n,t)^c)=0$. 

In FVIE write \[A_4(n,t):=\left\{\sum_{k=1}^{\lfloor nt\rfloor}(\xi^*_k-1) \leq \log(n)n \right\}\]
then Markov's inequality gives that $\limn\Pr(A_4(n,t)^c)=0$.
 
Write 
\begin{flalign}\label{A5}
A_5(n):=\left\{\max_{i,j}|\{k \leq \Delta_{\lfloor nt\rfloor}:X_{k-1}=\rho_i, \; X_k =\rho_{i,j}\}|\leq C_2\log(n)\right\}
\end{flalign}
be the event that any trap is entered at most $C_2\log(n)$ times. By Lemma \ref{numexc} the number of entrances into $\rho_{i,j}$ has the law of a geometric random variable of parameter $p=(\beta-1)/(2\beta-1)$ hence using a union bound we have that for  $C_2$ sufficiently large
 \begin{flalign*}
  \Pb(A_5(n,t)^c \cap A_4(n,t)) \; \leq \; \log(n)a_n\Pb\left(Geo(p)>C_2\log(n)\right) \; \leq \; L_1(n)n^{\frac{1}{\alpha-1}+C_2\log\left(1-p\right)} 
 \end{flalign*}
where $L_1$ is some slowly varying function hence the final term converges to $0$ for $C_2$ large therefore $\limn\Pb(A_5(n,t)^c)=0$. 

Propositions \ref{crttrpstab} and \ref{crtbnh} show that in IVFE and IVIE any time spent outside large traps is negligible. In FVIE and IVIE we only consider the large traps in large branches and write \[K(n)=\bigcup_{x \in \Dc^{(n)}}\{\Tc_y:y \in c(x)\setminus \{\rho_{|x|+1}\},\;\Hc(\Tc_y)\geq h_n^\varepsilon\} \] to be the vertices in large traps. In IVFE we require the entire large branch and write \[K(n)=\bigcup_{x \in \Dc^{(n)}}\{y \in \Tc^{*-}_x\}\] to be the vertices in large branches. In either case we write $\chi_{t,n}=|\{1\leq i\leq \Delta_{\lfloor nt\rfloor}: \;X_{i-1},X_i \in K(n)\}|$ to be the time spent up to $\Delta_{\lfloor nt \rfloor}$ in large traps.

\begin{prp}\label{crttrpstab}
 In IVIE, fix $\varepsilon>0$ then for any $t,\delta>0$ we have that as $\nin$ \[ \Pb\left(\left|\frac{\Delta_{\lfloor nt\rfloor}-\chi_{t,n}}{a_n^{1/\gamma}}\right|\geq \delta\right) \rightarrow 0.\]
\begin{proof}
On $\Dc(n,t)$ the root $\rho$ is not the root of a large branch and by the argument used to show that the walk never backtracks from one large branch to the previous one we have that with high probability the walk doesn't return to a large branch up to level $\lfloor nt\rfloor$ after time $\Delta_{\lfloor nt\rfloor}$. Therefore, with high probability, the time spent in large branches by time $\Delta_{\lfloor nt\rfloor}$ coincides with $\chi_{t,n}$.

On $A_4(n,t)$ there are at most $a_n\log(n)$ traps by level $\lfloor nt\rfloor$. We can order these traps so write $T^{(l,k)}$ to be the duration of the $k\th$ excursion into the $l\th$ trap and $\rho(l)$ to be the root of this trap (that is, the unique bud of $\Tc$ in the trap). Here we consider an excursion to start from the bud and end at the last hitting time of the bud before returning to the backbone. Using the estimates on $A_3,A_4$ and $A_5$ we have that \[ \Pb\left(\left|\frac{\Delta_{\lfloor nt\rfloor}-\chi_{t,n}}{a_n^{1/\gamma}}\right|\geq \delta\right) \leq o(1)+\Pb\left(C_1n+\sum_{l=0}^{a_n\log(n)}\sum_{k=0}^{C_2\log(n)}T^{(l,k)}\ind_{\{\Hc(\Tc_{\rho(l)})< h_n^\varepsilon\}} \geq \delta a_n^{\frac{1}{\gamma}}\right). \]
Since $a_n^{\frac{1}{\gamma}}\gg n$, for $n$ sufficiently large we have that, using Markov's inequality and (\ref{expexc}) with $m= h_n^\varepsilon$, the second term can be bounded above by
\begin{flalign*}
 2\delta^{-1}a_n^{-\frac{1}{\gamma}}\Eb\left[\sum_{l=0}^{a_n\log(n)}\sum_{k=0}^{C_2\log(n)}T^{(l,k)}\ind_{\{\Hc(\Tc_{\rho(l)})< h_n^\varepsilon\}}\right] \; \leq \; C_\delta\log(n)^2a_n^{1-\frac{1}{\gamma}}a_{n^{1-\varepsilon}}^{\frac{1}{\gamma}-1}. 
 \end{flalign*}
 Combining constants and slowly varying functions into a single function $L_\delta$ such that for any $\epsilon>0$ we have that $L_\delta(n)\leq n^\epsilon$ for $n$ sufficiently large thus  
 \begin{flalign*}
  \Pb\left(\left|\frac{\Delta_n-\chi_{1,n}}{a_n^{1/\gamma}}\right|\geq \delta\right)    \leq o(1) + L_\delta(n)n^{-\varepsilon\frac{\frac{1}{\gamma}-1}{\alpha-1}}
\end{flalign*}
 which converges to $0$ since $\alpha,\frac{1}{\gamma}>1$.
\end{proof}
\end{prp}

Using $A_3,A_5$ and the form of $A_4$ for FVIE, the technique used to prove Proposition \ref{crttrpstab} extends straightforwardly to prove Proposition \ref{crttrpfin} therefore we omit the proof.
\begin{prp}\label{crttrpfin}
 In FVIE, fix $\varepsilon>0$ then for any $t, \delta>0$ we have that as $\nin$ \[ \Pb\left(\left|\frac{\Delta_{\lfloor nt\rfloor}-\chi_{t,n}}{n^{1/\gamma}}\right|\geq \delta\right) \rightarrow 0.\]
 \end{prp}

Similarly, we can show a corresponding result for IVFE. 
\begin{prp}\label{crtbnh}
In IVFE, for any $t, \delta>0$, as $\nin$ \[\Pb\left(\left|\frac{\Delta_{\lfloor nt\rfloor}-\chi_{t,n}}{a_n}\right|\geq \delta\right) \rightarrow 0.\]
\begin{proof}
Let $c \in (0,2-\alpha)$ then, by Markov's inequality and the truncated first moment asymptotic:
\begin{flalign}\label{trn1}
 \Er\left[\xi^*\ind_{\{\xi^*\leq x\}}\right] \sim C x^{2-\alpha}L(x)
\end{flalign}
as $x\rightarrow \infty$ for some constant $C$ (see for example \cite{fe} IX.8), for $n$ large 
\begin{flalign*}
 \Pb\left(\sum_{k=0}^{\lfloor nt \rfloor}(\xi^*_{\rho_k}-1)\ind_{\{\xi^*_{\rho_k}-1\leq l_n^\varepsilon\}} \geq n^{\frac{1-c\varepsilon}{\alpha-1}}\right) \; \leq \; \frac{\Er\left[\sum_{k=0}^{\lfloor nt \rfloor}(\xi^*_{\rho_k}-1)\ind_{\{\xi^*_{\rho_k}-1\leq l_n^\varepsilon\}} \right]}{n^{\frac{1-c\varepsilon}{\alpha-1}}} \; \leq \;  n^{-\frac{\varepsilon(2-\alpha-c)}{\alpha-1}}L_1(n) 
\end{flalign*}
where $L_1(n)$ depends on $t$ and varies slowly at $\infty$. This converges to $0$ as $\nin$. We can order the traps in large branches and write $T^{(l,k)}$ to be the duration of the $k\th$ excursion in the $l\th$ large trap where we consider an excursion to start and end at the backbone. Using $A_3$ and $A_5$,
\begin{flalign*}
 \Pb\left(\left|\frac{\Delta_{\lfloor nt\rfloor}-\chi_{t,n}}{a_n}\right|\geq \delta\right) & \leq o(1) +  \Pb\left(\sum_{l=0}^{n^{\frac{1-c\varepsilon}{\alpha-1}}}\sum_{k=0}^{C_2\log(n)}T^{(l,k)} \geq \frac{\delta}{2}a_n\right). 
\end{flalign*}
Using Markov's inequality on the final term yields
\begin{flalign*}
   \Pb\left(\sum_{l=0}^{n^{\frac{1-c\varepsilon}{\alpha-1}}}\sum_{k=0}^{C_2\log(n)}T^{(l,k)} \geq \frac{\delta}{2}a_n\right)   \; \leq \; 2\delta^{-1}a_n^{-1}\Eb\left[\sum_{k=0}^{n^{\frac{1-c\varepsilon}{\alpha-1}}}\sum_{j=0}^{C_2\log(n)}T^{(l,k)}\right] \; \leq \;  n^{\frac{-c\varepsilon}{\alpha-1}} L_\delta(n)
\end{flalign*}
for some $L_\delta$ varying slowly at $\infty$. This converges to $0$ as $\nin$ hence the result holds.
\end{proof}
\end{prp}
 Since $\Delta_{\lfloor nt \rfloor}-\chi_{t,n}$ is non-negative and non-decreasing in $t$ we have that $\sup_{0\leq t\leq T}|\Delta_{\lfloor nt \rfloor}-\chi_{t,n}|=|\Delta_{\lfloor nT \rfloor}-\chi_{T,n}|$ therefore Corollary \ref{uspdiff} follows from Propositions \ref{crttrpstab}, \ref{crttrpfin} and \ref{crtbnh}. 
\begin{cly}\label{uspdiff}
 In each of IVFE, FVIE and IVIE, for any $T>0$ \[\sup_{0\leq t\leq T} \frac{|\Delta_{\lfloor nt \rfloor}-\chi_{t,n}|}{r_n}\]
 converges in $\Pb$-probability to $0$.
\end{cly}

Let $\Lambda$ be the set of strictly increasing continuous functions mapping $[0,T]$ onto itself and $I$ the identity map on $[0,T]$ then we consider the Skorohod $J_1$ metric \[d_{J_1}(f,g)=\inf_{\lambda \in \Lambda}\sup_{t \in [0,T]}\left(|f(t)-g(\lambda(t))|+|t-\lambda(t)|\right). \]

Write $\chi^i_n$ to be the total time spent in large traps of the $i\th$ large branch; that is \[\chi^i_n:=\left|\left\{m\geq 0:X_{m-1},X_m \in \left(\Tc^{*-}_{\rho_i^{(n)}}\cap K(n)\right)\right\}\right|\]
where $\rho_i^{(n)}$ is the element of $\Dc^{(n)}$ which is $i\th$ closest to $\rho$. Notice that, whereas $\chi_{n,t}$ only accumulates time up to reaching $\rho_{\lfloor nt\rfloor}$, each $\chi_n^{i}$ may have contributions at arbitrarily large times. Recall that $A_2^{(0)}(n,t)$ is the event that the walk never backtracks distance $\overline{C}\log(n)$ along the backbone from a backbone vertex up to level $\lfloor nt\rfloor$. On $A_2^{(0)}(n,T)$ we therefore have that for all $t\leq T$ \[ \sum_{i=1}^{|\Dc_{\lfloor nt-\overline{C}\log(n)\rfloor}^{(n)}|}\chi^i_n \; \leq \; \chi_{n,t} \; \leq \; \sum_{i=1}^{|\Dc_{\lfloor nt\rfloor}^{(n)}|}\chi^i_n \]
where the $J_1$ distance between the two sums in the above expression can be bounded above by $\overline{C}\log(n)/n$. In particular, using that $A_2^{(0)}(n,T)$ occurs with high probability and the tightness result we prove in Section \ref{tght}, in order to prove Theorems \ref{finexcthm}, \ref{finvarthm} and \ref{infallthm} it will suffice to consider the time spent in large traps up to level $\lfloor nt\rfloor$ under the appropriate scaling. 

Let $(X_n^{(i)})_{i\geq 1}$ be independent walks on the same tree as $X_n$ and $(Y_n^{(i)})_{i\geq 1}$ the corresponding backbone walks. Then for $i\geq 1$ let $\tilde{\chi}_n^i$ be the time spent in the $i\th$ large trap by $X_n^{(i)}$ and \[\tilde{\chi}_{t,n}:=\sum_{i=1}^{\lfloor nt q_n\rfloor }\tilde{\chi}^i_n.\]
$(\tilde{\chi}_n^i)_{i\geq 1}$ are then independent copies (under $\Pb$) of times spent in large branches. Moreover, on $\Dc(n,t)$, the root $\rho$ is not the root of a large branch and therefore $(\tilde{\chi}_n^i)_{i\geq 1}$ are identically distributed. Recalling that we write $r_n$ to be $a_n$ in IVFE, $n^{1/\gamma}$ in FVIE and $a_n^{1/\gamma}$ in IVIE we can now prove the following lemma.
\begin{lem}\label{sumiidstab}
In each of IVFE, FVIE and IVIE, 
\begin{enumerate}
 \item\label{J1dist} as $\nin$ \[d_{J_1}\left(\left(\sum_{i=1}^{|\Dc_{\lfloor nt\rfloor}^{(n)}|}\frac{\chi^i_n}{r_n}\right)_{t \in [0,T]},\left(\sum_{i=1}^{\lfloor tnq_n \rfloor}\frac{\chi^i_n}{r_n}\right)_{t \in [0,T]}\right) \]
 converges to $0$ in probability where $d_{J_1}$ denotes the Skorohod $J_1$ metric;
 \item for any bounded $H:D([0,T],\Rb)\rightarrow \Rb$ continuous with respect to the Skorohod $J_1$ topology we have that as $\nin$
\[\left|\Eb\left[H\left(\left(\sum_{i=1}^{\lfloor tnq_n \rfloor}\frac{\chi^i_n}{r_n}\right)_{t \in [0,T]}\right)\right]-\Eb\left[H\left(\left(\sum_{i=1}^{\lfloor tnq_n \rfloor}\frac{\tilde{\chi}^i_n}{r_n}\right)_{t \in [0,T]}\right)\right] \right|\rightarrow 0.\]
\end{enumerate}
\begin{proof}
By definition of $d_{J_1}$, the distance in statement \ref{J1dist} is equal to
\begin{flalign*}
 \inf_{\lambda \in \Lambda}\sup_{t \in [0,T]} \left(\left|\sum_{i=1}^{|\Dc_{\lfloor nt\rfloor}^{(n)}|}\frac{\chi^i_n}{r_n}-\sum_{i=1}^{\lfloor \lambda(t)nq_n \rfloor}\frac{\chi^i_n}{r_n}\right|+|\lambda(t)-t|\right).
\end{flalign*}
For $m \in \Nb$ let $\lambda_n(m/n):=|\Dc_{m}^{(n)}|(nq_n)^{-1}$ then define $\lambda_n(t)$ by the usual linear interpolation. It follows that $|\Dc_{\lfloor nt\rfloor}^{(n)}|=\lfloor \lambda_n(t)nq_n \rfloor$ and the above expression can be bounded above by 
\[\sup_{t \in [0,T]}\left|t-\frac{|\Dc_{\lfloor nt\rfloor}^{(n)}|}{nq_n}\right|\]
which converges to $0$ by Lemma \ref{numcrit} since $n^{2\varepsilon/3}(nq_n)^{-1} \rightarrow 0$.

For $i\geq 1$ let \[A_2^{(i)}(n,t):=\bigcap_{j=0}^{\lfloor nt\rfloor}\bigcap_{n \geq \Delta^{Y^{(i)}}_{\rho_j}}\{|Y_n^{(i)}|>j-\overline{C}\log(n)\}\] be the analogue of $A^{(0)}_2(n,t)$ for the $i\th$ copy and $\tilde{A}_2(n,t)=\Dc(n,t)\cap\bigcap_{i=0}^{\lfloor ntq_n\rfloor}A_2^{(i)}(n,t)$ be the event that, on each of the first $\lceil ntq_n\rceil$ copies, the walk never backtracks distance $\overline{C}\log(n)$ and that large branches are of distance at least $n^\kappa$ apart. Letting $\Eb$ denote the expectation on the enlarged space we have that
\begin{flalign*}
 \Eb\left[H\left(\left(\sum_{i=1}^{\lfloor tnq_n \rfloor}\frac{\chi^i_n}{r_n}\right)_{t \in [0,T]}\right)\ind_{\tilde{A}_2(n,T)}\right]=\Eb\left[H\left(\left(\sum_{i=1}^{\lfloor tnq_n \rfloor}\frac{\tilde{\chi}^i_n}{r_n}\right)_{t \in [0,T]}\right)\ind_{\tilde{A}_2(n,T)}\right] 
\end{flalign*}
therefore 
\begin{flalign*}
 &\left|\Eb\left[H\left(\left(\sum_{i=1}^{\lfloor tnq_n \rfloor}\frac{\chi^i_n}{r_n}\right)_{t \in [0,T]}\right)\right]-\Eb\left[H\left(\left(\sum_{i=1}^{\lfloor tnq_n \rfloor}\frac{\tilde{\chi}^i_n}{r_n}\right)_{t \in [0,T]}\right)\right] \right| \\
 & \qqqqqqqqquad \qqquad\leq||H||_\infty\left( \lceil nTq_n\rceil \Pb(A^{(0)}_2(n,T)^c)+\Pr(\Dc(n,T)^c)\right)
\end{flalign*}
which converges to $0$ as $\nin$ for $\overline{C}$ large by the same argument as (\ref{noback}) and that $\Pr(\Dc(n,T)^c)\rightarrow 0$.
\end{proof}
\end{lem}

Using Corollary \ref{uspdiff} and Lemma \ref{sumiidstab}, in order to show the convergence of $\Delta_{\lfloor nt\rfloor}/r_n$, it suffices to show the convergence of the scaled sum of independent random variables $\tilde{\chi}_{t,n}/r_n$.

\section{Excursion times in dense branches}\label{siid}
In this section we only consider IVFE. The main tool will be Theorem \ref{eqconv}, which is Theorem 10.2 in \cite{arfrgaha}, and is itself a consequence of Theorem IV.6 in \cite{pe}. 
\begin{thm}\label{eqconv}
Let $n(t):[0,\infty)\rightarrow \Nb$ and for each $t$ let $\{R_k(t)\}_{k=1}^{n(t)}$ be a sequence of i.i.d.\ random variables. Assume that for every $\epsilon>0$ it is true that
\[ \limt\Pb(R_1(t)>\epsilon)=0.\]
Now let $\Lc(x):\Rb\setminus\{0\}\rightarrow \Rb$ be a real, non-decreasing function satisfying $\lim_{x\rightarrow \infty}\Lc(x)=0$ and $\int_0^ax^2\d \Lc(x)<\infty$ for all $a>0$. Suppose $d \in \Rb$ and $\sigma\geq 0$, then the following statements are equivalent:
\begin{enumerate}
 \item As $t\rightarrow \infty$
 \begin{flalign*}
  \sum_{k=1}^{n(t)}R_k(t) \cd R_{d,\sigma,\Lc}
 \end{flalign*}
where $R_{d,\sigma,\Lc}$ has the law $\Ic(d,\sigma,\Lc)$, that is, \[\Eb[e^{itR_{d,\sigma,\Lc}}]=\exp\left(idt+\int_0^\infty \left(e^{itx}-1-\frac{itx}{1+x^2}\right)\d\Lc(x)\right).\]
\item For $\tau>0$ let $\overline{R}_\tau(t):=R_1(t)\ind_{\{|R_1(t)| \leq \tau\}}$ then for every continuity point $x$ of $\Lc$
\begin{flalign*}
d & = \limt n(t) \Eb[\overline{R}_\tau(t)]+\int_{|x|>\tau}\frac{x}{1+x^2}\d\Lc(x)-\int_{0<|x|\leq \tau}\frac{x^3}{1+x^2}\d\Lc(x), \\
\sigma^2 & = \lim_{\tau\rightarrow 0}\limsup_{t\rightarrow \infty}n(t)Var(\overline{R}_\tau(t)), \\
 \Lc(x) & = \begin{cases} \limt n(t)\Pb(R_1(t)\leq x) & x<0 \\ -\limt n(t)\Pb(R_1(t)>x) & x>0 \end{cases} 
\end{flalign*}
\end{enumerate}
\end{thm}

In our case, $n(t)$ will be the number of large branches up to level $\lfloor nt \rfloor$ and $\{R_k\}_{k=1}^{n(t)}$ independent copies of the time spent in a large branch. 

Since we are now working with i.i.d.\ random variables we will simplify notation by considering a dummy branch $\Tc^{*-}$, denote its root $\rho^n$ and the number of traps it contains $N (\ed \xi^*-1|\xi^*>l_n^\varepsilon)$. Each of these traps $\{\Tc_j\}_{j=1}^{N}$ is rooted at a bud of $\rho^n$ which we denote $\{\rho^n_j\}_{j=1}^{N}$. We then write $W^j=|\{m\geq 0: X_m=\rho^n, \; X_{m+1}=\rho^n_j\}|$ for $j\leq N$ to be the number of entrances into the $j\th$ trap and $T^{j,k}$ to be the duration of the $k\th$ excursion into the $j\th$ trap for $k\leq W^j$. Recall we consider the duration of the excursion to be the time between leaving and returning to $\rho^n$ thus $T^{j,k} \ed \tau^+_{\rho^n} | \tau^+_{\rho^n_j}<\tau^+_{\rho^n}<\infty$ for the walk started at $\rho^n$. We then have that, for any $i$, \begin{flalign}\label{chitil}
\tilde{\chi}^i_n \ed \sum_{j=1}^{N}\sum_{k=1}^{W^j}T^{j,k}=:\tilde{\chi}_n.
\end{flalign}

Figure \ref{dummy} shows an example of such a dummy tree $\Tc^{*-}$.
\begin{figure}[H]
\centering
 \includegraphics[scale=0.2]{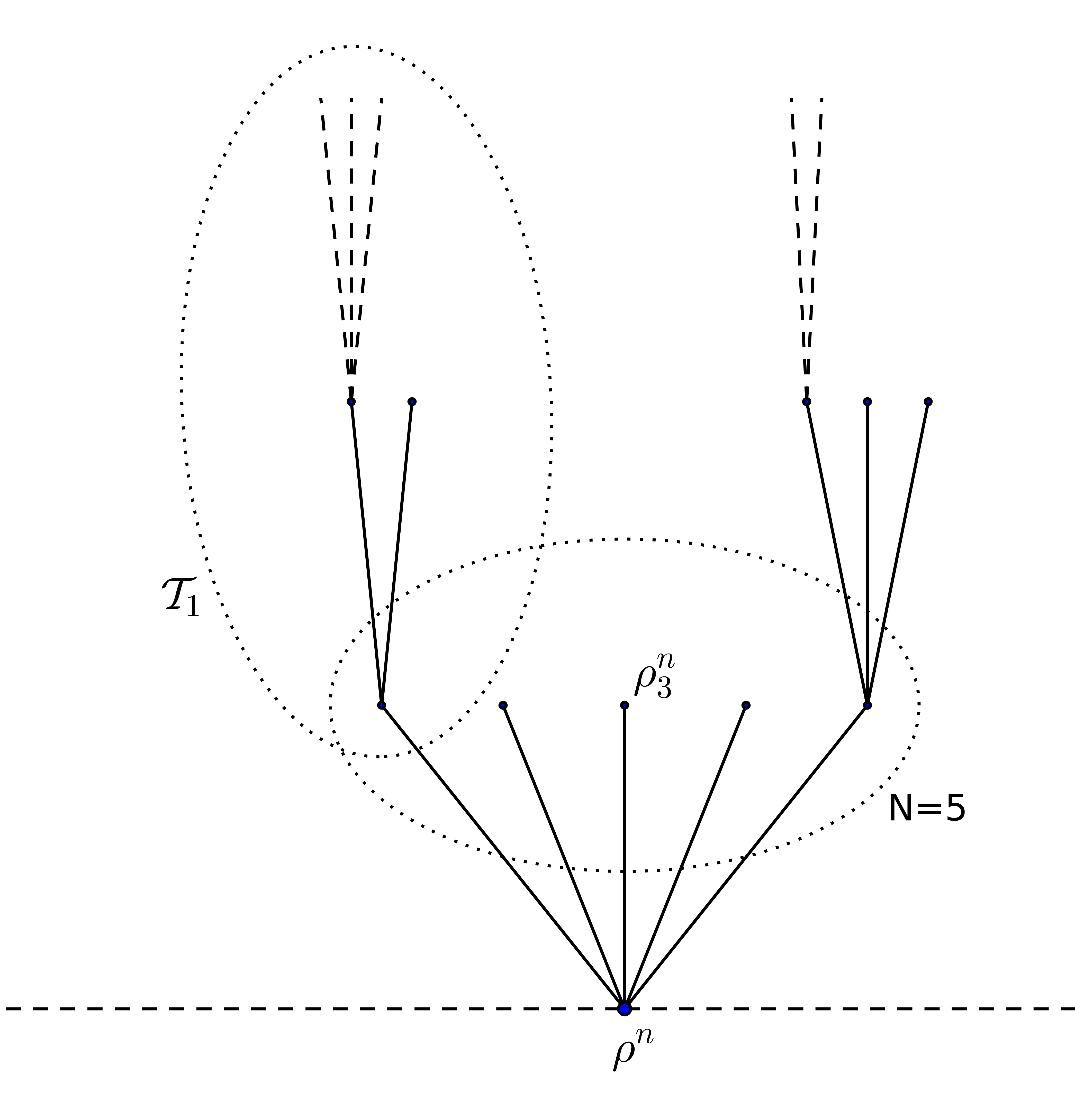} 
\caption{Dummy tree.}\label{dummy}
\end{figure}

For $K \geq l_n^\varepsilon-l_n^0$ write $\overline{L}_K=l_n^0+K$ then denote $\Pb^{K}(\cdot)=\Pb\left(\cdot|N=\overline{L}_K\right)$ and $\Pr^K(\cdot) =\Pr\left(\cdot|N=\overline{L}_K\right)$. We now proceed to show that under $\Pb^K$ \[ \zeta^{(n)}=\frac{1}{N}\sum_{j=1}^{N}\sum_{k=1}^{W^j}T^{j,k}\]
converges in distribution to some random variable $Z_\infty$ whose distribution doesn't depend on $K$.

We start by showing that $T^{j,k}$ don't differ too greatly from $\Et^\Tc[T^{j,k}]$. In order to do this we require moment bounds on $T^{j,k}$ however since $\xi$ has infinite variance it follows that we don't have finite variance of the excursion times and thus we require a more subtle treatment. Using (\ref{return}) we have that the expected excursion time in a trap $\Tc_{\rho_j^n}$ is  
\begin{flalign}\label{qnchexc}
\Et^\Tc[T^{j,k}]=\Et^{\Tc}_{\rho^n}[\tau^+_{\rho^n}|X_1=\rho_j^n]=\sum_{n=0}^\infty Z_n^{\Tc_{\rho_j^n}}\beta^n \leq \Hc(\Tc_{\rho_j^n}) \sup_nZ_n^{\Tc_{\rho_j^n}}\beta^n
\end{flalign}
where $\Tc_{\rho_j^n}$ under $\Pr^K$ has the distribution of an $f$-GW tree. Using that $\Pr(Z_n>0)\sim c_\mu\mu^n$ we see that for $n$ large there are no traps of height greater than $C\log(n)$ for some constant $C$ thus for our purposes it will suffice to study $\sup_nZ_n\beta^n$.

\begin{lem}\label{supmart}
 Let $Z_n$ be a subcritical Galton-Watson process with mean $\mu$ and offspring $\xi$ satisfying $\Er[\xi^{1+\epsilon}]<\infty$ for some $\epsilon>0$. Suppose $1<\beta<\mu^{-1}$, then there exists $\kappa>0$ such that for all $\delta \in (0,\kappa)$ we have that $(Z_n\beta^n)^{1+\delta}$ is a supermartingale.
 \begin{proof}
    Let $\Fc_n=\sigma(Z_k; \; k \leq n)$ denote the natural filtration of $Z_n$.
  \begin{flalign*}
   \Er[(Z_n\beta^n)^{1+\delta}|\Fc_{n-1}]  & = (Z_{n-1}\beta^{n-1})^{1+\delta}\beta^{1+\delta}\Er\left[\left(\sum_{k=1}^{Z_{n-1}}\frac{\xi_k}{Z_{n-1}}\right)^{1+\delta}\Big|Z_{n-1}\right] \\
   & \leq (Z_{n-1}\beta^{n-1})^{1+\delta}\beta^{1+\delta}\Er\left[\sum_{k=1}^{Z_{n-1}}\frac{\xi_k^{1+\delta}}{Z_{n-1}}\Big|Z_{n-1}\right] \\
   & = (Z_{n-1}\beta^{n-1})^{1+\delta}\beta^{1+\delta}\Er[\xi^{1+\delta}] \\
  \end{flalign*}
where the inequality follows by convexity of $f(x)=x^{1+\delta}$. From this it follows that for $\delta\in(0,\alpha-1)$
\begin{flalign*}
 \Er[(Z_n\beta)^{1+\delta}] \; \leq \; \Er[(Z_{n-1}\beta)^{1+\delta}]\Er[(\xi\beta)^{1+\delta}] \; \leq \; \Er[(\xi\beta)^{1+\delta}]^n \; < \; \infty.
\end{flalign*}

Fix $\lambda=(\mu/\beta)^{1/2}$ then $\mu<\lambda$ and for $\delta>0$ sufficiently small $\lambda \beta^{1+\delta}<1$. By dominated convergence $\Er[\xi^{1+\delta}]<\lambda$ for all $\delta$ small. In particular, $\beta^{1+\delta}\Er[\xi^{1+\delta}]<1$ for $\delta$ suitably small and therefore $(Z_n\beta^n)^{1+\delta}$ is a supermartingale.
 \end{proof}
\end{lem}

\begin{lem}\label{cnv}
 In IVFE, we can choose $\varepsilon>0$ such that for any $t>0$
 \[ \sup_{K\geq-(a_n-l_n^{\varepsilon})}\Pb^K\left(\left|\frac{1}{\overline{L}_K}\sum_{j=1}^{\overline{L}_K}\sum_{k=1}^{W^j}(T^{j,k}-\Et^\Tc[T^{j,1}])\right|>t\right) \leq r(n)n^{-\varepsilon}\]
 for some function $r:\Nb\rightarrow \Rb^+$ such that $r(n)=o(1)$.
 \begin{proof}
Write $E_m:=\bigcap_{j=1}^m \left\{\Hc(\Tc_j) \leq C\log(m)\right\}$ to be the event that none of the first $m$ trees have height greater than $C\log(m)$. Since we have that $\Pb(\Hc(\Tc_j)\geq m)\sim c_\mu\mu^m$ we can choose $c>c_\mu$ such that
  \begin{flalign*}
   \Pr(E_m^c) \; = \; 1-\Pb(\Hc(\Tc_j)\leq C\log(m))^m \;  \leq \;  1-(1-cm^{-C\log(\mu^{-1})})^m. 
  \end{flalign*}
Thus choosing $C>1/\log(\mu^{-1})$ and $c=C\log(\mu^{-1})-1>0$ we have that $\Pr(E^c)\leq \tilde{C}m^{-c}$ for $m$ sufficiently large. By Lemma \ref{supmart} we have that $(Z_k\beta^k)^{1+\delta}$ is a supermartingale for $\delta>0$ sufficiently small thus by Doob's supermartingale inequality 
\begin{flalign*}
 \Pr\left(\sup_{k \leq m}Z_k\beta^k \geq x\right)  = \Pr\left(\sup_{k \leq m}(Z_k\beta^k)^{1+\delta} \geq x^{1+\delta}\right)  \leq \Er[Z_0^{1+\delta}]x^{-(1+\delta)}.
\end{flalign*}
Thus, using (\ref{qnchexc}) it follows that 
\begin{equation*}
\Pr\left(\Et^{\Tc_j}[T^{j,1}]> x|\Hc(\Tc_j)\leq C\log(m)\right) \leq C\log(m)^{1+\delta}x^{-(1+\delta)}.
\end{equation*} 
In particular, for some slowly varying function $\tilde{L}$
\begin{equation}\label{tail}
 \Er\left[\Et^{\Tc_j}[T^{j,1}]^2\ind_{\{\Et^{\Tc_j}[T^{j,1}]\leq m\}}|\Hc(\Tc_j)\leq C\log(m)\right] \leq C\tilde{L}(m)m^{1-\delta}.
\end{equation}

Let $\kappa=\delta/(2(1+\delta))$ then write $\overline{E}_m:=E_m \cap \bigcap_{j=1}^m\{\Et^{\Tc_j}[T^{j,1}]\leq m^{1-\kappa}\}$ to be the event that no trap is of height greater than $C\log(m)$ and the expected time spent on an excursion in any trap is at most $m^{1-\kappa}$.
\begin{flalign*}
 \Pr(\overline{E}_m^c) & \leq \Pr\left(\bigcup_{j=1}^m\{\Et^{\Tc_j}[T^{j,1}]> m^{1-\kappa}\}\Big|\Hc(\Tc_j)\leq C\log(m) \; \forall j\leq m\right)+\Pr(E_m^c) \\
 & \leq mC\log(m)^{1+\delta}m^{-(1-\kappa)(1+\delta)}+o(m^{-c}). 
\end{flalign*}
Since $(1-\kappa)(1+\delta)>1$ we have that $\Pr(\overline{E}_m^c)\leq \tilde{C}\left(\log(m)^{1+\delta}m^{1-(1-\kappa)(1+\delta)}+m^{-c}\right)$ for some constant $\tilde{C}$ and $m$ sufficiently large. Write $\overline{\overline{E}}_m:=\overline{E}_m\cap \bigcap_{j=1}^m\{W^j\leq C'\log(m)\}$ for $C'>(2\beta-1)/(\beta-1)$ to be the event that no trap is of height greater than $C\log(m)$, entered more than $C'\log(n)$ times or has expected excursion time greater than $m^{1-\kappa}$. Then, by a union bound and the geometric distribution of $W^j$ from Lemma \ref{numexc}
\begin{flalign}
\Pb\left(\overline{\overline{E}}_m^c\right) & \leq \Pr(\overline{E}_m^c)+m\Pr(W^1> C'\log(m)) \notag \\
& \leq  \tilde{C}\left(\log(m)^{1+\delta}m^{1-(1-\kappa)(1+\delta)}+m^{-c}+ m^{1-C'\frac{\beta-1}{2\beta-1}}\right) \label{Ebarbar}
\end{flalign}
 for $m$ sufficiently large. Choosing $\varepsilon <\min\left\{(1-\kappa)(1+\delta)-1, \; c, \; C'\frac{\beta-1}{2\beta-1}-1 \right\}$ we have that $\Pb\left(\overline{\overline{E}}_m^c\right) = o(m^{-\varepsilon})$ and
\begin{flalign*}
  & \Pb\left(\left|\frac{1}{m}\sum_{j=1}^m\sum_{k=1}^{W^j}(T^{j,k}-\Et^{\Tc_j}[T_{j,k}])\right|>t\right)  \\
 & \qqqquad \leq \Eb\left[\frac{\sum_{j=1}^mC\log(m)Var_{\Pt^{\Tc_j}}((T^{j,1}-\Et^\Tc[T^{j,1}])\ind_{\overline{\overline{E}}_m})}{(mt)^2}\right] +\Pb\left(\overline{\overline{E}}_m^c\right) \\
 & \qqqquad \leq \frac{C\log(m)}{mt^2}m^{(1-\delta)}\tilde{L}(m)+o(m^{-\varepsilon}) 
\end{flalign*}
for some slowly varying function $\tilde{L}$. Here the first inequality comes from Chebyshev and the second holds due to (\ref{tail}). Since $\delta>0$ we can choose $\varepsilon>0$ such that 
\[\Pb\left(\left|\frac{1}{m}\sum_{j=1}^m\sum_{k=1}^{W^j}(T_{j,k}-\Et^{\Tc_j}[T_{j,k}])\right|>t\right) = o\left(m^{-\varepsilon}\right).\]
In particular this holds for $m=\overline{L}_K\geq a_{n^{1-\varepsilon}}$ thus the result holds for $\varepsilon$ sufficiently small since $\alpha<2$.
 \end{proof}
\end{lem}

Using this we can now show that the average time spent in a trap indeed converges to its expectation.
\begin{lem}\label{cnv2}
In IVFE, we can find $\varepsilon>0$ such that for sufficiently large $n$ we have that
 \[\sup_{K\geq-(a_n-l_n^{\varepsilon})}\Pb^K\left(\left|\frac{1}{\overline{L}_K}\sum_{j=1}^{\overline{L}_K}W^j(\Et^\Tc[T^{j,1}]-\Eb[T^{1,1}])\right|>t\right)\leq r(n)\left(n^{-\varepsilon}+\frac{C}{t}\right)\]
 uniformly over $t\geq0$ where $r(n)=o(1)$.
 \begin{proof}
 We continue using the notation defined in Lemma \ref{cnv} and also write \[E_m^j:=\{\Hc(\Tc_j) \leq \tilde{C}\log(m)\} \cap \{W^j\leq C\log(m)\} \cap \{\Et^{\Tc_j}[T^{j,1}]\leq m^{1-\kappa}\}.\]
 We then have that
  \begin{flalign*}
   & \Pb\left(\left|\frac{1}{m}\sum_{j=1}^mW^j(\Et^{\Tc_j}[T^{j,1}]-\Eb[T^{1,1}])\right|>t\right) \\
   & \qquad \qquad \leq \Eb\left[\Pb\left(\left|\frac{1}{m}\sum_{j=1}^mW^j(\Et^{\Tc_j}[T^{j,1}]\ind_{E_m^j}-\Eb[T^{1,1}]\ind_{E_m^j})\right|>t\Big| (W^j)_{j=1}^m\right) \right] +o(m^{-\varepsilon}). 
  \end{flalign*}
Since $\Eb[\Et^\Tc[T^{j,1}\ind_{E_m^j}]]=\Eb[T^{j,1}\ind_{E_m^j}] \neq \Eb[\Eb[T^{1,1}]\ind_{E_m^j}]$ we have that the summand in the right hand side doesn't have zero mean thus we perform the splitting:
\begin{flalign*}
 & \Eb\left[\Pb\left(\left|\frac{1}{m}\sum_{j=1}^mW^j(\Et^\Tc[T^{j,1}]\ind_{E_m^j}-\Eb[T^{j,1}]\ind_{E_m^j})\right|>t\Big| (W^j)_{j=1}^m\right) \right] \\
 & \qquad \qquad \leq \Eb\left[\Pb\left(\left|\frac{1}{m}\sum_{j=1}^mW^j(\Et^\Tc[T^{j,1}]\ind_{E_m^j}-\Eb[T^{j,1}\ind_{E_m^j}])\right|>t/3\Big| (W^j)_{j=1}^m\right) \right] \\
 & \qquad \qquad \qquad  + \Eb\left[\Pb\left(\left|\frac{1}{m}\sum_{j=1}^mW^j(\Eb[T^{j,1}\ind_{E_m^j}]-\Eb[T^{j,1}\ind_{E_m^j}]\ind_{E_m^j})\right|>t/3\Big| (W^j)_{j=1}^m\right) \right] \\
 & \qquad \qquad \qquad \qquad + \Eb\left[\Pb\left(\left|\frac{1}{m}\sum_{j=1}^mW^j(\Eb[T^{j,1}]\ind_{E_m^j}-\Eb[T^{j,1}\ind_{E_m^j}]\ind_{E_m^j})\right|>t/3\Big| (W^j)_{j=1}^m\right) \right]. 
\end{flalign*}
By Chebyshev's inequality and the tail bound $\Eb[\Et^\Tc[T^{j,1}]^2\ind_{\{E_m^j\}}]\leq Cm^{1-\delta}L(m)$ from (\ref{tail}) we have that the first term is bounded above by
\begin{flalign*}
  \Eb\left[\frac{C\log(m)^2}{(mt/3)^2}\sum_{j=1}^mVar(\Et^\Tc[T^{j,1}]\ind_{E_m^j})\right] \leq C_tm^{-\delta}\tilde{L}(m) 
\end{flalign*}
for some slowly varying function $\tilde{L}$. The second term is equal to
\begin{flalign*}
  \Eb\left[\Pb\left(\left|\frac{1}{m}\sum_{j=1}^mW^j\Eb[T^{1,1}\ind_{E_m^j}]\ind_{(E_m^j)^c}\right|>t/3\Big| (W^j)_{j=1}^m\right) \right]  \leq \Pr\left(\bigcup_{j=1}^m(E_m^j)^c\right) = o(m^{-\varepsilon})
\end{flalign*}
by (\ref{Ebarbar}). Finally, the final term can be written as
\begin{flalign*}
 \Pb\left(\frac{1}{m}\sum_{j=1}^mW^j\Eb[T^{j,1}\ind_{(E_m^j)^c}]\ind_{E_m^j}>t/3\right)  \leq \frac{3}{mt}\sum_{j=1}^m\Eb[W^j]\Eb[T^{j,1}\ind_{(E_m^j)^c}]  = \frac{C}{t}\Eb[T^{1,1}\ind_{(E_m^1)^c}] 
\end{flalign*}
which converges to $0$ as $\minf$ by dominated convergence since, by (\ref{expexc}), $\Eb[T^{1,1}]<\infty$. We therefore have that both statements hold by setting $m=\overline{L}_K$. 
 \end{proof}
\end{lem}

From Lemmas \ref{cnv} and \ref{cnv2} we have that as $n\rightarrow \infty$ \[ \Pb^K\left(\left|\zeta^{(n)}-\Eb[T^{1,1}]\sum_{j=1}^{\overline{L}_K}\frac{W^j}{\overline{L}_K}\right|>t\right)\rightarrow 0\]
uniformly over $K\geq -(a_n-l_n^\varepsilon)$. A simple computation using (\ref{return}) shows that $\Eb[T^{1,1}]=2/(1-\beta\mu)$. Write $\theta=(\beta-1)(1-\beta\mu)/(2\beta)$ and let $Z^\infty \sim \exp(\theta)$. 
\begin{cly}\label{feconv}
 In IVFE, we can find $\varepsilon>0$ such that for sufficiently large $n$ we have that
 \[\sup_{K\geq-(a_n-l_n^{\varepsilon})}\left|\Pb^K\left(\zeta^{(n)}>t\right)-\Pb\left(Z^\infty>t\right)\right|\leq \tilde{r}(n)\left(n^{-\varepsilon}+\frac{C}{t} \right)\]
  uniformly over $t\geq0$ where $\tilde{r}(n)=o(1)$.
  \begin{proof}
   By Lemma \ref{numexc} the sum of $W^j$ have a geometric law. In particular,
\begin{flalign*}
 \left|\Pb(Z^\infty>t)-\Pb^K\left(\Eb[T^{1,1}]\sum_{j=1}^{\overline{L}_K}\frac{W^j}{\overline{L}_K}>t\right)\right| & =\left|e^{-\theta t}-\Pb^K\left(Geo\left(\frac{\beta-1}{(\overline{L}_K+1)\beta-1}\right)>\overline{L}_Kt\Eb[T^{1,1}]\right)\right| \\
 & = \left|e^{-\theta t}-\left(1-\frac{\beta-1}{(\overline{L}_K+1)\beta-1}\right)^{\lceil \overline{L}_Kt\Eb[T^{1,1}]\rceil} \right| \\
 & = \left|e^{-\theta t}-e^{-\theta t \frac{\overline{L}_K \beta}{\overline{L}_K\beta+\beta-1}} \right| +o(\overline{L}_K^{-1})  \\
 &  \leq Ce^{-\theta t}\overline{L}_K^{-1} +o(\overline{L}_K^{-1}) 
\end{flalign*}
for some constant $C$ independent of $K$. It therefore follows that the laws of $\zeta^{(n)}$ converge under $\Pb^K$ to an exponential law. In particular, using Lemmas \ref{cnv} and \ref{cnv2} with the bound
\begin{flalign*}
 &\left|\Pb^K\left(\zeta^{(n)}>t\right)-\Pb\left(\Eb[T^{1,1}]\sum_{j=1}^{\overline{L}_K}\frac{W^j}{\overline{L}_K}>t\right)\right| \\
 & \qquad \leq \Pb^K\left(\left|\frac{1}{\overline{L}_K}\sum_{j=1}^{\overline{L}_K}W^j(\Et^\Tc[T^{j,1}]-\Eb[T^{1,1}])\right|>\delta\right) + \Pb(Z^\infty \in [t-\delta,t+\delta])+O(\overline{L}_K^{-1})
\end{flalign*}
with $\delta=r(n)^{1/2}t$, we have the result.
  \end{proof}
\end{cly}

\begin{cly}\label{colexp}
 In IVFE, for any $\tau>0$ fixed \[\limn \sup_{C\geq0}\sup_{K\geq -(a_n-l_n^\varepsilon)}(C\lor 1)\left|\Eb\left[Z^\infty\ind_{\{CZ^\infty\leq \tau\}}\right]-\Eb^K\left[\zeta^{(n)}\ind_{\{C\zeta^{(n)}\leq \tau\}}\right]\right|=0.\]
\end{cly}

Lemma \ref{expsv} shows that the product of an exponential random variable with a heavy tailed random variable has a similar tail to the heavy tailed variable.
\begin{lem}\label{expsv}
 Let $X \sim exp(\theta)$ and $\xi$ be an independent variable which belongs to the domain of attraction of a stable law of index $\alpha$. Then $\Pr(X\xi>x) \sim \theta^{-\alpha}\Gamma(\alpha+1)\Pr(\xi>x)$ as $x \rightarrow \infty$.
 \begin{proof}
  Fix $0<u<1<v<\infty$ then $\forall y \leq u$ we have that $x/y>x$ thus $\Pr(\xi\geq x/y)\leq \Pr(\xi\geq x)$ it therefore follows that
  \begin{flalign*}
   0  \leq \int_0^u \theta e^{-\theta y}\frac{\Pr(\xi\geq x/y)}{\Pr(\xi\geq x)}\d y \leq \int_0^u \theta e^{-\theta y}\d y  = 1-e^{\theta u}. 
  \end{flalign*}
For $y \in[u,v]$ we have that $\Pr(\xi\geq x/y)/\Pr(\xi\geq x)\rightarrow y^\alpha$ uniformly over $y$ therefore 
    \[ \lim_{x\rightarrow \infty}\int_u^v \theta e^{-\theta y}\frac{\Pr(\xi\geq x/y)}{\Pr(\xi\geq x)}\d y  = \int_u^v\theta e^{-\theta y}y^\alpha \d y. \]
Moreover, since this holds for all $u \geq 0$ and $1-e^{\theta u} \rightarrow 0$ as $u \rightarrow 0$ we have that 
\begin{equation}\label{lower}
\lim_{x \rightarrow \infty}\int_0^v \theta e^{-\theta y}\frac{\Pr(\xi\geq x/y)}{\Pr(\xi\geq x)}\d y = \int_0^v\theta e^{-\theta y}y^\alpha \d y. 
\end{equation}

Since $0<\Pr(\xi\geq x)\leq 1$ for all $x<\infty$ we have that $L$ is bounded away from $0,\infty$ on any compact interval thus satisfies the requirements of Potter's theorem (see for example \cite{bigote} 1.5.4) that if $L$ is slowly varying and bounded away from $0,\infty$ on any compact subset of $[0,\infty)$ then for any $\delta>0$ there exists $A_\delta>1$ such that for $x,y>0$ \[\frac{L(z)}{L(x)}\leq A_\delta \max\left\{ \left(\frac{z}{x}\right)^\delta, \left(\frac{x}{z}\right)^\delta\right\}.\]
 Moreover, $\exists c_1,c_2>0$ such that $c_1t^{-\alpha}L(t) \leq \Pr(\xi\geq t) \leq c_2t^{-\alpha}L(t)$ hence we have that for all $y >v$ $\Pr(\xi\geq x/y)/\Pr(\xi\geq x) \leq Cy^{\alpha+\delta}$. By dominated convergence we therefore have that
\[ \lim_{x \rightarrow \infty}\int_v^\infty \theta e^{-\theta y}\frac{\Pr(\xi\geq x/y)}{\Pr(\xi\geq x)}\d y  =\int_v^\infty \theta e^{-\theta y}y^{\alpha}\d y. \]
 Combining this with (\ref{lower}) we have that 
\begin{flalign*}
\lim_{x \rightarrow \infty}\frac{\Pr(X\xi\geq x)}{\Pr(\xi\geq x)}=\lim_{x \rightarrow \infty}\int_0^\infty \theta e^{-\theta y}\frac{\Pr(\xi\geq x/y)}{\Pr(\xi\geq x)}\d y  =\int_0^\infty \theta e^{-\theta y}y^{\alpha}\d y  = \theta^{-\alpha}\Gamma(\alpha+1).
\end{flalign*}
 \end{proof}
\end{lem}

Recall from (\ref{chitil}) that $\tilde{\chi}^i_n$ are independent copies of $\tilde{\chi}_n$ which is the time spent in a large branch. We write $\tilde{\chi}^\infty_n= NZ^\infty$, fix the sequence $(\lambda_n)_{n\geq 1}$ converging to some $\lambda >0$ and denote $M_n^\lambda:=\lfloor \lambda_nn^\varepsilon\rfloor$.  

\begin{prp}\label{convinffee}
 In IVFE, for any $\lambda>0$, as $\nin$ 
 \begin{flalign*}
  \sum_{i=1}^{M_n^\lambda}\frac{\tilde{\chi}^i_n}{a_n} \cd R_{d_\lambda,0,\Lc_\lambda}
 \end{flalign*}
where
\begin{flalign*}
 d_\lambda & = \int_0^\infty\frac{x}{1+x^2}\d\Lc_\lambda(x) \\
 \Lc_\lambda(x) & = \begin{cases}0 & x <0 \\ -\lambda x^{-(\alpha-1)}\theta^{-(\alpha-1)}\Gamma(\alpha) & x>0. \end{cases}
\end{flalign*}
 \begin{proof}
 Let $\epsilon>0$ then clearly by Markov's inequality
 \begin{flalign*}
   \Pb\left(\frac{\tilde{\chi}_n}{a_n}>\epsilon\right) & \leq \Pb(N\geq a_{n^{1-\varepsilon/2}})+\Pb\left(\sum_{j=1}^{a_{n^{1-\varepsilon/2}}}\sum_{k=1}^{W^j}T^{j,k}\geq \epsilon a_n\right) \\
   & \leq \frac{\Pb(\xi^*-1\geq a_{n^{1-\varepsilon/2}})}{\Pb(\xi^*-1\geq a_{n^{1-\varepsilon}})} + \frac{a_{n^{1-\varepsilon/2}}}{\epsilon a_n}\Eb[W^1]\Eb[T^{1,1}],
  \end{flalign*}
  which converges to $0$ as $\nin$. Thus, by Theorem \ref{eqconv}, it suffices to show that 
  \begin{enumerate}
  \item \[\lim_{\tau \rightarrow 0^+}\limsup_{n \rightarrow \infty} M_n^\lambda Var\left(\frac{\tilde{\chi}_n}{a_n} \ind_{\{\frac{\tilde{\chi}_n}{a_n} \leq \tau\}}\right)=0, \]
   \item \[\Lc_\lambda(x)=\begin{cases} \lim_{n \rightarrow \infty} M_n^\lambda \Pb(\frac{\tilde{\chi}_n}{a_n} \leq x) & x<0, \\ -\lim_{n \rightarrow \infty} M_n^\lambda \Pb(\frac{\tilde{\chi}_n}{a_n} >x) & x>0, \\ \end{cases} \]
    \item \[d_\lambda=\lim_{n \rightarrow \infty} M_n^\lambda \Eb\left[\frac{\tilde{\chi}_n}{a_n} \ind_{\{\frac{\tilde{\chi}_n}{a_n} \leq \tau\}}\right] +\int_{|x|>\tau}\frac{x}{1+x^2} d\Lc_\lambda(x)-\int_{0<|x|\leq \tau}\frac{x^3}{1+x^2}d\Lc_\lambda(x)\]
  \end{enumerate}
  where $d_\lambda$ and $\Lc_\lambda$ are as stated above.

We start with the first condition and since $\lambda_n \rightarrow \lambda$ there exists a constant $C$ such that
\begin{flalign}
 M_n^\lambda Var\left(\frac{\tilde{\chi}_n}{a_n} \ind_{\{\frac{\tilde{\chi}_n}{a_n} \leq \tau\}}\right)  \leq Cn^\varepsilon \Eb\left[\frac{\tilde{\chi}_n}{a_n}^2\ind_{\{\frac{\tilde{\chi}_n}{a_n} \leq \tau\}}\right]   \leq Cn^\varepsilon \left(\tau^2\Pr(N\geq a_n) +\tau\Eb\left[\frac{\tilde{\chi}_n}{a_n} \ind_{\{N<a_n\}}\right]\right). \label{sig}
\end{flalign}
By the definitions of $a_n$ and $N$ we have that
\begin{flalign}
 \Pr(N\geq a_n)  =\frac{\Pr(\xi^*\geq a_n)}{\Pr(\xi^*\geq a_{n^{1-\varepsilon}})}  \sim n^{-\varepsilon}. \label{tausqr}
\end{flalign}
 Conditional on $N$ we have that $W^j$ are independent from $T^{j,k}$ and both have finite means hence
 \begin{flalign*}
  \Eb\left[\frac{\tilde{\chi}_n}{a_n} \ind_{\{N<a_n\}}\right] & =\sum_{r=a_{n^{1-\varepsilon}}}^{a_n-1}\frac{\Pr(\xi^*-1=r)}{\Pr(\xi^*-1\geq a_{n^{1-\varepsilon}})}\Eb\left[\sum_{j=1}^r\sum_{k=1}^{W^j}\frac{T^{j,k}}{a_n}\Big| N=r\right] \\
  & \leq \frac{\Eb[W^1]\Eb[T^{1,1}]}{\Pr(\xi^*-1\geq a_{n^{1-\varepsilon}})} \Er\left[\frac{\xi^*-1}{a_n}\ind_{\{\xi^*-1\leq a_n\}}\right] \\
  & \sim C n^\varepsilon
 \end{flalign*}
where the asymptotic holds as $\nin$ by (\ref{trn1}). In particular, by combining this with (\ref{tausqr}) in (\ref{sig}) we have that $M_n^\lambda Var(\frac{\tilde{\chi}_n}{a_n} \ind_{\{\frac{\tilde{\chi}_n}{a_n} \leq \tau\}}) \leq C(\tau^2+\tau)$
for some constant $C$ depending on $\lambda$ hence, as $\tau \rightarrow 0^+$, we indeed have convergence to $0$ and therefore the first condition holds.

We now move on to the L\'{e}vy spectral function $\Lc_\lambda$. Clearly for $x<0$ we have that $\Lc_\lambda(x)=0$ since $\tilde{\chi}_n$ is a positive random variable. It therefore suffices to consider $x>0$. We have that $\zeta^{(n)}$ converges in distribution to an exponential random variable $Z^\infty$ with parameter $\theta$ (which is independent of $K$) therefore by Lemma \ref{expsv}
\begin{flalign}
M_n^\lambda\Pb\left(\frac{\tilde{\chi}^\infty_n}{a_n}>x\right) & \sim  \lambda n^\varepsilon \Pb\left(NZ^\infty>xa_n\right) \notag \\
 & \sim \frac{\lambda}{\Pr(\xi^*-1\geq a_n)} \sum_{K\geq l_n^\varepsilon-l_n^0}\Pr(\xi^*-1=\overline{L}_K)\Pb(\overline{L}_KZ^\infty>xa_n) \notag \\
  & = \lambda \frac{\Pb((\xi^*-1)Z^\infty\geq xa_n)}{\Pr(\xi^*-1\geq a_n)}-\sum_{j=0}^{l_n^\varepsilon-1}\frac{\lambda\Pr(\xi^*-1=j)\Pb(jZ^\infty>xa_n)}{\Pr(\xi^*-1\geq a_n)}  \notag \\
   & \sim  \lambda \theta^{-(\alpha-1)}\Gamma(\alpha) x^{-(\alpha-1)}. \label{levyconv}
\end{flalign}
Where the final asymptotic holds because  
\begin{flalign*}
 \sum_{j=0}^{l_n^\varepsilon-1}\frac{\lambda\Pr(\xi^*-1=j)\Pb(jZ^\infty>xa_n)}{\Pr(\xi^*-1\geq a_n)}   \leq \lambda \frac{\Pb(Z^\infty>xa_n/a_{n^{1-\varepsilon}})}{\Pr(\xi^*-1\geq a_n)}  = \lambda \frac{e^{-\theta x \frac{a_n}{l_n^\varepsilon}}}{\Pr(\xi^*-1\geq a_n)}
\end{flalign*}
which converges to $0$ as $\nin$. 
It now suffices to show that $n^\varepsilon  \left(\Pb(\frac{\tilde{\chi}^\infty_n}{a_n}>x)-\Pb(\frac{\tilde{\chi}_n}{a_n}>x)\right)$ converges to $0$ as $\nin$.
To do this we condition on the number of buds:
\begin{flalign*}
 \Pb\left(\frac{\tilde{\chi}^\infty_n}{a_n}>x\right)-\Pb\left(\frac{\tilde{\chi}_n}{a_n}>x\right) =  \sum_{K\geq l_n^\varepsilon-l_n^0}\Pr(N=\overline{L}_K)\left(\Pb\left(\frac{\overline{L}_KZ^\infty}{a_n}>x\right)-\Pb^K\left(\frac{\overline{L}_K\zeta^{(n)}}{a_n}>x\right)\right).
\end{flalign*}
 We consider positive and negative $K$ separately. For $K \geq 0$ we have that
 \begin{flalign}
  & \sum_{K=0}^\infty n^\varepsilon\Pr(N=\overline{L}_K)\left|\Pb\left(\frac{\overline{L}_KZ^\infty}{a_n}>x\right)-\Pb^K\left(\frac{\overline{L}_K\zeta^{(n)}}{a_n}>x\right)\right| \notag \\
  & \qqqqquad \leq n^\varepsilon\Pr(N\geq a_n)\sup_{\substack{\scriptscriptstyle{c\leq 1} \\ \scriptscriptstyle{K\geq 0}}} \left|\Pb^K(Z^\infty>cx)-\Pb^K(\zeta^{(n)}>cx)\right|. \label{Kpos}
 \end{flalign}
 By (\ref{tausqr}) $n^\varepsilon\Pr(N\geq a_n)$ converges as $\nin$ hence, using Corollary \ref{feconv}, (\ref{Kpos}) converges to $0$.
 For $K\leq 0$, by Corollary \ref{feconv} we have that 
 \begin{flalign*}
  & \sum_{K=-\infty}^0 \ind_{\{K \geq l_n^\varepsilon-l_n^0\}}n^\varepsilon\Pr(N=\overline{L}_K)\left|\Pb\left(\frac{\overline{L}_KZ^\infty}{a_n}>x\right)-\Pb^K\left(\frac{\overline{L}_K\zeta^{(n)}}{a_n}>x\right)\right| \\
  & \qqqqquad \leq n^\varepsilon\sum_{K=-\infty}^0\ind_{\{K \geq l_n^\varepsilon-l_n^0\}}\Pr(N=\overline{L}_K)\tilde{r}(n)\left(n^{-\varepsilon}+\frac{C_x\overline{L}_K}{a_n}\right) \\
  & \qqqqquad \leq o(1)+\frac{C_x\tilde{r}(n)n^\varepsilon}{a_n}\sum_{K=-\infty}^0\ind_{\{K \geq l_n^\varepsilon-l_n^0\}}\frac{\Pr(\xi^*-1=\overline{L}_K)}{\Pr(\xi^*-1\geq l_n^\varepsilon)}\overline{L}_K. 
 \end{flalign*}
 For some constant $C$ we have that $\Pr(\xi^*-1\geq l_n^\varepsilon)\sim Cn^{-(1-\varepsilon)}$ thus by (\ref{trn1})
 \begin{flalign*}
   \frac{C_x\tilde{r}(n)n^\varepsilon}{a_n}\sum_{K=-\infty}^0\ind_{\{K \geq l_n^\varepsilon-l_n^0\}}\frac{\Pr(\xi^*-1=\overline{L}_K)}{\Pr(\xi^*-1\geq l_n^\varepsilon)}\overline{L}_K   \leq C_x\tilde{r}(n)n \Eb\left[\frac{\xi^*-1}{a_n}\ind_{\{\xi^*-1\leq a_n\}}\right] \sim C_x\tilde{r}(n).
 \end{flalign*}
 In particular, since $\tilde{r}(n)=o(1)$, we indeed have that this converges to zero and thus we have the required convergence for $\Lc_\lambda$.

Finally, we consider the drift term $d_\lambda$. Clearly, since $\int_{0<x\leq \tau}x\d\Lc_\lambda(x)<\infty$ we have that
\begin{flalign*}
 d_\lambda & = \lim_{n \rightarrow \infty} M_n^\lambda \Eb\left[\frac{\tilde{\chi}_n}{a_n} \ind_{\{\frac{\tilde{\chi}_n}{a_n} \leq \tau\}}\right] +\int_0^\infty\frac{x}{1+x^2} d\Lc_\lambda(x)-\int_0^\tau x\d\Lc_\lambda(x). 
\end{flalign*}
We want to show that $d_\lambda=\int_0^\infty\frac{x}{1+x^2}\d\Lc_\lambda(x)$ thus we need to show that the other terms cancel. By definition of $N$ we have that 
 \begin{flalign*}
  \Eb\left[\frac{\tilde{\chi}^\infty_n}{a_n}\ind_{\{\frac{\tilde{\chi}^\infty_n}{a_n}\leq \tau\}}\right]  = \frac{1}{\Pr(\xi^*-1\geq l_n^\varepsilon)}\Eb\left[\frac{(\xi^*-1)Z^\infty}{a_n}\ind_{\{\frac{(\xi^*-1)Z^\infty}{a_n}\leq \tau\}}\right].
 \end{flalign*}
 By Lemma \ref{expsv}, $(\xi^*-1)Z^\infty$ belongs to the domain of attraction of a stable law of index $\alpha-1$ and satisfies the scaling properties of $\xi^*$ (up to a constant factor). Therefore, we have that
 \begin{flalign*}
  M_n^\lambda \Eb\left[\frac{\tilde{\chi}^\infty_n}{a_n}\ind_{\{\frac{\tilde{\chi}^\infty_n}{a_n}\leq \tau\}}\right] & \sim \frac{\lambda n^\varepsilon}{\Pr(\xi^*-1\geq l_n^\varepsilon)}\Eb\left[\frac{(\xi^*-1)Z^\infty}{a_n}\ind_{\{\frac{(\xi^*-1)Z^\infty}{a_n}\leq \tau\}}\right]  \sim \frac{\alpha-1}{2-\alpha}\tau^{2-\alpha}\lambda \theta^{-(\alpha-1)}\Gamma(\alpha).
 \end{flalign*}
 Using the form of the L\'{e}vy spectral function we have that
 \begin{flalign*}
  \int_0^\tau x\d\Lc_\lambda(x) \; = \; \lambda\theta^{-(\alpha-1)}\Gamma(\alpha) \int_{\tau^{-(\alpha-1)}}^\infty x^{-\frac{1}{\alpha-1}}\d x \; = \; \frac{\alpha-1}{2-\alpha}\tau^{2-\alpha}\lambda \theta^{-(\alpha-1)}\Gamma(\alpha)
 \end{flalign*}
 thus it remains to show that \[ n^\varepsilon \left(\Eb\left[\frac{\tilde{\chi}^\infty_n}{a_n}\ind_{\{\frac{\tilde{\chi}^\infty_n}{a_n}\leq \tau\}}\right]-\Eb\left[\frac{\tilde{\chi}_n}{a_n}\ind_{\{\frac{\tilde{\chi}_n}{a_n}\leq \tau\}}\right]\right)\rightarrow 0.\]
Similarly to the previous parts we condition on $N=\overline{L}_K$ and consider the sums over $K$ positive and negative separately. For $K \leq 0$
\begin{flalign*}
  & n^\varepsilon\sum_{K\leq 0}\Pr(N=\overline{L}_K)\left|\Eb\left[\frac{\overline{L}_KZ^\infty}{a_n}\ind_{\{\frac{\overline{L}_KZ^\infty}{a_n}\leq \tau \}}\right]-\Eb^K\left[\frac{\overline{L}_K\zeta^{(n)}}{a_n}\ind_{\{\frac{\overline{L}_K\zeta^{(n)}}{a_n}\leq \tau \}}\right]\right| \\
  & \qquad \leq \frac{n^\varepsilon}{\Pr(\xi^*\geq l_n^\varepsilon)}\Eb\left[\frac{\xi^*-1}{a_n}\ind_{\{\xi^*\leq a_n\}}\right] \; \sup_{K\leq 0} \left|\Eb\left[Z^\infty\ind_{\{Z^\infty\leq \tau\frac{a_n}{\overline{L}_K} \}}\right]-\Eb^K\left[\zeta^{(n)}\ind_{\{\zeta^{(n)}\leq \tau\frac{a_n}{\overline{L}_K} \}}\right]\right|.
\end{flalign*}
By definition of $l_n^\varepsilon$ and properties of stable laws $n^\varepsilon\Eb\left[(\xi^*-1)/a_n\ind_{\{\xi^*\leq a_n\}}\right]/\Pr(\xi^*\geq l_n^\varepsilon)$ converges to some constant as $\nin$. By Corollary \ref{colexp} we therefore have that this converges to $0$. Similarly for $K\geq 0$ we have that 
\begin{flalign*}
  & n^\varepsilon\sum_{K\geq 0}\Pr(N=\overline{L}_K)\left|\Eb\left[\frac{\overline{L}_KZ^\infty}{a_n}\ind_{\{\frac{\overline{L}_KZ^\infty}{a_n}\leq \tau \}}\right]-\Eb^K\left[\frac{\overline{L}_K\zeta^{(n)}}{a_n}\ind_{\{\frac{\overline{L}_K\zeta^{(n)}}{a_n}\leq \tau \}}\right]\right| \\
  & \qquad \leq \frac{n^\varepsilon\Pr(\xi^*\geq l_n^0)}{\Pr(\xi^*\geq l_n^\varepsilon)} \; \sup_{K\geq 0}\frac{\overline{L}_K}{a_n}\left|\Eb\left[Z^\infty\ind_{\{Z^\infty\leq \tau\frac{a_n}{\overline{L}_K} \}}\right]-\Eb^K\left[\zeta^{(n)}\ind_{\{\zeta^{(n)}\leq \tau\frac{a_n}{\overline{L}_K} \}}\right]\right|.
\end{flalign*}
We have that $n^\varepsilon\Pr(\xi^*\geq l_n^0)/\Pr(\xi^*\geq l_n^\varepsilon)$ converges to some constant as $\nin$. The result then follows by Corollary \ref{colexp}.
 \end{proof}
\end{prp}

This shows the convergence result of Theorem \ref{finexcthm} in the sense of finite dimensional distributions. In Section \ref{tght} we prove a tightness result which concludes the proof.

\section{Excursion times in deep branches}\label{wlkindp}
We now want to decompose the time spent in deep branches. In FVIE this will be very similar to the decomposition used in \cite{arfrgaha} and we won't consider the argument in great detail. However, the decomposition required in IVIE requires greater delicacy. In this section we consider a construction of a GW-tree conditioned on its height by Geiger and Kersting \cite{geke} to show that the time spent in deep traps essentially consists of some geometric number of excursions from the deepest point in the trap to itself. That is, as in \cite{arfrgaha}, excursions which don't reach the deepest point are negligible as is the time taken for the walk to reach the deepest point from the root of the trap and the time taken to return to the root from the deepest point when this happens before returning to the deepest point. 

Following notation of \cite{arfrgaha}, denote $(\phi_{n+1},\psi_{n+1})_{n \geq 0}$ a sequence of i.i.d.\ pairs with joint law 
\begin{flalign}\label{phipsi}
\Pr(\phi_{n+1}=j,\psi_{n+1}=k) = \frac{\Pr(\xi=k)\Pr(\Hc(\Tc)\leq n-1)^{j-1}\Pr(\Hc(\Tc)=n)\Pr(\Hc(\Tc)\leq n)^{k-j}}{\Pr(\Hc(\Tc)=n+1)} 
\end{flalign}
for $k = 1,2,...$ and $j=1,...,k$. Under this law $\psi_{n+1}$ has the law of the degree of the root of a GW tree conditioned to be of height $n+1$ and $\phi_{n+1}$ has the law over the first bud to give rise onto a tree of height exactly $n$. We then construct a sequence of trees recursively as follows: Set $T_0=\{\delta\}$ then
\begin{enumerate}
 \item Let the first generation of $T_{n+1}$ be of size $\psi_{n+1}$.
 \item Let $T_n$ be the subtree rooted at the $\phi_{n+1}\th$ first generation vertex of $T_{n+1}$.
 \item Attach $f$-GW trees conditioned to have height at most $n-1$ to the first $\phi_{n+1}-1$ vertices of the first generation of $T_{n+1}$.
 \item Attach $f$-GW trees conditioned to have height at most $n$ to the remaining $\psi_{n+1}-\phi_{n+1}$ first generation vertices of $T_{n+1}$.
\end{enumerate}

Under this construction $T_{n+1}$ has the distribution of an $f$-GW tree conditioned to have height exactly $n+1$. Write $\delta_0=\delta$ to be the deepest point of the tree and for $n=1,2,...$ write $\delta_n$ to be the ancestor of $\delta$ of distance $n$. The sequence $\delta_0,\delta_1,...$ form a `spine' from the deepest point to the root of the tree. We denote $\Tc^-$ to be the tree asymptotically attained. By a subtrap of $\Tc^-$ we mean some vertex $x$ on the spine together with a descendant $y$ off the spine and all of the descendants of $y$. This is itself a tree with root $x$ and we write $S_x$ to be the collection of subtraps rooted at $x$.  Figure \ref{trapch} shows a construction of $T_4$ where the solid line represents the spine and the dashed lines represent subtraps.
\begin{figure}[H]
\centering
 \includegraphics[scale=0.8]{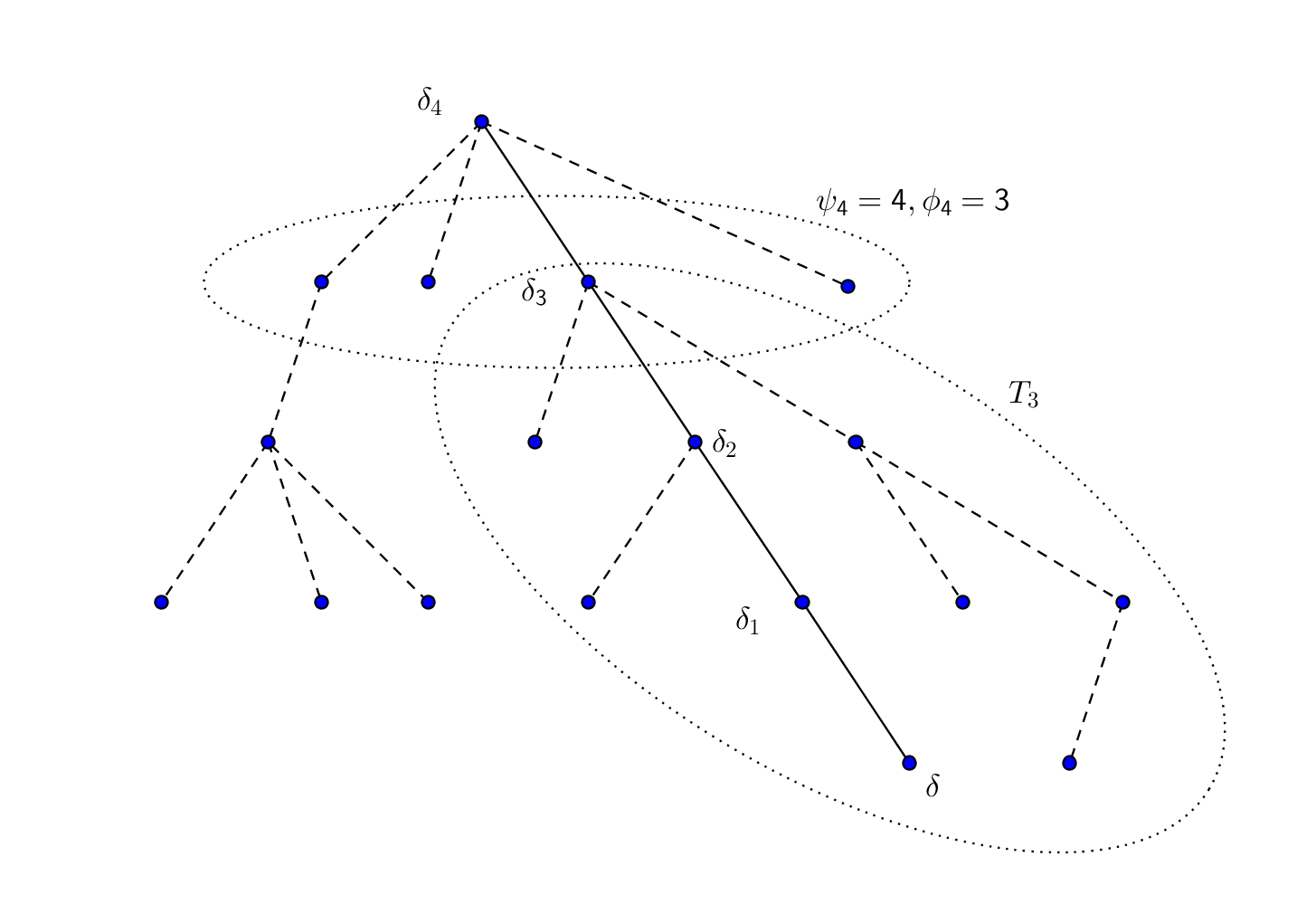} 
\caption{GW-tree conditioned on its height.}\label{trapch}
\end{figure}

We denote $S^{n,j,1}$ to be the $j\th$ subtrap conditioned to have height at most $n-1$ attached to $\delta_n$ and $S^{n,j,2}$ to be the $j\th$ subtrap conditioned to have height at most $n$ attached to $\delta_n$. Recall that $d(x,y)$ denotes the graph distance between $x,y \in \Tc$ then for $k=1,2$ let \[\Pi^{n,j,k}=2\sum_{x \in S^{n,j,k}\setminus\{\delta_n\}}\beta^{d(x,\delta_n)}\] denote the weight of $S^{n,j,k}$ under the invariant measure associated to the conductance model with conductances $\beta^{i+1}$ between levels $i,i+1$ and the roots of $S^{n,j,k}$ (spinal vertices) denoting level $0$. We then write \[\Lambda_n=\sum_{j=1}^{\phi_n-1}\Pi^{n,j,1}+\sum_{j=1}^{\psi_n-\phi_n}\Pi^{n,j,2}\] to denote the total weight of the subtraps of $\delta_n$ then, 
\begin{flalign}\label{Rinfty}
\Et[\Rc_\infty]=2\sum_{n=0}^\infty \beta^{-n}(1+\Lambda_n)
\end{flalign}
is the expected time $\Rc_\infty$ taken for a walk on $\Tc^-$ started from $\delta$ to return to $\delta$. 

\begin{lem}\label{maxexc}
Suppose that $\xi$ belongs to the domain of attraction of a stable law of index $\alpha \in (1,2]$ and $\beta\mu>1$ then  
 \[\Eb[\Rc_\infty]<\infty.\]
 \begin{proof}
  Since $\beta>1$ we have that $2\sum_{n=0}^\infty \beta^{-n}=2/(1-\beta^{-1})$ thus it suffices to find an appropriate bound on $\Er[\Lambda_n]$. 

$\Er[\Pi^{n,j,1}]\leq \Er[\Pi^{n,j,2}]$ since conditioning the height of the trap to be small only reduces the weight; therefore, by independence of $\psi_n$ and $\Pi^{n,j,2}$
\begin{flalign}\label{Lambda}
 \Er[\Lambda_n] \; = \; \Er\left[\sum_{j=1}^{\phi_n-1}\Pi^{n,j,1}+\sum_{j=1}^{\psi_n-\phi_n}\Pi^{n,j,2}\right] \; \leq \; \Er[\Pi^{n,1,2}]\Er[\psi_n].
\end{flalign}
Using that conditioning the height of a GW-tree $\Tc$ to be small only decreases the expected generation sizes and that $\mu\beta>1$, by (\ref{return})
\begin{flalign}\label{PI}
 \Er[\Pi^{n,1,2}] \; = \; 2\sum_{k=1}^n\beta^k\Er[Z_k|\Hc(\Tc)\leq n] \; \leq \; c(\beta\mu)^n
\end{flalign}
for some constant $c$ where $Z_k$ are the generation sizes of $\Tc$. Summing over $j$ in (\ref{phipsi}) shows that $\Pr(\psi_{n+1}=k)  = \Pr(Z_1=k|\Hc(\Tc)=n+1)$. Recalling that $s_n=\Pr(\Hc(\Tc)< n)$, 
\begin{flalign*}
 \Er[\psi_{n+1}]\; = \; \Er[Z_1|\Hc(\Tc)=n+1] \; = \; \sum_{k=1}^\infty kp_k\left(\frac{s_{n+1}^k-s_n^k}{s_{n+2}-s_{n+1}}\right). 
\end{flalign*}
By (\ref{cmu}) $1-s_{n+1} \sim c\mu^n$ for some positive constant $c$ thus $s_{n+2}-s_{n+1} \sim c\mu^n$ for some other positive constant $c$. In particular, when $\sigma^2<\infty$, there exists some constant $c$ such that
\begin{flalign*}
 \sum_{k=1}^\infty kp_k\left(\frac{s_{n+1}^k-s_n^k}{s_{n+2}-s_{n+1}}\right)  \leq c\sum_{k=1}^\infty kp_k\left(\frac{1-s_n^k}{1-s_n}\right) \leq c\sigma^2
\end{flalign*}
where the final inequality comes from that $(1-s^k)(1-s)^{-1}$ is increasing in $s$ and converges to $k$ for each $k\geq1$. It therefore follows that $\Er[\Lambda_n]\leq C(\beta\mu)^n$ so indeed 
\begin{flalign*}
\Eb[\Rc_\infty]  \leq C\sum_{n=0}^\infty \beta^{-n}(\beta\mu)^n  <\infty.
\end{flalign*}

When $\xi$ has infinite variance but belongs to the domain of attraction of a stable law
\begin{flalign*}
 \sum_{k=1}^\infty kp_k(s_{n+1}^k-s_n^k) & = \mu\left(\left(1-\frac{s_nf'(s_n)}{\mu}\right)-\left(1-\frac{s_{n+1}f'(s_{n+1})}{\mu}\right)\right) 
\end{flalign*}
hence by (\ref{prbhn}) as $\nin$ we have that $ \Er[\psi_{n+1}] \sim c\mu^{n(\alpha-2)}L_2(\mu^n)$. Combining this with (\ref{Lambda}) and (\ref{PI}) we have
\begin{flalign}\label{Lambi}
 \Er[\Lambda_n] \; \leq \; C(\beta\mu)^n\mu^{n(\alpha-2)}L_2(\mu^n) \; = \; C(\beta\mu^{\alpha-1})^nL_2(\mu^n)
\end{flalign}
 therefore using (\ref{Rinfty})
 \begin{flalign*}
  \Eb[\Rc_\infty] & \leq C\left(1+\sum_{n=1}^\infty \mu^{n(\alpha-1)}L_2(\mu^n)\right)  <\infty 
 \end{flalign*}
for $C$ chosen sufficiently large.
 \end{proof}
\end{lem}

We therefore have that the expected time taken for a walk started from the deepest point in a trap (of height $H$) to return to the deepest point is bounded above by $\Eb[\Rc_\infty]<\infty$ independently of its height. The following lemma gives the probabilities of reaching the deepest point in a trap, escaping the trap from the deepest point and the transition probabilities for the walk in the trap conditional on reaching the deepest point before escaping. The proof is straightforward by comparison with the biased walk on $\Zb$ with nearest neighbour edges so we omit it. Recall that $\tau^+_x$ is the first return time to $x$.

\begin{lem}\label{transit}
 For any tree $T$ of height $H+1$ (with $H \geq 1$), root $\rho$ and deepest vertex $\delta$ we have that \[\Pt^T_{\delta_H}(\tau^+_\delta<\tau^+_\rho)=\frac{1-\beta^{-1}}{1-\beta^{-(H+1)}}\]
 is the probability of reaching the deepest point without escaping and  \[\Pt^{T}_{\delta}(\tau^+_{\rho}<\tau^+_{\delta})=\frac{1-\beta^{-1}}{\beta^H-\beta^{-1}}\] 
 is the probability of escaping from the deepest point before returning.
 Moreover, \[ \Pt^T_{\delta_k}(\tau^+_{\delta_{k-1}}<\tau^+_{\delta_{k+1}}|\tau^+_{\delta}<\tau^+_\rho)=\frac{1-\beta^{-(H+2-k)}}{1-\beta^{-(H+1-k)}}\cdot\frac{\beta}{\beta+1} \] is the probability that the walk restricted to the spine conditioned on reaching $\delta$ before returning to $\rho$ moves towards $\delta$.
\end{lem}

Since the first two probabilities are independent of the structure of the tree except for the height we write 
\begin{flalign}\label{p1}
p_1(H)=\frac{1-\beta^{-1}}{1-\beta^{-(H+1)}}
\end{flalign}
to be the probability that the walk reaches the deepest vertex in the tree before returning to the root starting from the bud and 
\begin{flalign}\label{p2}
p_2(H)=\frac{1-\beta^{-1}}{\beta^{H}-\beta^{-1}}
\end{flalign}
 to be the probability of escaping from the tree.

For the remainder of the section we will consider only the case that the offspring distribution belongs to the domain of attraction of some stable law of index $\alpha\in(1,2)$. The aim is to prove Proposition \ref{thbck} which shows that the time on excursions in deep traps essentially consists of some geometric number of excursions from the deepest point to itself. We will then conclude with Corollary \ref{thbckst} which is an adaptation for FVIE and of which we omit the proof. 

Write $\rho_i^n$ to be the root of the $i\th$ large branch. This has some number $N^i$ buds which are roots of large traps where, by Proposition \ref{numcritt}, $N^i$ converges to a heavy tailed distribution. Let $\rho_{i,j}^n$ be the bud of the $j\th$ large trap in this branch then $W^{(i,j)}=|\{m\geq 0:X_{m-1}=\rho_i^n, X_m=\rho_{i,j}^n\}| $ is the number of times that the $j\th$ large trap in the $i\th$ large branch is visited. Let $\omega^{(i,j,1)}=\tau^+_{\rho_{i,j}^n}$ then for $k\leq W^{(i,j)}$ write $\omega^{(i,j,k)}=\min\{m > \omega^{(i,j,k-1)}:X_{m-1}=\rho^n_i,X_m=\rho^n_{i,j}\}$ to be the start time of the $k\th$ excursion into $\Tc_{\rho^n_{i,j}}$ and $T^{(i,j,k)}=|\{m \in [\omega^{(i,j,k)}, \omega^{(i,j,k+1)}): X_m \in \Tc^*_{\rho^n_{i,j}}\}|$ its duration. We can then write the time spent in large traps of the $i\th$ large branch as \[\chi^i_n=\sum_{j=1}^{N^i}\sum_{k=1}^{W^{(i,j)}}T^{(i,j,k)}. \]

For $0\leq k\leq \Hc(\Tc_{\rho_{i,j}^n})$ write $\delta^{(i,j)}_k$ to be the spinal vertex of distance $k$ from the deepest point in $\Tc_{\rho_{i,j}^n}$. Let $T^{*(i,j,k)}=0$ if there does not exist $m\in [\omega^{(i,j,k)},\omega^{(i,j,k+1)}]$ such that $X_m=\delta_0^{(i,j)}=:\delta^{(i,j)}$ and 
\begin{flalign*}
T^{*(i,j,k)} & =\sup\{m \in [\omega^{(i,j,k)},\omega^{(i,j,k+1)}]: \; X_m=\delta^{(i,j)}\} \\
& \qquad -\inf\{m \in [\omega^{(i,j,k)},\omega^{(i,j,k+1)}]: \; X_m=\delta^{(i,j)}\}
\end{flalign*}
otherwise to be the duration of the $k\th$ excursion into $\Tc^*_{\rho^n_{i,j}}$ without the first passage to the deepest point and the final passage from the deepest point to the exit. We can then define \begin{flalign}\label{chistar}
\chi^{i*}_n=\sum_{j=1}^{N^i}\sum_{k=1}^{W^{(i,j)}}T^{*(i,j,k)} 
\end{flalign}
to be the time spent in the trap $\Tc_{\rho_{i,j}^n}$ without the first passage to and last passage from $\delta^{(i,j)}$ on each excursion. We want to show that the difference between this and $\chi^i_n$ is negligible. In particular, recalling that $\Dc_n^{(n)}$ is the collection of large branches by level $n$, we will show that for all $t>0$ as $\nin$
\begin{flalign}\label{outback}
\Pb\left(\left|\sum_{i=1}^{|\Dc_n^{(n)}|}\left(\chi^i_n-\chi^{i*}_n\right) \right|\geq ta_n^\frac{1}{\gamma}\right) \rightarrow 0.
\end{flalign}

For $\epsilon>0$ denote 
\begin{flalign}\label{A6}
A_6(n)=  \bigcap_{i=0}^n\{\Hc(\Tc^{*-}_{\rho_i})\leq h_n^{-\epsilon}\}
\end{flalign}
to be the event that there are no $h_n^{-\epsilon}$-branches by level $n$. Using a union bound and (\ref{numcrttrp}) we have that $\Pr(A_6(n)^c)  \leq n\Pr(\Hc(\Tc^{*-}_{\rho_0})>h_n^{-\varepsilon}) \rightarrow 0$ as $\nin$.

Write \[A_7(n)=\bigcap_{i=0}^{|\Dc_n^{(n)}|}\{N^i\leq n^{\frac{2\varepsilon}{\alpha-1}}\}\] to be the event that all large branches up to level $n$ of the backbone have fewer than $n^{\frac{2\varepsilon}{\alpha-1}}$ large traps. Conditional on the number of buds, the number of large traps in the branch follows a binomial distribution therefore 
\[\Pr(N^i\geq Cn^{\frac{2\varepsilon}{\alpha-1}}) \leq \frac{\Pr(\xi^*\geq n^{\frac{1+\varepsilon/2}{\alpha-1}})}{\Pr(\Hc(\Tc^{*-})> h_n^\varepsilon)}+\frac{\Pr\left(Bin\left(n^{\frac{1+\varepsilon/2}{\alpha-1}},\Pr(\Hc(\Tc)\geq h_n^\varepsilon)\right)\geq n^{\frac{2\varepsilon}{\alpha-1}}\right)}{\Pr(\Hc(\Tc^{*-})> h_n^\varepsilon)}.\]
By (\ref{xistar}) and (\ref{numcrttrp}) $\Pr(\Hc(\Tc^{*-})\geq h_n^\varepsilon)\leq Cn^{-(1-\varepsilon)}\overline{L}(n)$ for $n$ large and some slowly varying function $\overline{L}$ hence the first term decays faster than $n^{-\varepsilon}$. Using a Chernoff bound the second term has a stretched exponential decay. Therefore, by (\ref{numcrit}) and a union bound,  $\Pr(A_7(n)^c)  \leq o(1)+Cn^\varepsilon \Pr(N^i\geq n^{\frac{2\varepsilon}{\alpha-1}}) \rightarrow 0$ as $\nin$.
 
For $x \in \Tc$, write $d_x:=|c(x)|$ to be the number of children of $x$ and \[A_8(n)=\bigcap_{i=1}^{|\Dc_n^{(n)}|}\bigcap_{j=1}^{N^i}\left\{\sum_{k=0}^{\Hc(\Tc_{\rho_{i,j}^n})}d_{\delta^{(i,j)}_k} \leq n^{3\varepsilon/(\alpha-1)^2}\right\}\]  to be the event that there are fewer than $n^{3\varepsilon/(\alpha-1)^2}$ subtraps on the spine in any $h_n^\varepsilon$-trap. $\Pr(\xi\geq n|\Hc(\Tc)\geq m)$ is non-decreasing in $m$; therefore, the number of offspring from a vertex on the spine of a trap can be stochastically dominated by the size biased distribution. Using this and (\ref{numcrit}) along with the bounds on $A_6$ and $A_7$ we have that
 \begin{flalign*}
 \Pr(A_8(n)^c) & \leq o(1)+ Cn^\varepsilon n^{\frac{2\varepsilon}{\alpha-1}}\Pr\left(\sum_{k=0}^{h_n^{-\epsilon}}\xi^*_k\geq n^{3\varepsilon/(\alpha-1)^2}\right) \\
  & \leq o(1)+ Cn^\varepsilon n^{\frac{2\varepsilon}{\alpha-1}}h_n^{-\epsilon}\Pr(\xi^* \geq n^{3\varepsilon/(\alpha-1)^2}/h_n^{-\epsilon}) \\
 & \leq o(1)+ n^\varepsilon n^{-\frac{\varepsilon}{\alpha-1}}\overline{L}(n) 
\end{flalign*}
for some slowly varying function $\overline{L}$ thus $\Pr(A_8(n)^c)\rightarrow 0$ as $\nin$.

\begin{prp}\label{thbck}
In IVIE, for any $t>0$ as $\nin$
 \begin{flalign*}
\Pb\left(\left|\sum_{i=1}^{|\Dc_n^{(n)}|}\left(\chi^i_n-\chi^{i*}_n\right) \right|\geq ta_n^\frac{1}{\gamma}\right) \rightarrow 0.
\end{flalign*}
\begin{proof}

Let $A'(n)=\bigcap_{i=1}^8A_i(n)$ then using the bounds on $A_i$ for $i=1,...,8$ it follows that
\begin{flalign}\label{upper}
 \Pb\left(\left|\sum_{i=1}^{|\Dc_n^{(n)}|}\left(\chi^i_n-\chi^{i*}_n\right) \right|\geq ta_n^\frac{1}{\gamma}\right) \leq o(1) + \frac{Cn^{\varepsilon\left(\frac{\alpha+1}{\alpha-1}\right)}}{ta_n^\frac{1}{\gamma}}\Eb\left[\ind_{A'(n)}\sum_{k=1}^{W^{(1,1)}}\left(T_n^{(1,1,k)}-T_n^{*(1,1,k)}\right)\right]. 
\end{flalign}

Since $W^{(i,j)}$ are independent of the excursion times and have marginal distributions of geometric random variables with parameter $(\beta-1)/(2\beta-1)$ 
\begin{flalign*}
\Eb\left[\ind_{A'(n)}\sum_{k=1}^{W^{(1,1)}}\left(T_n^{(1,1,k)}-T_n^{*(1,1,k)}\right)\right] & = \Eb[W^{(1,1)}]\Eb\left[\ind_{A'(n)}\left(T_n^{(1,1,1)}-T_n^{*(1,1,1)}\right)\right]. 
\end{flalign*}
For a given excursion either the walk reaches the deepest point before returning to the root or it doesn't. In the former case the difference $T_n^{(1,1,1)}-T_n^{*(1,1,1)}$ is the time taken to reach $\delta^{(1,1)}$ conditional on the walker reaching $\delta^{(1,1)}$ before exiting the trap added to the time taken to escape the trap from $\delta^{(1,1)}$ conditional on the walk escaping before returning to $\delta^{(1,1)}$. In the latter case the difference is the time taken to return to the root given that the walker returns to the root without reaching $\delta^{(1,1)}$. In particular we have that
\begin{flalign}\label{split}
 \Eb[\ind_{A'(n)}(T_n^{(1,1,1)}-T_n^{*(1,1,1)})] & \leq \Er\left[\ind_{A'(n)}\Et_{\rho_n^{(1,1)}}[\ind_{A'(n)}\tau^+_{\delta^{(1,1)}}|\tau^+_{\delta^{(1,1)}}<\tau^+_{\rho_n^{(1,1)}}]\right] \\
 & \qquad + \Er\left[\ind_{A'(n)}\Et_{\delta^{(1,1)}}[\ind_{A'(n)}\tau^+_{\rho_n^{(1,1)}}|\tau^+_{\rho_n^{(1,1)}}<\tau^+_{\delta^{(1,1)}}]\right] \notag \\
 & \qquad \qquad + \Er\left[\ind_{A'(n)}\Et_{\rho_n^{(1,1)}}[\ind_{A'(n)}\tau^+_{\rho_n^{(1,1)}}|\tau^+_{\rho_n^{(1,1)}}<\tau^+_{\delta^{(1,1)}}]\right]. \notag
\end{flalign}

We want to show that each of the terms in (\ref{split}) can be bounded appropriately. This follows similarly to Lemmas 8.2 and 8.3 of \cite{arfrgaha} so we only sketch the details. Conditional on the event that the walk returns to the root of the trap before reaching the deepest point we have that:
\begin{enumerate}
 \item\label{unchng} the transition probabilities of the walk in subtraps are unchanged,
 \item\label{intotrp} from any vertex on the spine, the walk is more likely to move towards the root than to any vertex in the subtrap,
 \item\label{drfsp} from any vertex on the spine, excluding the root and deepest point, the probability of moving towards the root is at least $\beta$ times that of moving towards the deepest point.
\end{enumerate}
Property \ref{drfsp} above shows that the probability of escaping the trap from any vertex on the spine is at least the probability $p_\infty$ of a regeneration for the $\beta$-biased random walk on $\Zb$. From this we have that the number of visits to any spinal vertex can be stochastically dominated by a geometric random variable with parameter $p_\infty$. Similarly, using property \ref{intotrp} above, we see that the number of visits to any subtrap can be stochastically dominated by a geometric random variable with parameter $p_\infty/2$.

Using a union bound with $A_1,A_7,A_8$ and (\ref{cmu}) we have that with high probability there are no subtraps of height greater than $h_n^\varepsilon$. In particular, by (\ref{expexc}), the expected time in any subtrap can be bounded above by $C(\beta\mu)^{h_n^\varepsilon}$ for some constant $C$ using property \ref{unchng}. From this it follows that
\begin{flalign*}
 \Er\left[\ind_{A'(n)}\Et_{\rho_n^{(1,1)}}[\ind_{A'(n)}\tau^+_{\rho_n^{(1,1)}}|\tau^+_{\rho_n^{(1,1)}}<\tau^+_{\delta^{(1,1)}}]\right] & \leq \Er\left[\ind_{A'(n)}\Et_{\delta^{(1,1)}}[\ind_{A'(n)}\tau^+_{\rho_n^{(1,1)}}|\tau^+_{\rho_n^{(1,1)}}<\tau^+_{\delta^{(1,1)}}]\right] \\
 & \leq o(1) + h_n^{-\varepsilon}\Et[Geo(p_\infty)]+Cn^{\frac{3\varepsilon}{(\alpha-1)^2}}(\beta\mu)^{h_n^\varepsilon} \\
 & \leq o(1) + C\overline{L}(n)n^{\frac{(1-\varepsilon)}{\alpha-1}\frac{\log(\beta\mu)}{\log(\mu^{-1})}+\frac{3\varepsilon}{(\alpha-1)^2}}
\end{flalign*}
for some constant $C$, slowly varying function $\overline{L}$.

A symmetric argument shows that the same bound can be achieved for the first term in (\ref{split}). It then follows that the second term in (\ref{upper}) can be bounded above by $C_tL_1(n)n^{-\frac{1}{\alpha-1}+\tilde{\varepsilon}}$ where $\tilde{\varepsilon}$ can be made arbitrarily small by choosing $\varepsilon$ sufficiently small.
\end{proof}
\end{prp}

A straightforward adaptation of Proposition 8.1 of \cite{arfrgaha} (similar to the previous calculation) shows Corollary \ref{thbckst} which is the corresponding result for FVIE.
\begin{cly}\label{thbckst}
In FVIE, for any $t>0$ as $\nin$
 \begin{flalign*}
\Pb\left(\left|\sum_{i=1}^{|\Dc_n^{(n)}|}\left(\chi^i_n-\chi^{i*}_n\right) \right|\geq tn^\frac{1}{\gamma}\right) \rightarrow 0.
\end{flalign*}
\end{cly}

By Proposition \ref{thbck} and Corollary \ref{thbckst}, in FVIE and IVIE, almost all time up to the walk reaching level $n$ is spent on excursions from the deepest point in deep traps. In the remainder of the section we decompose the time spent on such excursions in a single large branch into an appropriate sum of excursions with finite expected duration. We then use this to show that show that $\chi^{i*}_n$ suitably scaled converges in distribution along the identified subsequences.

Let $\Tc^*$ have the distribution of $\Tc^{*-}$ conditioned on having height greater than $h_n^\varepsilon$ in which we prune the buds which are not roots of trees of height at least $h_n^\varepsilon$. Therefore, $\Tc^*$ has the distribution of a tree where the root $\rho$ has $N$ offspring $(\rho_j)_{j=1}^{N}$, each of which is the root of an independent $f$-GW tree with height $H_j\geq h_n^\varepsilon$. We write $\overline{H}=\Hc(\Tc^*)-1$ to be the height of the largest trap and for $K \in \Zb$ let $\overline{H}_n^{\scriptscriptstyle{K}}= h_n^0+K$ then denote $\Pb^{\scriptscriptstyle{K}}(\cdot) =\Pb(\cdot|\overline{H}=\overline{H}_n^{\scriptscriptstyle{K}})$ and $\Pr^{\scriptscriptstyle{K}}(\cdot) =\Pr(\cdot|\overline{H}=\overline{H}_n^{\scriptscriptstyle{K}})$.

We write $W^j$ to be the total number of excursions into $\Tc^{*}_{\rho_j}$ where by Lemma \ref{numexc}, conditional on $N$, $(W^j)_{j=1}^{N}$ have a joint negative multinomial distribution. We then denote the number of excursions which reach the deepest point $\delta^j$ as $B^j$ which is binomially distributed with $W^j$ trials and success probability $p_1(H_j)$. For each $k\leq B^j$ we define $G^{j,k}$ to be the number of return times to $\delta^j$ on the $k\th$ excursion which reaches $\delta^j$. Then for $l=1,...,G^{j,k}$ let $\Rc^{j,k,l}$ denote the duration of the $l\th$ excursion from $\delta^j$ to itself on the $k\th$ excursion into $\Tc^*_{\rho_j}$ which reaches $\delta^j$. $G^{j,k}$ is geometrically distributed with failure probability $p_2(H_j)$. It then follows that each $\chi^{i*}_n$ is equal in distribution to \[\chi_n^*=\sum_{j=1}^{N}\sum_{k=1}^{B^j}\sum_{l=1}^{G^{j,k}}\Rc^{j,k,l}.\]

Define the scaled excursion time in large traps of a large branch as 
\begin{flalign}\label{zetan}
\zeta^{(n)}=\chi_n^*\beta^{-\overline{H}}=\beta^{-\overline{H}}\sum_{j=1}^{N}\sum_{k=1}^{B^j}\sum_{l=1}^{G^{j,k}}\Rc^{j,k,l} 
\end{flalign}
 then we shall show that $\zeta^{(n)}$ converges in distribution under $\Pb^{\scriptscriptstyle{K}}$ along subsequences $n_k(t)$. 

For $\tilde{\varepsilon}>0$ write \[ A_9(n)=\bigcap_{j=1}^{N}\left\{1 \leq \frac{\beta^{H_j}}{1-\beta^{-1}}\Et[G^{j,1}]^{-1} \leq 1+\tilde{\varepsilon}\right\}. \]

Since $G^{j,k}$ are independent geometric random variables there exist independent exponential random variables $e_{j,k}$ such that \[G^{j,k}=\left\lfloor \frac{e_{j,k}}{-\log(1-p_2(H_j))}\right\rfloor \sim Geo(p_2(H_j)). \]
By (\ref{p2}) we then have that \[\Et[G^{j,1}]=\left(1-\frac{1-\beta^{-1}}{\beta^{H_j}-\beta^{-1}}\right)\left(1-\beta^{-(H_j+1)}\right)\frac{\beta^{H_j}}{1-\beta^{-1}}\] therefore, since $H_j\geq h_n^\varepsilon$, for any $\tilde{\varepsilon}>0$ there exists $n$ large such that $\Pr(A_9(n))=1$.
Write 
\[A_{10}^{(j,k)}(n)=\left\{ (1-\tilde{\varepsilon})G^{j,k}\leq \Et[G^{j,k}]e_{j,k} \leq (1+\tilde{\varepsilon})G^{j,k}\right\}.\]
Then, using convergence of scaled geometric variables to exponential variables (see the proof of part (3) of Proposition 9.1 in \cite{arfrgaha}), we have that there exists a constant $\tilde{C}$ such that for any $\tilde{\varepsilon}>0$ there exists $n$ large such that $\Pt(A_{10}^{(j,k)}(n)^c)\leq \tilde{C}p_2( h_n^\varepsilon)$. Therefore, writing \[A_{10}(n)=\bigcap_{j=1}^{N}\bigcap_{k=1}^{B^{j}} A_{10}^{(j,k)}(n)\] and using that $B^{j}  \leq W^{j} \leq C\log(n)$ and $N\leq \log(n)$ with high probability, a union bound gives us that $\Pt(A_{10}(n)^c) \rightarrow 0$ as $\nin$.

By comparison with the biased random walk on $\Zb$ we have that $p_1(H_j)\geq p_\infty=1-\beta^{-1}$ therefore we can define a random variable $B_\infty^{j}\sim Bin(B^{j},p_\infty/p_1(H_j))$. It then follows that $B^{j}\geq B_\infty^j \sim Bin(W^{j},p_\infty)$. Moreover, for $H_j\geq 1$
\begin{flalign}
 p_1(H_j)-p_\infty  = \frac{1-\beta^{-1}}{1-\beta^{-(H_j+1)}}-(1-\beta^{-1})  \leq  \beta^{-H_j}. \label{p1dif}
\end{flalign}

Write \[A_{11}(n)=\bigcap_{j=1}^{N}\left\{B^{j}= B_\infty^{j}\right\}.\]
Since the marginal distribution of $W^{1}$ doesn't depend on $n$, using (\ref{p1dif}), that $N\leq \log(n)$ with high probability and the coupling between $B^{1}$ and $B_\infty^{1}$ we have that
\begin{flalign}
 \Pb(A_{11}(n)^c) & \leq o(1)+\log(n)\sum_{k=0}^\infty \Pb(W^{1}=k)\Pb(B^{1} \neq B_\infty^{1}|W^{1}=k) \notag\\
 & \leq o(1)+\log(n)\sum_{k=0}^\infty \Pb(W^{1}=k)k(p_1(H_1)-p_\infty) \notag\\
 & \leq o(1) + \log(n)\beta^{- h_n^\varepsilon}\Eb[W^{1}] \label{bnm}
\end{flalign}
which decays to $0$ as $\nin$.
 
By choosing $\varepsilon>0$ sufficiently small we can choose $\kappa$ in the range $\varepsilon/\gamma<\kappa< \min\{2(\alpha-1),\; 1/\gamma\}$ then write  
\[A_{12}(n)=\bigcap_{j=1}^{N}\{\Et[(\Rc_n^{j,1,1})^2]< n^{\frac{\gamma^{-1}-\kappa}{\alpha-1}}\}\] 
to be the event that there are no large traps with expected squared excursion time too large. 

\begin{lem}\label{expsqriv}
 In IVIE, as $\nin$ we have that $\Pb(A_{12}(n)^c)\rightarrow 0$.
 \begin{proof}
  Recall from (\ref{A6}) that, for $\epsilon>0$, $A_6(n)$ is the event that all large branches are shorter than $h_n^{-\epsilon}$ and since $N\leq \log(n)$ with high probability we have that
\begin{flalign*}
 \Pb(A_{12}(n)^c)  \leq o(1)+\log(n)\Pb\left(\ind_{\{A_6(n)\}}\Et[(\Rc_n^{1,1,1})^2]^{1/2}>n^{\frac{\gamma^{-1}-\kappa}{2(\alpha-1)}}\right).
\end{flalign*}
A straightforward argument using conductances (see the proof of Lemma 9.1 in \cite{arfrgaha}) gives \[\Et[(\Rc_n^{1,1,1})^2]^{1/2}\leq C \sum_{y \in \Tc^*_{\rho_1}}\beta^{d(y,\delta^1)/2}\pi(y)\]
where $\pi$ is the invariant measure scaled so that $\pi(\delta^1)=1$ and $d$ denotes the graph distance. We then have that
 \begin{flalign*}
  \Er\left[\ind_{\{A_6(n)\}}\Et[(\Rc_n^{(1,1,1)})^2]^{1/2}\right] & \leq C \Er\left[\ind_{\{A_6(n)\}} \sum_{y \in \Tc^*_{\rho_1}}\beta^{d(y,\delta^1)/2}\pi(y)\right] \\
  & \leq C \Er\left[\ind_{\{A_6(n)\}} \sum_{i\geq 1}\beta^{i/2}\beta^{-i}(1+\Lambda_i)\right] \\
  & \leq C \sum_{i=0}^{ h_n^{-\epsilon}}(\beta^{1/2}\mu^{\alpha-1-\epsilon})^i
 \end{flalign*}
 where the final inequality follows by (\ref{Lambi}). If $\beta^{1/2}\mu^{\alpha-1-\epsilon}\leq 1$ then by Markov's inequality we clearly have that $\Pb(A_{12}(n)^c)\rightarrow 0$ as $\nin$ since $\kappa<\gamma^{-1}$. Otherwise by Markov's inequality 
 \[\Pb(A_{12}(n)^c)\leq o(1)+C\log(n)(\beta^{1/2}\mu^{\alpha-1-\epsilon})^{h_n^{-\epsilon}}n^{\frac{\kappa-\gamma^{-1}}{2(\alpha-1)}}\leq \overline{L}(n)n^{\frac{\kappa}{2(\alpha-1)}-1+\frac{\epsilon}{\alpha-1}\left(\frac{1}{2\gamma}+2-\alpha+\epsilon\right)} \]
for some slowly varying function $\overline{L}$. In particular, since $\kappa <2(\alpha-1)$ we can choose $\epsilon$ sufficiently small such that this converges to $0$ as $\nin$.
 \end{proof}
\end{lem}

Write 
\begin{flalign*}
 A_{13}(n)=  \bigcap_{j=1}^{N}\bigcap_{k=1}^{B^{j}} \bigg\{ (1-\tilde{\varepsilon})G^{j,k}\Et[\Rc_n^{j,1,1}] \leq \sum_{l=1}^{G^{j,k}}\Rc^{j,k,l} \leq (1+\tilde{\varepsilon})G^{j,k}\Et[\Rc_n^{j,1,1}]\bigg\} 
\end{flalign*}
to be the event that on each excursion that reaches the deepest point of a large trap, the total excursion time before leaving the trap is approximately the product of the number of excursions and the expected excursion time. 

\begin{lem}\label{lln} 
In IVIE, as $\nin$ we have that $\Pb(A_{13}(n)^c)\rightarrow 0$.
 \begin{proof}
  With high probability we have that no trap is visited more than $C\log(n)$ by (\ref{A5}) and also $N\leq \log(n)$. Any excursion is of length at least $2$ hence $\Et[\Rc_n^{1,1,1}]\geq 2$. Therefore, by Lemma \ref{expsqriv} and Chebyshev's inequality
\begin{flalign*}
  \Pb(A_{13}(n)^c) 
 &\leq o(1) +C\log(n)^2\Pb\Bigg(\Bigg|\sum_{l=1}^{G^{1,1}}\frac{\Rc_n^{1,1,l}}{\Et[\Rc_n^{1,1,1}]G^{1,1}}-1\Bigg|> \tilde{\varepsilon},  G^{1,1} >0,  \Et[(\Rc_n^{1,1,1})^2]<n^{\frac{\gamma^{-1}-\kappa}{\alpha-1}}\Bigg) \\
 & \leq o(1)+\frac{C\log(n)^2n^{\frac{\gamma^{-1}-\kappa}{\alpha-1}}}{\tilde{\varepsilon}^2}\Et\left[\frac{\ind_{\{G^{1,1}>0\}}}{G^{1,1}}\right]. 
\end{flalign*}
It then follows that since $G^{1,1} \sim Geo(p_2(H_1))$ (where from (\ref{p2}) $p_2(H)$ is the probability that a walk reaches the deepest point in the trap of height $H$) and $p_2(H_1)\leq c\beta^{-h_n^\varepsilon}=ca_{n^{1-\varepsilon}}^{-\frac{1}{\gamma}}$
\begin{flalign*}
 \Et\left[\frac{\ind_{\{G^{(1,1,1)}>0\}}}{G^{(1,1,1)}}\right] \; \leq \; \Et\left[-\frac{p_2(H_1)}{1-p_2(H_1)}\log\left(p_2(H_1)\right)\right] \; \leq \; \overline{L}(n)n^{-\frac{1-\varepsilon}{\gamma(\alpha-1)}}
\end{flalign*}
for some slowly varying function $\overline{L}$. In particular, $\Pb(A_{13}(n)^c) \leq o(1)+L_{\tilde{\varepsilon}}(n)n^{\frac{\frac{\varepsilon}{\gamma}-\kappa}{\alpha-1}}$ which converges to zero by the choice of $\kappa> \varepsilon/\gamma$.
 \end{proof}
\end{lem}

Lemma \ref{itriv} illustrates that the expected time spent on an excursion from the deepest point of a trap of height at least $h_n^\varepsilon$ doesn't differ too greatly from the expected excursion time in an infinite version of the trap. Let $\Rc^j_\infty$ be an excursion time from $\delta^j$ to itself in an extension of $\Tc^*_{\rho_j}$ to an infinite trap constructed according to the algorithm at the beginning of the section where $T_{H_j}$ is replaced by $\Tc^*_{\rho_j}$. Write \[A_{14}(n)=\bigcap_{j=1}^{N}\left\{\Et[\Rc_\infty^j]-\Et[\Rc^{j,k,l}] <\tilde{\varepsilon}\right\}.\]

\begin{lem}\label{itriv}
In IVIE, as $\nin$ we have that $\Pb(A_{14}(n)^c)\rightarrow 0$.
 \begin{proof}
 A straightforward computation similar to that in Proposition 9.1 of \cite{arfrgaha} yields that for some constant $c$ and $n$ sufficiently large
\[0\leq\Et[\Rc_\infty^j]-\Et[\Rc^{j,k,l}] \leq c\beta^{- h_n^\varepsilon/2}\sum_{k=0}^{ h_n^\varepsilon/2} \beta^{-k}(1+\Lambda_k)+2\sum_{k= h_n^\varepsilon/2+1}^\infty \beta^{-k}(1+\Lambda_k) \]
for all $j=1,...,N$ where $\Lambda_k$ are the weights of the extension of $\Tc^*_{\rho_j}$. Recall that $N\leq \log(n)$ with high probability, therefore by (\ref{Lambi}) and Markov's inequality
 \begin{flalign*}
   \Pr(A_{14}(n)^c) & \leq \frac{C\log(n)}{\tilde{\varepsilon}}\Er[ \Et[\Rc_\infty^j]-\Et[\Rc_n^{j,1,1}]]  \\
  & \leq C_{\tilde{\varepsilon}}\log(n)\left(\beta^{-\frac{h_n^\varepsilon}{2}}\sum_{k=0}^\infty \left(\beta^{-k}+\mu^{k(\alpha-1-\tilde{\varepsilon})}\right)+\sum_{k=h_n^\varepsilon/2+1}^\infty \mu^{k(\alpha-1-\tilde{\varepsilon})}\right) \\
   & \leq C_{\tilde{\varepsilon}}\log(n)\left(\beta^{-\frac{h_n^\varepsilon}{2}}+\mu^{h_n^\varepsilon\frac{(\alpha-1-\tilde{\varepsilon})}{2}}\right).
 \end{flalign*}
 Since we can choose $\tilde{\varepsilon}<\alpha-1$ we indeed have the desired result.
 \end{proof}
\end{lem}

The height of the branch and the total number of traps in the branch have a strong relationship. Lemma \ref{hgtnum} shows the exact form of this relationship in the limit as $\nin$. Recall that $\overline{H}_n^{\scriptscriptstyle{K}}=h_n^0+K$ where $h_n^0$ is given in Definition \ref{iviedefn} and $\Pr^{\scriptscriptstyle{K}}$ denotes the law $\Pr$ conditioned on the height of the branch equalling $\overline{H}_n^{\scriptscriptstyle{K}}$. Write $b_n^{\scriptscriptstyle{K}}=\mu^{-\overline{H}_n^{\scriptscriptstyle{K}}}/c_\mu$ and recall from (\ref{cmu}) that $c_\mu$ is the positive constant such that $\Pr(\Hc(\Tc)\geq n) \sim c_\mu \mu^n$ as $\nin$. 

\begin{lem}\label{hgtnum}
 In IVIE, under $\Pr^{\scriptscriptstyle{K}}$ we have that the sequence of random variables $(\xi^*-1)/b_n^{\scriptscriptstyle{K}}$ converge in distribution to some random variable $\overline{\xi}$ satisfying \[\Pr(\overline{\xi}\geq t)=\frac{\alpha-1}{\Gamma(2-\alpha)(1-\mu^{\alpha-1})}\int_t^\infty y^{-\alpha}(e^{-\mu y}-e^{-y})\d y.\]
 \begin{proof}
  We prove this by showing the convergence of 
  \begin{equation}\label{baynumhgt}
  \Pr\left(\xi^*-1\geq tb_n^{\scriptscriptstyle{K}}|\overline{H}=\overline{H}_n^{\scriptscriptstyle{K}}\right)=\Pr\left(\overline{H}=\overline{H}_n^{\scriptscriptstyle{K}}|\xi^*-1\geq tb_n^{\scriptscriptstyle{K}}\right)\frac{\Pr(\xi^*-1\geq tb_n^{\scriptscriptstyle{K}})}{\Pr(\overline{H}=\overline{H}_n^{\scriptscriptstyle{K}})} 
  \end{equation}
for all $t>0$. To begin we consider $\Pr\left(\overline{H}=\overline{H}_n^{\scriptscriptstyle{K}}|\xi^*-1\geq tb_n^{\scriptscriptstyle{K}}\right)$. 

The heights of individual traps are independent under this conditioning hence
\begin{flalign*}
 & \Pr(\overline{H}\leq\overline{H}_n^{\scriptscriptstyle{K}}|\xi^*-1\geq tb_n^{\scriptscriptstyle{K}}) = \Er[\Pr(\Hc(\Tc)\leq\overline{H}_n^{\scriptscriptstyle{K}})^{\xi^*-1}|\xi^*-1\geq tb_n^{\scriptscriptstyle{K}}].
\end{flalign*}
We know the asymptotic form of $\Pr(\Hc(\Tc)\leq\overline{H}_n^{\scriptscriptstyle{K}})$ from (\ref{cmu}) thus we need to consider the distribution of $\xi^*-1$ conditioned on $\xi^*-1\geq tb_n^{\scriptscriptstyle{K}}$. By the tail formula for $\xi^*-1$ following Definition \ref{infall} we have that for $r \geq 1$ as $\nin$
\begin{flalign*}
 \Pr\left(\frac{\xi^*-1}{tb_n^{\scriptscriptstyle{K}}} \geq r\Big| \xi^*-1 \geq tb_n^{\scriptscriptstyle{K}}\right) \; = \; \frac{\Pr(\xi^*-1 \geq rtb_n^{\scriptscriptstyle{K}})}{\Pr(\xi^*-1 \geq tb_n^{\scriptscriptstyle{K}})} \; \sim \; r^{-(\alpha-1)}.
\end{flalign*}

We therefore have that, conditional on $\xi^*-1\geq tb_n^{\scriptscriptstyle{K}}$, $(\xi^*-1)/tb_n^{\scriptscriptstyle{K}}$ converges in distribution to some variable $Y$ with tail $\Pr(Y\geq r)=r^{-(\alpha-1)} \land 1$.  Using the form of $b_n^{\scriptscriptstyle{K}}$ we then have that  \[\Pr(\Hc(\Tc)\leq \overline{H}_n^{\scriptscriptstyle{K}})^{tb_n^{\scriptscriptstyle{K}}} =e^{-t\mu(1+o(1))}.\]
It therefore follows that \[ \limn \Pr(\overline{H}\leq\overline{H}_n^{\scriptscriptstyle{K}}|\xi^*-1\geq tb_n^{\scriptscriptstyle{K}}) = \Er[e^{-t\mu Y}]. \] Repeating with $\overline{H}_n^{\scriptscriptstyle{K}}$ replaced by $\overline{H}_n^{\scriptscriptstyle{K}}-1$ we have that 
$\Pr(\overline{H}=\overline{H}_n^{\scriptscriptstyle{K}}|\xi^*-1\geq tb_n^{\scriptscriptstyle{K}}) \rightarrow \Er[e^{-t\mu Y}] - \Er[e^{-tY}]$ as $\nin$. For $\theta>0$
\begin{flalign*}
\Er[e^{-\theta tY}] =  (\alpha-1)t^{\alpha-1}\int_t^\infty e^{-\theta y}y^{-\alpha}\d y 
\end{flalign*}
therefore
\begin{equation}\label{hgteqcon}
\limn \Pr(\overline{H}=\overline{H}_n^{\scriptscriptstyle{K}}|\xi^*-1\geq tb_n^{\scriptscriptstyle{K}}) = (\alpha-1)t^{\alpha-1}\int_t^\infty y^{-\alpha}(e^{-\mu y}-e^{-y})\d y.  
\end{equation}

By (\ref{numcrttrp}) we have that as $\nin$
\begin{flalign*}
 \frac{\Pr(\xi^*-1\geq tb_n^{\scriptscriptstyle{K}})}{\Pr(\overline{H}=\overline{H}_n^{\scriptscriptstyle{K}})} \; \sim \; \frac{\Pr(\xi^*-1\geq tb_n^{\scriptscriptstyle{K}})}{\Gamma(2-\alpha)(1-\mu^{\alpha-1})c_\mu^{\alpha-1}\Pr(\xi^*-1\geq c_\mu b_n^{\scriptscriptstyle{K}})} \; \sim \; \frac{t^{-(\alpha-1)}}{\Gamma(2-\alpha)(1-\mu^{\alpha-1})}. 
\end{flalign*}
Combining this with (\ref{hgteqcon}) in (\ref{baynumhgt}) we have that \[ \limn\Pr(\xi^*-1\geq tb_n^{\scriptscriptstyle{K}}|\overline{H}=\overline{H}_n^{\scriptscriptstyle{K}})=\frac{\alpha-1}{\Gamma(2-\alpha)(1-\mu^{\alpha-1})}\int_t^\infty y^{-\alpha}(e^{-\mu y}-e^{-y})\d y.\]
\end{proof}
\end{lem}

Define  \[Z_\infty^n=\frac{1}{1-\beta^{-1}}\sum_{j=1}^{N}\beta^{H_j-\overline{H}}\Et[\Rc_\infty^j]\sum_{k=1}^{B_\infty^j}e_{j,k} \]
whose distribution depends on $n$ only through $N$ and $(H_j-\overline{H})_{j= 1}^{N}$. Recalling the definition of $\zeta^{(n)}$ in (\ref{zetan}), since $e_{j,k}$ are the exponential random variables defining $G^{j,k}$, $B_\infty^j \sim Bin(B^j,p_\infty/p_1(H_1))$ and the random variable $N$ is the same in both equations, we have that $\zeta^{(n)}$ and $Z_\infty^n$ are defined on the same probability space.  

\begin{prp}\label{conver}
 In IVIE, for any $K \in \Zb$ and $\delta>0$ \[\limn\Pb^{\scriptscriptstyle{K}}\left(|\zeta^{(n)}-Z_\infty^n|>\delta\right)=0.\]
 \begin{proof}
 Using the bounds on $A_{11}, A_{13}$ and $A_{14}$ from (\ref{bnm}) and Lemmas \ref{lln} and \ref{itriv} respectively there exists some function $g:\Rb\rightarrow \Rb$ such that $\lim_{\tilde{\varepsilon} \rightarrow 0^+}g(\tilde{\varepsilon})=0$ and for sufficiently large $n$ (independently of $K$)
\begin{flalign*}
\Pb^{\scriptscriptstyle{K}}\left(|\zeta^{(n)}-Z_\infty^n|>\delta\right) \leq o(1) + 2\Pb^{\scriptscriptstyle{K}}\left(g(\tilde{\varepsilon})Z_\infty^n>\delta\right). 
\end{flalign*}

It therefore suffices to show that $(Z_\infty^n)_{n\geq 0}$ are tight under $\Pb^K$. Write \[S_j=\frac{1}{1-\beta^{-1}}\Et[\Rc_\infty^j]\sum_{k=1}^{B_\infty^{j}}e_{j,k}.\]
$\Et[\Rc^j_\infty]$, $B_\infty^j$ and $e_{j,k}$ are independent, don't depend on $K$ and have finite expected value (by Lemma \ref{maxexc}, the geometric distribution of $W^j$ and exponential distribution of $e^{j,k}$) therefore $\Eb^{\scriptscriptstyle{K}}[S_j]\leq C<\infty$ uniformly over $K$. We can then write \[Z_\infty^n=\sum_{j=1}^{N}\beta^{H_j-\overline{H}_n^{\scriptscriptstyle{K}}}S_j.\]

Clearly, $N$ is dominated by the total number of traps in the branch thus by Lemma \ref{hgtnum} with high probability $Z_\infty^n$ can be stochastically dominated by \[\sum_{j=1}^{b_n^{\scriptscriptstyle{K}}(\overline{\xi}+1)}\beta^{H_j-\overline{H}_n^{\scriptscriptstyle{K}}}S_j.\]

Conditional on trap $j$ being the first in the branch which attains the maximum height we have that the heights of the remaining traps are independent and either at most the height of the largest or strictly shorter. Furthermore, the distribution of $S_j$ is independent of the height of the trap. Write $\Phi=\inf\{r\geq 1:H_r=\overline{H}_n^{\scriptscriptstyle{K}}\}$ then we have that 
 
\begin{flalign*}
 \Pb^{\scriptscriptstyle{K}}(Z_\infty^n\geq t)  &  \leq \Pb^{\scriptscriptstyle{K}}\left(\sum_{j=1}^{b_n^{\scriptscriptstyle{K}}(\overline{\xi}+1)}\beta^{H_j-\overline{H}_n^{\scriptscriptstyle{K}}}S_j \geq t\Big|\Phi=1\right) +o(1) \\
 & \leq \Pb(S_1\geq \log(t))+o(1) + \Pb\left(\sum_{j=2}^{b_n^{\scriptscriptstyle{K}}(\overline{\xi}+1)}\beta^{H_j-\overline{H}_n^{\scriptscriptstyle{K}}}S_j \geq t-\log(t)\Big| H_j\leq \overline{H}_n^{\scriptscriptstyle{K}} \; \forall j\geq 2 \right). 
\end{flalign*}

The distributions of $S_1, \overline{\xi}$ are independent of $n$ therefore $\lim_{t\rightarrow \infty}\Pb(S_1\geq \log(t))= 0$ and $\lim_{t\rightarrow \infty}\Pb(\overline{\xi}+1\geq \log(t))=0$ thus we can consider only the events in which $\overline{\xi}+1\leq \log(t)$. By Markov's inequality and independence we have that 

 \begin{flalign*}
 \Pb\left(\sum_{j=2}^{b_n^{\scriptscriptstyle{K}}(\overline{\xi}+1)}\beta^{H_j-\overline{H}_n^{\scriptscriptstyle{K}}}S_j \geq t-\log(t)\Big|\bigcap_{j\geq2} H_j\leq \overline{H}_n^{\scriptscriptstyle{K}} \right)& \leq \frac{b_n^{\scriptscriptstyle{K}}\log(t)\Eb[S_1]\Eb[\beta^{H_1}|H_1\leq \overline{H}_n^{\scriptscriptstyle{K}}]}{\beta^{\overline{H}_n^{\scriptscriptstyle{K}}}(t-\log(t))}+o(1).
  \end{flalign*}
  
  We have that $\Pb(H_1=l|H_1\leq \overline{H}_n^{\scriptscriptstyle{K}}) \leq \Pb(H_1\geq l|H_1\leq \overline{H}_n^{\scriptscriptstyle{K}}) \leq \Pb(H_1\geq l) \leq C\mu^l$ for some constant $C$ therefore the result follows from
  \begin{flalign*}
  \Eb[\beta^{H_1}|H_1\leq \overline{H}_n^{\scriptscriptstyle{K}}] \; = \; \sum_{l=0}^{\overline{H}_n^{\scriptscriptstyle{K}}}\beta^l \Pb(H_1=l|H_1\leq \overline{H}_n^{\scriptscriptstyle{K}}) \; \leq \; C (\beta\mu)^{\overline{H}_n^{\scriptscriptstyle{K}}}.
 \end{flalign*}
 \end{proof}
\end{prp}

The next proposition shows that, under $\Pb^{\scriptscriptstyle{K}}$, we have that $\zeta^{(n)}$ converge in distribution along certain subsequences. 

\begin{prp}\label{nkconv}
 In IVIE, under $\Pb^{\scriptscriptstyle{K}}$ we have that $Z_\infty^{n_k}$ converges in distribution (as $k \rightarrow \infty$) to some random variable $Z_\infty$.
  \begin{proof}
We begin by showing that it suffices to replace $N$ with $\xi^*-1$ (i.e. the total number of traps in the branch). This will simplify matters by removing the condition that the traps we consider are of at least some height which varies with $n$ and also allows us to use $\xi^*-1$ under $\Pr^{\scriptscriptstyle{K}}$ which we understand by Lemma \ref{hgtnum}. Fix $\tilde{\varepsilon}>0$ and let $H_j$ be ordered such that $H_j\geq H_{j+1}$ for all $j$ then we want to show that \[\Pb^{\scriptscriptstyle{K}}\left(\sum_{j=N+1}^{\xi^*-1}\beta^{H_j-\overline{H}_n^{\scriptscriptstyle{K}}}S_j>\tilde{\varepsilon} \right) \rightarrow 0\]
  as $\nin$. By Lemma \ref{hgtnum}, for any $\delta>0$, we have that $\Pb^{\scriptscriptstyle{K}}(\xi^*-1\geq a_{n^{1+\delta}})\rightarrow 0$ as $\nin$. We therefore have that 
  \begin{flalign*}
   \Pb^{\scriptscriptstyle{K}}\left(\sum_{j=N+1}^{\xi^*-1}\beta^{H_j-\overline{H}_n^{\scriptscriptstyle{K}}}S_j>\tilde{\varepsilon} \right) & \leq \Pb\left(\sum_{j=1}^{a_{n^{1+\delta}}}\beta^{H_j-\overline{H}_n^{\scriptscriptstyle{K}}}S_j>\tilde{\varepsilon}\Big|\bigcap_{j=1}^{a_{n^{1+\delta}}}\left\{H_j\leq  h_n^\varepsilon\right\} \right)+o(1).
  \end{flalign*}
  By Markov's inequality we then have that 
  \begin{flalign*}
 \Pb\left(\sum_{j=1}^{a_{n^{1+\delta}}}\beta^{H_j-\overline{H}_n^{\scriptscriptstyle{K}}}S_j>\tilde{\varepsilon}\Big|\bigcap_{j=1}^{a_{n^{1+\delta}}}\left\{H_j\leq  h_n^\varepsilon\right\} \right) \; \leq \; \frac{a_{n^{1+\delta}}\Eb[S_j]\Eb[\beta^H|H\leq  h_n^\varepsilon]}{\beta^{\overline{H}_n^{\scriptscriptstyle{K}}}} \; \leq \; \frac{C_{\scriptscriptstyle{K}}a_{n^{1+\delta}}(\beta\mu)^{ h_n^\varepsilon}}{a_n^{1/\gamma}} 
  \end{flalign*}
 Rearranging the terms in the final expression, we see that choosing $\delta<\varepsilon\frac{\log(\beta\mu)}{\log(\mu^{-1})}$ ensures that this indeed converges to $0$ for any $K \in \Zb$. 
  
  It now suffices to show that \[\sum_{j=1}^{\xi^*-1}\beta^{H_j-\overline{H}_n^{\scriptscriptstyle{K}}}S_j\] converges in distribution under $\Pb^{\scriptscriptstyle{K}}$ along the given subsequences. We do this by considering a generating function approach. Recall that $\Phi=\inf\{r\geq 1:H_r=\overline{H}_n^{\scriptscriptstyle{K}}\}$ is the index of the first trap of the maximum height. Let $H\ed \Hc(\Tc)$ have the distribution of the height of an $f$-GW tree and $S\ed S_1$. Writing
  \begin{flalign*}
 \psi_i(h,\lambda) \;&  = \; \Eb[e^{-\lambda S\beta^{H-h}}|H\leq h+1-i] \; = \; \frac{\Eb[e^{-\lambda S\beta^{H-h}}\ind_{\{H\leq h+1-i\}}]}{\Pb(H\leq h+1-i)}, \\
  \phi_i(h,\lambda) \;&  = \; \Eb[e^{-\lambda S\beta^{H-h}}\ind_{\{H\leq h+1-i\}}]
\end{flalign*}
for $i=1,2$, gives us that by independence of the height of the traps conditioned on $\Phi$
  \begin{flalign}
   \varphi_{\scriptscriptstyle{K}}(\lambda) & :=\Eb^{\scriptscriptstyle{K}}\left[e^{-\lambda\sum_{j=1}^{\xi^*-1}\beta^{H_j-\overline{H}_n^{\scriptscriptstyle{K}}}S_j}\right] \notag \\
   & =\Eb^{\scriptscriptstyle{K}}\left[\sum_{k=1}^{\xi^*-1}\Pr^{\scriptscriptstyle{K}}(\Phi=k|\xi^*)\Eb^{\scriptscriptstyle{K}}\left[e^{-\lambda\sum_{j=1}^{\xi^*-1}\beta^{H_j-\overline{H}_n^{\scriptscriptstyle{K}}}S_j}\Big| \Phi=k,\xi^*\right]\right] \notag \\
   & = \Eb^{\scriptscriptstyle{K}}\left[\Eb[e^{-\lambda S}]\sum_{k=1}^{\xi^*-1}\Pr^{\scriptscriptstyle{K}}(\Phi=k|\xi^*) \psi_2(\overline{H}_n^{\scriptscriptstyle{K}},\lambda)^{k-1}\psi_1(\overline{H}_n^{\scriptscriptstyle{K}},\lambda)^{\xi^*-1-k}\right]. \label{crft}
  \end{flalign}
Using that
\begin{flalign*}
 \Pr^{\scriptscriptstyle{K}}(\Phi=k|\xi^*) &=  \frac{\Pr(H=\overline{H}_n^{\scriptscriptstyle{K}})}{\Pr(\overline{H}=\overline{H}_n^{\scriptscriptstyle{K}}|\xi^*)}\Pr(H\leq\overline{H}_n^{\scriptscriptstyle{K}}-1)^{k-1}\Pr(H\leq\overline{H}_n^{\scriptscriptstyle{K}})^{\xi^*-1-k}, \\
 \Pr(H \leq h+1-i)\psi_i(h,\lambda) & =\phi_i(h,\lambda)
 \end{flalign*}
we can write (\ref{crft}) as
\begin{flalign*}
&  \Eb^{\scriptscriptstyle{K}}\left[\frac{\Eb[e^{-\lambda S}]\Pr(H=\overline{H}_n^{\scriptscriptstyle{K}})}{\Pr(\overline{H}=\overline{H}_n^{\scriptscriptstyle{K}}|\xi^*)}\phi_1(\overline{H}_n^{\scriptscriptstyle{K}},\lambda)^{\xi^*-2}\sum_{k=1}^{\xi^*-1} \left(\frac{\phi_2(\overline{H}_n^{\scriptscriptstyle{K}},\lambda)}{\phi_1(\overline{H}_n^{\scriptscriptstyle{K}},\lambda)}\right)^{k-1}\right] \\
& \qquad = \Eb^{\scriptscriptstyle{K}}\left[
\frac{\Eb[e^{-\lambda S}]\Pr(H=\overline{H}_n^{\scriptscriptstyle{K}})}{\Pr(\overline{H}=\overline{H}_n^{\scriptscriptstyle{K}}|\xi^*)}
\left(\frac{\phi_1(\overline{H}_n^{\scriptscriptstyle{K}},\lambda)^{\xi^*-1}-\phi_2(\overline{H}_n^{\scriptscriptstyle{K}},\lambda)^{\xi^*-1}}
{\phi_1(\overline{H}_n^{\scriptscriptstyle{K}},\lambda)-\phi_2(\overline{H}_n^{\scriptscriptstyle{K}},\lambda)}\right)
\right] \\
& \qquad  =\Eb^{\scriptscriptstyle{K}}\left[\frac{\Eb[e^{-\lambda S\beta^{H-\overline{H}_n^{\scriptscriptstyle{K}}}}\ind_{\{H\leq \overline{H}_n^{\scriptscriptstyle{K}}\}}]^{\xi^*-1}-\Eb[e^{-\lambda S\beta^{H-\overline{H}_n^{\scriptscriptstyle{K}}}}\ind_{\{H\leq \overline{H}_n^{\scriptscriptstyle{K}}-1\}}]^{\xi^*-1}}{\Pr(H\leq \overline{H}_n^{\scriptscriptstyle{K}})^{\xi^*-1}-\Pr(H\leq \overline{H}_n^{\scriptscriptstyle{K}}-1)^{\xi^*-1}}\right]
\end{flalign*}
where the final equality comes from 
 \begin{flalign*}
 \phi_1(\overline{H}_n^{\scriptscriptstyle{K}},\lambda)-\phi_2(\overline{H}_n^{\scriptscriptstyle{K}},\lambda) & = \Eb[e^{-\lambda S}]\Pr(H=\overline{H}_n^{\scriptscriptstyle{K}}), \\
 \Pr(\overline{H}=\overline{H}_n^{\scriptscriptstyle{K}}|\xi^*) & = \Pr(H\leq \overline{H}_n^{\scriptscriptstyle{K}})^{\xi^*-1}-\Pr(H\leq \overline{H}_n^{\scriptscriptstyle{K}}-1)^{\xi^*-1}.
\end{flalign*}
We want to study
\begin{flalign*}
 \Eb\left[e^{-\lambda S\beta^{H-\overline{H}_n^{\scriptscriptstyle{K}}}}\ind_{\{H \leq \overline{H}_n^{\scriptscriptstyle{K}}\}}\right]^{\xi^*-1} & =  \Eb\left[e^{-\lambda S\beta^{H-\overline{H}_n^{\scriptscriptstyle{K}}}}\right]^{\xi^*-1}\left(1-\frac{ \Eb\left[e^{-\lambda S\beta^{H-\overline{H}_n^{\scriptscriptstyle{K}}}}\ind_{\{H > \overline{H}_n^{\scriptscriptstyle{K}}\}}\right]}{ \Eb\left[e^{-\lambda S\beta^{H-\overline{H}_n^{\scriptscriptstyle{K}}}}\right]}\right)^{\xi^*-1}.
\end{flalign*}
Using that for any $\tilde{\varepsilon}>0$ then we can find $N$ sufficiently large such that for all $h\geq N$ 
\begin{flalign}\label{hgtbnd}
1-(1+\tilde{\varepsilon})c_\mu \mu^h \leq \Pr(H\leq h) \leq 1-(1-\tilde{\varepsilon})c_\mu \mu^h
\end{flalign}
we obtain that $\Pr(H\leq \overline{H}_n^{\scriptscriptstyle{K}})^{\xi^*-1}-\Pr(H\leq \overline{H}_n^{\scriptscriptstyle{K}}-1)^{\xi^*-1}\rightarrow e^{-c_\mu \overline{\xi}}-e^{-c_\mu  \mu^{-1}\overline{\xi}}$ as $\nin$ under $\Pb^{\scriptscriptstyle{K}}$. Let $\delta_{\tilde{\varepsilon}}=1-\mu+\tilde{\varepsilon}(1+\mu)$ then using (\ref{hgtbnd}) we have that $c_\mu \mu^h \delta_{-\tilde{\varepsilon}}  \leq \Pb(H=h) \leq c_\mu \mu^h \delta_{\tilde{\varepsilon}}$ for sufficiently large $h \in \Nb$ and $\delta_{\tilde{\varepsilon}}-\delta_{-\tilde{\varepsilon}}=2\tilde{\varepsilon}(1+\mu)$ can be chosen arbitrarily small. In particular,
\begin{flalign*}
 \Eb\left[e^{-\lambda S\beta^{H-h}}\ind_{\{H > h\}}\right]  \; \leq \; \Eb\left[\frac{\delta_{\tilde{\varepsilon}}c_\mu}{1-\mu} \mu^{h+1} \sum_{k=0}^\infty e^{-(\lambda \beta) S\beta^k}\mu^k(1-\mu)\right] \; = \; \frac{\delta_{\tilde{\varepsilon}}c_\mu}{1-\mu} \mu^{h+1} \Eb\left[e^{-(\lambda \beta) S\beta^G}\right]
\end{flalign*}
where $G\sim Geo(\mu)$ independently of everything else. A similar lower bound yields that for $i=0,1$
  \[ \Eb\left[e^{-\lambda S\beta^{H-\overline{H}_n^{\scriptscriptstyle{K}}}}\ind_{\{H > \overline{H}_n^{\scriptscriptstyle{K}}-i\}}\right] \sim c_\mu\mu^{\overline{H}_n^{\scriptscriptstyle{K}}+1-i} \Eb\left[e^{-(\lambda \beta) S\beta^{G-i}}\right].\]
$\Eb\left[e^{-\lambda S\beta^{H-\overline{H}_n^{\scriptscriptstyle{K}}}}\right]$ converges to $1$ $\Pb^{\scriptscriptstyle{K}}$-a.s.\ therefore have that under $\Pb^{\scriptscriptstyle{K}}$, for $i=0,1$,
\begin{flalign*}
 \left(1-\frac{ \Eb\left[e^{-\lambda S\beta^{H-\overline{H}_n^{\scriptscriptstyle{K}}}}\ind_{\{H > \overline{H}_n^{\scriptscriptstyle{K}}-i\}}\right]}{ \Eb\left[e^{-\lambda S\beta^{H-\overline{H}_n^{\scriptscriptstyle{K}}}}\right]}\right)^{\xi^*-1} & \cd e^{-c_\mu \mu^{1-i} \varphi^{SG}(\lambda \beta)\overline{\xi}} 
\end{flalign*}
where $\varphi^{SG}$ is the moment generating function of $S\beta^G$ thus by Lemma \ref{hgtnum} we have that \[ \varphi_{\scriptscriptstyle{K}}(\lambda) \sim \Eb\left[\frac{\Eb[e^{-\lambda S\beta^{H-\overline{H}_n^{\scriptscriptstyle{K}}}}]^{b_n^{\scriptscriptstyle{K}} \overline{\xi}}\left(e^{-c_\mu\mu\varphi^{SG}(\lambda\beta)\overline{\xi}}-e^{-c_\mu\varphi^{SG}(\lambda)\overline{\xi}}\right)}{e^{-c_\mu \overline{\xi}}-e^{-c_\mu \mu^{-1} \overline{\xi}}}\right].\]

The only part of this equation which depends on $n$ is $\Eb[e^{-\lambda S\beta^{H-\overline{H}_n^{\scriptscriptstyle{K}}}}]^{b_n^{\scriptscriptstyle{K}}}$ thus it remains to determine how this behaves asymptotically. We start by showing that it suffices to replace $H$ with the geometric random variable $G$. Since $\beta^{\overline{H}_n^{\scriptscriptstyle{K}}}=(b_n^{\scriptscriptstyle{K}} c_\mu)^\frac{1}{\gamma}$, letting $\theta = \lambda c_\mu^{-\frac{1}{\gamma}}$, we have $\Eb[e^{-\lambda S\beta^{H-\overline{H}_n^{\scriptscriptstyle{K}}}}]^{b_n^{\scriptscriptstyle{K}}}  =\Eb[e^{-\theta (b_n^{\scriptscriptstyle{K}})^{-\frac{1}{\gamma}}S\beta^H}]^{b_n^{\scriptscriptstyle{K}}}$. We then have that
\begin{flalign*}
 \Eb[e^{-\theta b^{-\frac{1}{\gamma}}S\beta^H}] & = 1-\int_0^\infty e^{-x} \Pb(S\beta^H\geq xb^\frac{1}{\gamma}/\theta) \d x.
\end{flalign*}
Since $\gamma<1$ we have $\Eb\left[S^\gamma\right] <\infty$ therefore by independence of $S$ and $G$
\begin{flalign}\label{upSbetG}
  \Pb(S\beta^G\geq xb^\frac{1}{\gamma}/\theta) \; = \Eb\left[\Pb\left(G\geq \frac{\log(xb^{\frac{1}{\gamma}}(S\theta)^{-1})}{\log(\beta)}\Big|S\right)\right] \; \leq \; \left(\frac{xb^\frac{1}{\gamma}}{\theta}\right)^{-\gamma} \Eb\left[S^\gamma\right]  \; = \; C_\theta x^{-\gamma}b^{-1}.
\end{flalign}
Similarly since $c_1 \Pb(G\geq y)\leq \Pb(H\geq y) \leq c_2\Pb(G\geq y)$ for some positive constants $c_1,c_2$ we can choose $C_\theta$ large enough such that $\Pb(S\beta^H\geq xb^\frac{1}{\gamma}/\theta) \leq C_\theta x^{-\gamma}b^{-1}$.
Let $\tilde{\varepsilon}>0$ then choose $\delta>0$ such that 
\begin{flalign*}
\max\{1,c_\mu\} \int_0^\delta e^{-x}\Pb(S\beta^H\geq xb^\frac{1}{\gamma}/\theta) \d x \; \leq \; C b^{-1}\int_0^\delta e^{-x}x^{-\gamma}\d x \;\leq\; \tilde{\varepsilon}b^{-1}
\end{flalign*}
then, since these integrals are positive, we have that
\begin{flalign}\label{intdel}
\left|\int_0^\delta e^{-x}c_\mu\Pb(S\beta^G\geq xb^\frac{1}{\gamma}/\theta) \d x-\int_0^\delta e^{-x}\Pb(S\beta^H\geq xb^\frac{1}{\gamma}/\theta) \d x\right|\leq \tilde{\varepsilon}b^{-1}.
\end{flalign}
Let $M:\Rb\rightarrow \Rb$ such that $M(b) \rightarrow \infty$ as $b \rightarrow \infty$ and $M(b) \ll b^\frac{1}{\gamma}$. Using (\ref{hgtbnd}) and that $M(b)\ll b^{\frac{1}{\gamma}}$ we can choose $b$ sufficiently large such that for all $x>\delta, \; y<M(b)$ we have that 
\[\left|\frac{c_\mu\Pb\left(\beta^G \geq \frac{xb^\frac{1}{\gamma}}{\theta y}\right)}{\Pb\left(\beta^H \geq \frac{xb^\frac{1}{\gamma}}{\theta y}\right)}-1\right| \leq \tilde{\varepsilon}\]
which therefore gives us that
\begin{flalign}\label{GHrat}
\left|\frac{c_\mu\Pb\left(S\beta^G \geq \frac{xb^\frac{1}{\gamma}}{\theta}\Big| S<M(b)\right)}{\Pb\left(S\beta^H \geq \frac{xb^\frac{1}{\gamma}}{\theta}\Big| S<M(b)\right)}-1\right| \leq \tilde{\varepsilon}.
\end{flalign}
Then, for $x>\delta$, using the tail of $H$ we have
\begin{flalign*}
 \Pb(S\beta^H\geq xb^\frac{1}{\gamma}/\theta|S \geq M(b)) \; \leq \; C \Eb\left[\left(\frac{xb^\frac{1}{\gamma}}{S\theta}\right)^\frac{\log(\mu)}{\log(\beta)}\Big|S\geq M(b)\right] \; \leq \; \frac{C_{\delta,\theta}}{b} \Eb[S^{\gamma}|S\geq M(b)].
\end{flalign*}
Writing $\eta(b)=\Eb[S^{\gamma}\ind_{\{S\geq M(b)\}}]$ we have that $\eta(b) \rightarrow 0$ as $b \rightarrow \infty$ since $\Eb[S^{\gamma}]<\infty$ therefore repeating the argument for $G$ we have that both $\Pb(S\beta^H\geq xb^\frac{1}{\gamma}\theta^{-1}, \; S \geq M(b))$ and $\Pb(S\beta^G\geq xb^\frac{1}{\gamma}\theta^{-1}, \;S \geq M(b))$ can be bounded above by $C_{\delta,\theta}b^{-1}\eta(b)$. Combining this with (\ref{upSbetG}), (\ref{intdel}) and (\ref{GHrat}) it is straightforward to see that as $b \rightarrow \infty$  
\begin{flalign*}
 \Eb[e^{-\theta b^{-\frac{1}{\gamma}}S\beta^H}]^b \sim \left(1-c_\mu \int_0^\infty e^{-y} \Pb(S\beta^G \geq yb^\frac{1}{\gamma} /\theta )\d y\right)^b.
\end{flalign*}

Writing $f(x)=\left\lceil \frac{\log(x)}{\log(\beta)}\right\rceil -\frac{\log(x)}{\log(\beta)}$  we have that for any $m \in \Zb$ that $f(x)=f(xm^{\log(\beta)})$ for all $x \in \Rb$ and 
\begin{flalign*}
 \Pb(S\beta^G\geq x) & = \mu^\frac{\log(x)}{\log(\beta)} \Eb\left[S^{\gamma}\mu^{-\left\lfloor \frac{\log(S)}{\log(\beta)} +f(x)\right\rfloor + \frac{\log(S)}{\log(\beta)} +f(x) }\right].
\end{flalign*}
Define \[I(x)=\Eb\left[S^{\gamma}\mu^{-\left\lfloor \frac{\log(S)}{\log(\beta)} +f(x)\right\rfloor + \frac{\log(S)}{\log(\beta)} +f(x) }\right]\]
then $\Pb(S\beta^G\geq x)=x^{-\gamma}I(x)$ and $I(x)=I(xm^{\log(\beta)})$ for all $x \in \Rb$ and $m \in \Zb$. We then have that 
\begin{flalign*}
 \Eb[e^{-\lambda b^{-\frac{1}{\gamma}}S\beta^H}]^b \; \sim \; \left(1-\lambda^{\gamma}b^{-1} \int_0^\infty e^{-y} y^{-\gamma}I(yb^\frac{1}{\gamma}/\theta)\d y\right)^b \;  \sim \; e^{-\lambda^{\gamma}\int_0^\infty e^{-y} y^{-\gamma}I(yb^\frac{1}{\gamma}/\theta) \d y}. 
\end{flalign*}
This expression is constant along the given subsequences which proves the proposition.
  \end{proof}
\end{prp}

In order to prove the convergence result for sums of i.i.d.\ variables we shall require that $\zeta^{(n)}$ can be dominated (independently of $K\geq  h_n^\varepsilon- h_n^0$) by some random variable $Z_{sup}$ such that $\Eb[Z_{sup}^{(\alpha-1)\gamma+\epsilon}]<\infty$ for $\epsilon$ sufficiently small. Lemma \ref{dominate} shows that we indeed have the domination required for the sums of i.i.d.\ variables result.

\begin{lem}\label{dominate}
In IVIE, there exists a random variable $Z_{sup}$ such that under $\Pb^{\scriptscriptstyle{K}}$ for any $K \in \Zb$ we have that $Z_{sup} \succeq \zeta^{(n)}$ for all $n$ sufficiently large and $\Eb[Z_{sup}^{1-\epsilon}]<\infty$ for any $\epsilon>0$.
 \begin{proof}
 $N$ is dominated by the number of traps in the branch. Similarly to Lemma \ref{hgtnum} we consider 
   \begin{equation*}
  \Pr(\xi^*-1\geq tb_n^{\scriptscriptstyle{K}}|\overline{H}=\overline{H}_n^{\scriptscriptstyle{K}})=\Pr(\overline{H}=\overline{H}_n^{\scriptscriptstyle{K}}|\xi^*-1\geq tb_n^{\scriptscriptstyle{K}})\frac{\Pr(\xi^*-1\geq tb_n^{\scriptscriptstyle{K}})}{\Pr(\overline{H}=\overline{H}_n^{\scriptscriptstyle{K}})}. 
  \end{equation*}
Using the tail of $H$ from (\ref{cmu}), for large $n$ (independently of $t\geq 0$) and some constant $c$, we can bound $\Pr(\overline{H}=\overline{H}_n^{\scriptscriptstyle{K}}|\xi^*-1\geq tb_n^{\scriptscriptstyle{K}})$ above by 
\begin{flalign*}
 \Er\left[\Pr(H\leq\overline{H}_n^{\scriptscriptstyle{K}})^{\xi^*-1} \Big|\xi^*-1\geq tb_n^{\scriptscriptstyle{K}}\right] \; \leq \; \Er\left[e^{-c\left(\frac{\xi^*-1}{b_n^{\scriptscriptstyle{K}}}\right)} \Big|\xi^*-1\geq tb_n^{\scriptscriptstyle{K}}\right] \; \leq \; e^{-ct}. 
\end{flalign*}

For each $t\geq 0$ we have that $\Pr(\xi^*-1\geq tb_n^{\scriptscriptstyle{K}})\sim  Ct^{-(\alpha-1)}\Pr(\overline{H}=\overline{H}_n^{\scriptscriptstyle{K}})$ as $\nin$. Since $\Pr(\overline{H}=\overline{H}_n^{\scriptscriptstyle{K}}) $ doesn't depend on $t$ we can choose a constant $c$ such that for $n$ sufficiently large we have that $\Pr(\overline{H}=\overline{H}_n^{\scriptscriptstyle{K}}) \leq c\Pr(\xi^*-1\geq b_n^{\scriptscriptstyle{K}})$ thus for $t\geq 1$
\begin{flalign*}
 \frac{\Pr(\xi^*-1\geq tb_n^{\scriptscriptstyle{K}})}{\Pr(\overline{H}=\overline{H}_n^{\scriptscriptstyle{K}})}  \leq \frac{\Pr(\xi^*-1\geq tb_n^{\scriptscriptstyle{K}})}{c\Pr(\xi^*-1\geq b_n^{\scriptscriptstyle{K}})}  \leq c^{-1}.
\end{flalign*}
In particular, for $t\geq 1$ we have that $\Pr(\xi^*-1\geq tb_n^{\scriptscriptstyle{K}}|\overline{H}=\overline{H}_n^{\scriptscriptstyle{K}}) \leq c_1e^{-c_2t}$ for some constants $c_1,c_2$. It therefore follows that there exists some random variable $\overline{\overline{\xi}}$ which is independent of $\overline{H}$, has an exponential tail and $\overline{\overline{\xi}}b_n^{\scriptscriptstyle{K}}\geq \xi^*-1$ on $\overline{H}=\overline{H}_n^{\scriptscriptstyle{K}}$ for $n$ suitably large (independently of $K$). Let $\Rc_\infty^{j,k,l}$ be distributed as excursions from the deepest points of the infinite trap $\Tc^-$ then using that $W^j\geq B^j$ we then have that for $n$ suitably large, under $\Pb^{\scriptscriptstyle{K}}$ 
\[ \zeta^{(n)} \preceq \sum_{j=1}^{\overline{\overline{\xi}}b_n^{\scriptscriptstyle{K}}}\sum_{k=1}^{W^{j}}\sum_{l=1}^{G^{j,k}}\frac{\Rc_\infty^{j,k,l}}{\beta^{\overline{H}_n^{\scriptscriptstyle{K}}}}.  \]

Since $ \Et[G^{j,k}] \leq \beta^{H_j+1}/(\beta+1)$ there is some constant $c$ such that, writing \[Y_j^{(n)}= c\sum_{k=1}^{W^j}\sum_{l=1}^{G^{j,k}}\frac{\Rc_\infty^{j,k,l}}{\Et[G^{j,k}]} \] 
which are identically distributed under $\Pb$, we have that under $\Pb^{\scriptscriptstyle{K}}$, \[\zeta^{(n)} \preceq \frac{1}{\beta^{\overline{H}_n^{\scriptscriptstyle{K}}}}\sum_{j=1}^{\overline{\overline{\xi}}b_n^{\scriptscriptstyle{K}}}\beta^{H_j}Y_j^{(n)}.\]

For $m\geq 1$ write $X^n(m)=\frac{1}{m}\sum_{j=1}^m \beta^{H_j}Y_j^{(n)}\ind_{\{j\neq \Phi\}}$ (where we recall that $\Phi$ is the first index $j$ such that $H_j=\overline{H}_n^{\scriptscriptstyle{K}}$) then by Markov's inequality
\begin{flalign*}
 \Pb^{\scriptscriptstyle{K}}(X^n(m)\geq t) \; \leq \; \frac{1}{m}\sum_{j=1}^m\frac{\Eb^{\scriptscriptstyle{K}}\left[\beta^{H_j}Y_j^{(n)}\ind_{\{j\neq\Phi\}}\right]}{t}\; = \; \frac{1}{m}\sum_{j=1}^m\frac{\Eb^{\scriptscriptstyle{K}}\left[\beta^{H_j}\ind_{\{j\neq\Phi\}}\right]\Eb^{\scriptscriptstyle{K}}\left[Y_j^{(n)}\right]}{t} 
\end{flalign*}
since $\Eb[Y_j^{(n)}|H_j,  \Phi]$ is independent of $H_j$ and $\Phi$. Since $W^1$ has a geometric distribution (independently of $n$) we have that $\Eb[W^1]<\infty$ and by Lemma \ref{maxexc} we have that $\Eb[\Rc_\infty]<\infty$ therefore $\Eb^{\scriptscriptstyle{K}}[Y_j^{(n)}]\leq \Eb[W^1]\Eb[\Rc_\infty]\leq C<\infty$ for all $n$. Using geometric bounds on the tail of $H$ from (\ref{cmu}) and that $\Pb(H\geq j|H \leq \overline{H}_n^{\scriptscriptstyle{K}})\leq \Pb(H\geq j)$ we have that
\begin{flalign*}
 \Eb^{\scriptscriptstyle{K}}[\beta^{H}] \; \leq \; \sum_{j=0}^{\infty} \beta^j\Pb(H\geq j|H \leq \overline{H}_n^{\scriptscriptstyle{K}}) \; \leq \; C(\beta\mu)^{\overline{H}_n^{\scriptscriptstyle{K}}}.
\end{flalign*}
We therefore have that $\Pb^{\scriptscriptstyle{K}}(X^n(m)\geq t) \leq C(\beta\mu)^{\overline{H}_n^{\scriptscriptstyle{K}}}/t$  thus there exists some sequence of random variables $X_{sup}^n \succeq X^n(m)$ for any $m$ such that $\Pb^{\scriptscriptstyle{K}}(X_{sup}^n\geq t)=1 \land  C(\beta\mu)^{\overline{H}_n^{\scriptscriptstyle{K}}}t^{-1}$. In particular, $X_{sup}^n \succeq X^n(\overline{\overline{\xi}}b_n^{\scriptscriptstyle{K}})$. Therefore, 
\[\frac{1}{\beta^{\overline{H}_n^{\scriptscriptstyle{K}}}}\sum_{j=1}^{\overline{\overline{\xi}}b_n^{\scriptscriptstyle{K}}}\beta^{H_j}Y_j^{(n)}= \frac{\overline{\overline{\xi}}b_n^{\scriptscriptstyle{K}}}{\beta^{\overline{H}_n^{\scriptscriptstyle{K}}}}X^n(\overline{\overline{\xi}}b_n^{\scriptscriptstyle{K}})+Y_\Phi^{(n)} \preceq \frac{\overline{\overline{\xi}}X_{sup}^n}{c_\mu(\beta\mu)^{\overline{H}_n^{\scriptscriptstyle{K}}}}+Y_\Phi^{(n)} \]
under $\Pb^{\scriptscriptstyle{K}}$. We then have that 
\begin{flalign*}
 \Pb^{\scriptscriptstyle{K}}\left(\frac{\overline{\overline{\xi}}X_{sup}^n}{c_\mu(\beta\mu)^{\overline{H}_n^{\scriptscriptstyle{K}}}}\geq t\right) & =  \Eb^{\scriptscriptstyle{K}}\left[\Pb^{\scriptscriptstyle{K}}\left(X_{sup}^n\geq \frac{tc_\mu(\beta\mu)^{\overline{H}_n^{\scriptscriptstyle{K}}}}{\overline{\overline{\xi}}}\Big| \overline{\overline{\xi}}\right)\right]=1 \land  C\frac{\Eb^{\scriptscriptstyle{K}}\left[\overline{\overline{\xi}}\right]}{t}
\end{flalign*}
where $\overline{\overline{\xi}}$ has finite first moment since $\Pr(\overline{\overline{\xi}}\geq t)=c_1e^{-c_2t} \land 1$.

It follows that there exists $X_{sup}\succeq X_{sup}^n$ for any $n$ such that $\Pb(X_{sup}\geq t)=1\land Ct^{-1}$. Since $\Eb_{\scriptscriptstyle{K}}[Y_\Phi^n]$ is bounded independently of $K$ and $n$, by Markov's inequality we have that there exists $Y_{sup}\succeq Y_\Phi^n$ for all $n$ such that $\Pb(Y_{sup}\geq t)=1\land Ct^{-1}$. It therefore follows that $\zeta^{(n)}$ under $\Pb^{\scriptscriptstyle{K}}$ is stochastically dominated by $X_{sup}+Y_{sup}$ under $\Pb$ where
\begin{flalign*}
 \Pb(X_{sup}+Y_{sup}\geq t) \; \leq \; \Pb(X_{sup}\geq t/2)+ \Pb(Y_{sup}\geq t/2) \; \leq \; Ct^{-1}
\end{flalign*}
hence $X_{sup}+Y_{sup}$ has finite moments up to $1-\epsilon$ for all $\epsilon>0$.
 \end{proof}
\end{lem}

\section{Convergence along subsequence}\label{convsub}
In this section we prove the main theorems concerning convergence to infinitely divisible laws in FVIE and IVIE. Both cases follow the proof from \cite{arfrgaha}; in FVIE the result follows directly whereas in IVIE adjustments need to be made to deal with slowly varying functions. 

\subsection{Proof of Theorem \ref{finvarthm} (FVIE)}

Recall that in FVIE $\gamma=\log(\mu^{-1})/\log(\beta)<1$, $n_l(t)=\lfloor t\mu^{-l}\rfloor$ and by Corollary \ref{hgttl} we have that $\Pr(\Hc(\Tc^{*-}_\rho)\geq n) \sim C_\Dc\mu^n = C_\Dc\beta^{-n\gamma}$ where $C_\Dc=c_\mu\Er[\xi^*-1]$. For $i,l\geq 1$ let $\zeta_i^l=\tilde{\chi}_{n_l}^{i*}\beta^{-\Hc(\Tc_{\rho_i})}$ under $\Hc(\Tc_{\rho_i}) \geq h_{n_l}^\varepsilon$ where $(\tilde{\chi}_{n_l}^{i*})_{i\geq 1}$ are i.i.d.\ with the law of $\chi_{n_l}^{1*}$ from (\ref{chistar}) and $(\Tc_{\rho_i})_{i\geq 1}$ are the associated trees. Then for $K \geq -(l-h_{n_l}^\varepsilon)$ let $\zeta_i^{l,K}$ to be $\zeta_i^l$ under $\Hc(\Tc_{\rho_i})=l+K$ when this makes sense and $0$ otherwise. For $K \in \Zb$ and $l\geq 0$ define $F_{\scriptscriptstyle{K}}^l(x)=\Pr(\zeta_i^{l,K}>x)$. By a simple adaptation of Corollary \ref{conver} and Lemma \ref{dominate}
\begin{enumerate}
 \item\label{Zinfty} $\exists Z_i^\infty$ random variables such that for all $K \in \Zb$ we have that $\zeta_i^{l,K} \cd Z_i^\infty$ as $l \rightarrow \infty$;
 \item\label{Zsup} $\exists Z_{sup}$ random variable such that for all $l \geq 0$ and $K \geq -(l-h_{n_l}^\varepsilon)$ we have that $\zeta_i^{l,K}\preceq Z_{sup}$ and $\Er[Z_{sup}^{\gamma+\delta}]<\infty$ for some $\delta>0$.
\end{enumerate}
More specifically, since $N=1$ with high probability in FVIE \[Z_\infty^n=\frac{1}{1-\beta^{-1}}\Et[\Rc_\infty]\sum_{k=1}^{B_\infty}e_k\]
for some binomial variable $B_\infty$ and independent exponential variables $e_k$. These are independent of $n$, hence an adaptation of Proposition \ref{conver} shows that $\zeta^{(n)}$ converge in distribution.

Set  \[ S_N^l =  \sum_{i=1}^N \tilde{\chi}_{n_l}^{i*} \quad \Big| \Hc(\Tc_{\rho_i})\geq h_{n_l}^\varepsilon \quad \forall i=1,...,N.\]
For $(\lambda_l)_{l\geq 0}$ converging to $\lambda>0$ define $M_l^\lambda=\lfloor \lambda_l^\gamma \beta^{\gamma(l-h_{n_l}^\varepsilon)}\rfloor $ and $K_l^\lambda=\lambda \beta^l$ then denote $\overline{F}_\infty(x)=\Pr(Z_1^\infty>x)$. Theorem \ref{10.1} is Theorem 10.1 of \cite{arfrgaha}.
\begin{thm}\label{10.1}
 Suppose $\gamma<1$ and properties \ref{Zinfty} and \ref{Zsup} hold then
\[S_{M_l^\lambda}^l/K_l^\lambda \cd R_{d_\lambda,0,\Lc_\lambda}\] where $R_{d_\lambda,0,\Lc_\lambda}$ has an infinitely divisible law with drift \[d_\lambda =\lambda^{1+\gamma}(1-\beta^{-\gamma})\sum_{K \in \Zb}\beta^{(1+\gamma)K}\Er\left[\frac{Z_1^\infty}{(\lambda\beta^K)^2+(Z_1^\infty)^2}\right],\] $0$ variance and L\'{e}vy spectral function $\Lc_\lambda$ satisfying $\Lc_\lambda(x)=\lambda^\gamma \Lc_1(\lambda x)$ for all $\lambda>0, x \in \Rb$ with $\Lc_1(x)=0$ for $x<0$ and \[\Lc_1(x)=-(1-\beta^{-\gamma})\sum_{K \in \Zb} \beta^{K\gamma} \overline{F}_\infty(x\beta^{K})\] for $x\geq 0$.
\end{thm}
Combining this with Corollary \ref{sumiidstab} with $\lambda=(tC_\Dc)^{1/\gamma}=(tc_\mu\Er[\xi^*-1])^{1/\gamma}$ and $l=h_{n_l}^0$ we have that 
\[\frac{\Delta_{n_l(t)}}{(C_\Dc n_l(t))^\frac{1}{\gamma}} \cd R_{d_{(tC_\Dc)^{1/\gamma}},0,\Lc_{(tC_\Dc)^{1/\gamma}}}\] which proves Theorem \ref{finvarthm}.

\subsection{Proof of Theorem \ref{infallthm} (IVIE)}
In IVIE write $\gamma_\alpha =(\alpha-1)\log(\mu^{-1})/\log(\beta)=(\alpha-1)\gamma$. By (\ref{numcrttrp}) we have that \[\Pr(\Hc(\Tc^{*-})> n) \sim c_\mu^{\alpha-1}\Gamma(2-\alpha)\Pr(\xi^* \geq \mu^{-n}) \sim C_{\mu,\alpha} \beta^{-\gamma_\alpha n} L\left(\beta^{\gamma_\alpha n}\right) \] for a known constant $C_{\mu,\alpha}$. Due to the slowly varying term, we cannot apply Theorem \ref{10.1} directly however Theorem \ref{10.1} is proved using Theorem \ref{eqconv}. Using this, it will suffice to show convergence of the drift, variance and L\'{e}vy spectral function in this case. 

Recall that we consider subsequences $n_l(t)$ such that $a_{n_l(t)} \sim t\mu^{-l}$. For $i,l\geq 1$ let $\zeta_i^l=\tilde{\chi}_{n_l}^{i*}\beta^{-\Hc(\Tc_{\rho_i})}$ under $\Hc(\Tc_{\rho_i}) \geq h_{n_l}^\varepsilon$ where $(\tilde{\chi}_{n_l}^{i*})_{i\geq 1}$ are i.i.d.\ with the law of $\chi_{n_l}^{1*}$ and $(\Tc_{\rho_i})_{i\geq 1}$ are the associated trees. For $K \geq -(l-h_{n_l}^\varepsilon)$ let $\zeta_i^{l,K}$ to be $\zeta_i^l$ under $\Hc(\Tc_{\rho_i})=l+K$ when this makes sense and $0$ otherwise. Denote $\overline{F}_{\scriptscriptstyle{K}}^l(x)=\Pr^{\scriptscriptstyle{K}}(\zeta_1^l>x)$. From Propositions \ref{conver} and \ref{nkconv} we then have that for any $K \in \Zb$ the laws of $\zeta_i^{l,K}$ converge to the laws of $Z_\infty$ as $l \rightarrow \infty$. Let $(Z_\infty^i)_{i\geq 1}$ be an independent sequence of variables with this law and denote $\overline{F}_\infty(x)=\Pr(Z_\infty>x)$. By Lemma \ref{dominate}, $\exists Z_{sup}$ such that $\zeta_i^{l,K} \preceq Z_{sup}$ for all $l \in \Nb, K \geq -(l-h_{n_l}^\varepsilon)$ and $\Er[Z_{sup}^{\gamma_\alpha +\delta}]<\infty$ for some $\delta>0$; we denote $\overline{F}_{sup}(x)=\Pr(Z_{sup}>x)$. For $(\lambda_l)_{l\geq 0}$ converging to $\lambda>0$ define $K_l^\lambda=\lambda \beta^{l}$ and for $C_{\alpha,\mu}=\mu^{-1}(2-\alpha)/(\alpha-1)$  \[M_l^\lambda=\left\lfloor \lambda_l^{\gamma_\alpha}  \beta^{\gamma_\alpha l} \frac{\Pr(\xi^*>\mu^{-h_{n_l}^\varepsilon})}{C_{\alpha,\mu}L(\mu^{-h_{n_l}^0})}\right\rfloor. \] 

\begin{prp}\label{iveeinfdiv} In IVIE, for any $\lambda>0$, as $l \rightarrow \infty$ 
 \begin{flalign*}
  \sum_{i=1}^{M_l^\lambda}\frac{\tilde{\chi}_{n_l}^{i*}}{K_l^\lambda} \cd R_{d_\lambda,0,\Lc_\lambda}
 \end{flalign*}
where
\begin{flalign*}
 d_\lambda & = \lambda^{1+\gamma_\alpha}(1-\beta^{-\gamma_\alpha})\sum_{K \in \Zb}\beta^{(1+\gamma_\alpha)K}\Eb\left[\frac{Z_\infty}{(\lambda\beta^{K})^2+Z_\infty^2}\right], \\
 \Lc_\lambda(x) & = \begin{cases} 0 & x\leq 0; \\ -\lambda^{\gamma_\alpha}(1-\beta^{-\gamma_\alpha}) \sum_{K \in \Zb}\beta^{K\gamma_\alpha}\overline{F}_\infty(\lambda x \beta^{K\gamma_\alpha}) & x>0.  \end{cases}
\end{flalign*}
 \begin{proof}
By Theorem \ref{eqconv} it suffices to show the following:
\begin{enumerate}
 \item\label{decay} for all $\varepsilon>0$ \[\lim_{l\rightarrow \infty} \Pr\left(\frac{\tilde{\chi}_{n_l}^{1*}}{K_l^\lambda}>\varepsilon \right)=0;\]
 \item\label{lvyspc} for all $x$ continuity points \[ \Lc_\lambda(x)=\begin{cases} 0 & x\leq 0, \\ -\lim_{l\rightarrow \infty} M_l^\lambda \Pr\left(\frac{\tilde{\chi}_{n_l}^{1*}}{K_l^\lambda}>x \right) & x>0;  \end{cases}\]
 \item\label{drft} for all $\tau>0$ continuity points of $\Lc$ \[d_\lambda  = \lim_{l \rightarrow \infty}M_l^\lambda\Er\left[\frac{\tilde{\chi}_{n_l}^{1*}}{K_l^\lambda}\ind_{\{\tilde{\chi}_{n_l}^{1*}\leq \tau K_l^\lambda\}}\right]+\int_{|x|\geq \tau}\frac{x}{1+x^2}\d\Lc_\lambda(x) -\int_{\tau\geq |x|>0}\frac{x^3}{1+x^2}\d\Lc_\lambda(x);\] 
 \item\label{sigsqr} \[\lim_{\tau\rightarrow 0}\limsup_{l\rightarrow \infty} M_l^\lambda Var\left(\frac{\tilde{\chi}_{n_l}^{1*}}{K_l^\lambda}\ind_{\{\tilde{\chi}_{n_l}^{1*}\leq \tau K_l^\lambda\}}\right)=0.\]
\end{enumerate}

We prove each of these in turn but we start by introducing a relation which will be fundamental to proving the final parts. For $K \in \Zb$ let $c_l^{\scriptscriptstyle{K}}=\Pr(\Hc(\Tc^{*-})> l+K| \Hc(\Tc^{*-})> h_{n_l}^\varepsilon)$ denote the probability that a deep branch is of height at least $l+K$. Then by the asymptotic (\ref{numcrttrp}) we have that, for $K$ such that $l+K\geq h_{n_l}^\varepsilon$, as $l \rightarrow \infty$
\begin{flalign*}
 c_l^{\scriptscriptstyle{K}} \; = \; \frac{\Pr(\Hc(\Tc^{*-})>l+K)}{\Pr(\Hc(\Tc^{*-})>h_{n_l}^\varepsilon)} \; \sim \; \mu^{(\alpha-1)K}\frac{\Pr(\Hc(\Tc^{*-})>l)}{\Pr(\Hc(\Tc^{*-})>h_{n_l}^\varepsilon)}.
\end{flalign*}
 In particular, using (\ref{numcrttrp}) and that $\beta^{\gamma_\alpha}=\mu^{-(\alpha-1)}$
 \begin{flalign*}
  M_l^\lambda c_l^{\scriptscriptstyle{K}} \sim \lambda^{\gamma_\alpha} \left(\frac{\Pr(\xi^*>\mu^{h_{n_l}^\varepsilon})}{\Pr(\Hc(\Tc^{*-})>h_{n_l}^\varepsilon)}\right)\left(\frac{\Pr(\Hc(\Tc^{*-})>h_{n_l}^0)}{\Pr(\xi^*>\mu^{h_{n_l}^0})}\right) \frac{\beta^{-\gamma_\alpha  l}}{\mu^{-(\alpha-1)l}}\mu^{(\alpha-1)K} \sim \lambda^{\gamma_\alpha}\beta^{-\gamma_\alpha K}
 \end{flalign*}
 thus $M_l^\lambda (c_l^{\scriptscriptstyle{K}}-c_l^{\scriptscriptstyle{K}+1})\rightarrow \lambda^{\gamma_\alpha}\beta^{-\gamma_\alpha  K}(1-\beta^{-\gamma_\alpha })$ and for any $\epsilon>0$ and large enough $l$ 
 \begin{flalign}\label{imbnd}
 M_l^\lambda c_l^{\scriptscriptstyle{K}} \leq C_\epsilon\lambda^{\gamma_\alpha}  \beta^{-\gamma_\alpha  K}\beta^{\epsilon|K|}. 
 \end{flalign}
 
To prove (\ref{decay}), notice that 
 \begin{flalign*}
 \Pr\left(\frac{\tilde{\chi}_{n_l}^{1*}}{K_l^\lambda}>\varepsilon \right)  \leq \Pr(\Hc(\Tc^{*-})\geq h_n^{\varepsilon/2}|\Hc(\Tc^{*-})\geq h_n^{\varepsilon}) + \Pr\left(\beta^{h_n^{\varepsilon/2}-l}Z_{sup}>\lambda\varepsilon \right). 
 \end{flalign*}
Both terms converge to $0$ as $l \rightarrow \infty$ by the tail formula of a branch (\ref{numcrttrp}), the fact that $Z_{sup}$ has no atom at $\infty$ and that $\beta^{h_{n_l}^{\varepsilon/2}-l}\rightarrow 0$ which follows from $l\sim h_n^0$.

For (\ref{lvyspc}) we have that
\begin{flalign*}
M_l^\lambda \Pb\left(\frac{\tilde{\chi}_{n_l}^{1*}}{K_l^\lambda}>x\right) & = \sum_{K \in \Zb}\ind_{\{K\geq -(l-h_{n_l}^\varepsilon\}}M_l^\lambda \Pr(\Hc(\Tc^{*-})=l+K)\Pb^{\scriptscriptstyle{K}}\left(\frac{\tilde{\chi}_{n_l}^{1*}}{K_l^\lambda}>x\right) \\
& =  \sum_{K \in \Zb}\ind_{\{K\geq -(l-h_{n_l}^\varepsilon\}}M_l^\lambda (c_l^{\scriptscriptstyle{K}}-c_l^{\scriptscriptstyle{K}+1})\overline{F}_{\scriptscriptstyle{K}}^l(\lambda \beta^{-K}x).
\end{flalign*}
If $x>0$ is a continuity point of $\Lc_\lambda$ then $\lambda x\beta^{-K}$ is a continuity point of $\overline{F}_\infty$ hence for any $K\in \Zb$ as $l \rightarrow \infty$ \[ \ind_{\{K \geq -(l-h_{n_l}^\varepsilon)\}}M_l^\lambda (c_l^{\scriptscriptstyle{K}}-c_l^{\scriptscriptstyle{K}+1})\overline{F}_{\scriptscriptstyle{K}}^l(\lambda\beta^{-K}x) \rightarrow  \lambda^{\gamma_\alpha}\beta^{-\gamma_\alpha  K}(1-\beta^{-\gamma_\alpha })\overline{F}_\infty(\lambda\beta^{-K}x).\]
We need to exchange the sum and the limit; we do this using dominated convergence. Since $\gamma_\alpha<1$ we can choose $\delta>0$ such that $\gamma_\alpha+\delta<1$ and $\delta<\gamma_\alpha$. By (\ref{imbnd}), for $l$ sufficiently large $M_l^\lambda c_l^{\scriptscriptstyle{K}} \leq C_{\delta,\lambda}\beta^{-\gamma_\alpha  K}\beta^{\frac{\delta}{2}|K|}$ hence 
\[ \sum_{K\geq -(l-h_{n_l}^\varepsilon)}M_l^\lambda(c_l^{\scriptscriptstyle{K}}-c_l^{\scriptscriptstyle{K}+1})\overline{F}_{\scriptscriptstyle{K}}^l(\lambda\beta^{-K}x)  \leq C\sum_{K\in \Zb}\overline{F}_{sup}(\lambda x\beta^{-K})\beta^{-\gamma_\alpha  K}\beta^{\frac{\delta}{2}|K|}. \]
Since $Z_{sup}$ has moments up to $\gamma_\alpha +\delta$ we have that for $y=\lambda x$
\begin{flalign*}
 \sum_{K< 0}\overline{F}_{sup}(\lambda x\beta^{-K})\beta^{-\gamma_\alpha  K}\beta^{\frac{\delta}{2}|K|} \; = \; \Er\left[\sum_{K=0}^{\left\lfloor \frac{\log(Z_{sup}/y)}{\log(\beta)}\right\rfloor} \beta^{K(\gamma_\alpha  + \delta/2)}\right] \; \leq \; C_y\Er\left[Z_{sup}^{\gamma_\alpha +\frac{\delta}{2}}\right]
\end{flalign*}
which is finite. Clearly
\[ \sum_{K\geq 0}\overline{F}_{sup}(\lambda x\beta^{-K})\beta^{-\gamma_\alpha  K}\beta^{\frac{\delta}{2}|K|}  \leq \sum_{K\geq 0}\beta^{\left(\frac{\delta}{2}-\gamma_\alpha \right) K}   <\infty\]
by choice of $\delta$. It therefore follows that for $x>0$ \[ -\lim_{l\rightarrow \infty} M_l^\lambda \Pr\left(\frac{\tilde{\chi}_{n_l}^{1*}}{K_l^\lambda}>x \right)=-\lambda^{\gamma_\alpha} (1-\beta^{-\gamma_\alpha })\sum_{K\in \Zb}\overline{F}_\infty(\lambda x \beta^{\gamma_\alpha  K})\beta^{\gamma_\alpha K}.\] 
Moreover, for $x<0$ we have that $ \Pr\left(\tilde{\chi}_{n_l}^{1*}/K_l^\lambda<x \right)=0$ which gives (\ref{lvyspc}). 

For (\ref{drft}) we have that $\int_0^\tau x \d\Lc_\lambda$ is well defined therefore
\begin{flalign*}
 \int_\tau^\infty \frac{x}{1+x^2} \d\Lc_\lambda - \int_0^\tau \frac{x^3}{1+x^2} \d\Lc_\lambda \; = \; \int_0^\infty \frac{x}{1+x^2} \d\Lc_\lambda -\int_0^\tau x \d\Lc_\lambda.
\end{flalign*}
We therefore want to show that \[ \lim_{l \rightarrow \infty}\frac{M_l^\lambda}{K_l^\lambda} \Er[\tilde{\chi}_{n_l}^{1*}\ind_{\{\tilde{\chi}_{n_l}^{1*}\leq \tau K_l^\lambda\}}] = \int_0^\tau x \d\Lc_\lambda. \]
Write $G_{\scriptscriptstyle{K}}^l(u)=\Er^{\scriptscriptstyle{K}}\left[Z_n\ind_{\{Z_n\leq u\}} \right]$ and $G_\infty(u)=\Er[Z_\infty\ind_{\{Z_\infty\leq u\}}]$. Then we have that 
\begin{flalign*}
 \frac{M_l^\lambda}{K_l^\lambda} \Er[\tilde{\chi}_{n_l}^{1*}\ind_{\{\tilde{\chi}_{n_l}^{1*}\leq \tau K_l^\lambda\}}]  & = \lambda^{-1}\sum_{K\geq -(l-h_{n_l}^\varepsilon)}M_l^\lambda(c_l^{\scriptscriptstyle{K}}-c_l^{\scriptscriptstyle{K}+1})\beta^{K} G_{\scriptscriptstyle{K}}^l(\tau \lambda \beta^{-K}). 
\end{flalign*}
For each $K\in \Zb$ as $l\rightarrow \infty$ \[M_l^\lambda(c_l^{\scriptscriptstyle{K}}-c_l^{\scriptscriptstyle{K}+1})\beta^{K} G_{\scriptscriptstyle{K}}^l(\tau \lambda \beta^{-K}) \rightarrow \lambda^{\gamma_\alpha}(1-\beta^{-\gamma_\alpha })\beta^{(1-\gamma_\alpha )K}G_\infty(\tau\lambda \beta^{-K}).\] 
We want to exchange the limit and the sum which we do by dominated convergence. For any $\kappa \in [0,1]$ and random variable $Y$ we have that $\Er[Y\ind_{\{Y \leq u\}}]\leq u^\kappa\Er[Y^{1-\kappa}\ind_{\{Y\leq u\}}]$. Using this with $u=\tau\lambda \beta^{-K}$ where $\kappa=1-\gamma_\alpha-2\delta/3$ for $K\geq 0$ and $\kappa=1$ for $K<0$, alongside (\ref{imbnd}) we have that
\begin{flalign*}
 & \sum_{K\in \Zb}\ind_{\{K \geq -(l-h_{n_l}^\varepsilon)\}}M_l^\lambda(c_l^{\scriptscriptstyle{K}}-c_l^{\scriptscriptstyle{K}+1})\beta^K G_{\scriptscriptstyle{K}}^l(\tau \lambda \beta^{-K}) \\
  & \qquad \qquad \leq \sum_{K\geq 0}M_l^\lambda(c_l^{\scriptscriptstyle{K}}-c_l^{\scriptscriptstyle{K}+1})\beta^K\tau \lambda \beta^{-K} \\
  & \qquad \qquad \qquad + \sum_{K<0} M_l^\lambda(c_l^{\scriptscriptstyle{K}}-c_l^{\scriptscriptstyle{K}+1})\beta^K \left(\beta^{\frac{2\delta}{3} K}(\tau \lambda)^{1-\gamma_\alpha -\frac{2\delta}{3}}\Er[Z_{sup}^{\gamma_\alpha +\frac{2\delta}{3}}]\beta^{(\gamma_\alpha -1)K}\right) \\
  & \qquad \qquad \leq C_\lambda \tau \sum_{K\geq 0} \beta^{-(\gamma_\alpha-\delta/2)  K} + C_\lambda\tau^{1-\gamma_\alpha -\frac{2\delta}{3}}\Er[Z_{sup}^{\gamma_\alpha +\frac{2\delta}{3}}]\sum_{K<0}\beta^{\frac{\delta}{6} K}
\end{flalign*}
 which is finite since $\gamma_\alpha>\delta/2$ and $Z_{sup}$ has moments up to $\gamma_\alpha+\delta$. We therefore have that \[\lim_{l\rightarrow \infty} \frac{M_l^\lambda}{K_l^\lambda} \Er[\tilde{\chi}_{n_l}^{1*}\ind_{\{\tilde{\chi}_{n_l}^{1*}\leq \tau K_l^\lambda\}}] = \lambda^{\gamma_\alpha -1}(1-\beta^{-\gamma_\alpha })\sum_{K\in \Zb}\beta^{K(\gamma_\alpha -1)}G_\infty(\tau\lambda\beta^K). \]
 By definition we have that 
 \begin{flalign*}
  \int_0^\tau x \d\Lc_\lambda & = \lambda^{\gamma_\alpha} (1-\beta^{-\gamma_\alpha })\int_0^\tau x\sum_{K\in \Zb}\beta^{\gamma_\alpha  K}\d(-\overline{F}_\infty)(\lambda x \beta^K) \\
  & = \lambda^{\gamma_\alpha-1} (1-\beta^{-\gamma_\alpha })\sum_{K\in \Zb}\beta^{(\gamma_\alpha-1) K}\int_{\lambda x\beta^K\leq \lambda \tau\beta^K} \lambda x\beta^K\d(-\overline{F}_\infty)(\lambda x \beta^K) \\
  & = \lambda^{\gamma_\alpha -1}(1-\beta^{-\gamma_\alpha })\sum_{K\in \Zb}\beta^{K(\gamma_\alpha -1)}G_\infty(\tau\lambda\beta^K).
 \end{flalign*}
It therefore remains to calculate $\int_0^\infty \frac{x}{1+x^2}\d\Lc_\lambda$.
\begin{flalign*}
 \int_0^\infty \frac{x}{1+x^2}\d\Lc_\lambda & = \lambda^{\gamma_\alpha}(1-\beta^{-\gamma_\alpha})\int_0^\infty \frac{x}{1+x^2} \sum_{K \in \Zb}\beta^{\gamma_\alpha K}\d(-\overline{F}_\infty)(\lambda x\beta^K) \\
  & = \lambda^{\gamma_\alpha+1}(1-\beta^{-\gamma_\alpha})\sum_{K \in \Zb}\beta^{(\gamma_\alpha+1)K} \Eb\left[\frac{Z_\infty}{(\lambda\beta^K)^2+Z_\infty^2}\right].
\end{flalign*}
 The final sum is finite since for $K<0$
 \[\beta^{(\gamma_\alpha+1)K} \Eb\left[\frac{Z_\infty}{(\lambda\beta^K)^2+Z_\infty^2}\right] = \lambda^{-1}\beta^{\gamma_\alpha K} \Eb\left[\frac{\lambda \beta^KZ_\infty}{(\lambda\beta^K)^2+Z_\infty^2}\right]\leq \lambda^{-1}\beta^{\gamma_\alpha K} \]
Which is summable and for $K\geq 0$
\begin{flalign*}
 \Eb\left[\frac{Z_\infty}{(\lambda\beta^K)^2+Z_\infty^2}\right]   \leq \Eb\left[\frac{Z_\infty}{(\lambda\beta^K)^2}\ind_{\{Z_\infty \leq \lambda\beta^K\}}+ Z_\infty^{-1}\ind_{\{Z_\infty \geq \lambda\beta^K\}}\right]  \leq C_\lambda\Eb\left[Z_{sup}^{\gamma_\alpha+\delta/2}\right]\beta^{-K(1+\gamma_\alpha+\delta/2)}
\end{flalign*}
which, multiplied by $\beta^{(\gamma_\alpha+1)K}$, is summable.

It now remains to prove (\ref{sigsqr}). It suffices to show that 
\begin{flalign}\label{sigdec}
\lim_{\tau\rightarrow 0^+}\lim_{l\rightarrow \infty} \frac{M_l^\lambda}{(K_l^\lambda)^2}\Er\left[(\tilde{\chi}_{n_l}^{1*})^2\ind_{\{\tilde{\chi}_{n_l}^{1*}\leq \tau K_l^\lambda\}}\right] =0.
\end{flalign}
Write $H_{\scriptscriptstyle{K}}^l(u)=\Er^{\scriptscriptstyle{K}}\left[(\zeta_1^l)^2\ind_{\{\zeta_1^l\leq u\}} \right]$ then 
 \begin{flalign*}
  \frac{M_l^\lambda}{(K_l^\lambda)^2}\Er\left[(\tilde{\chi}_{n_l}^{1*})^2\ind_{\{\tilde{\chi}_{n_l}^{1*}\leq \tau K_l^\lambda\}}\right] & = \frac{M_l^\lambda}{(K_l^\lambda)^2}\sum_{K\in \Zb}(c_l^{\scriptscriptstyle{K}}-c_l^{\scriptscriptstyle{K}+1})\beta^{2(l+K)}H_{\scriptscriptstyle{K}}^l(\tau\lambda\beta^{-K}) \\
  & \leq C_\lambda \sum_{K\in \Zb}\beta^{(2-\gamma_\alpha )K}\beta^{\frac{\delta}{2}|K|}H_{\scriptscriptstyle{K}}^{l}(\tau\lambda\beta^{-K}). 
 \end{flalign*}
Using that for any random variable $Y$ we have $\Eb[Y^2\ind_{Y\leq u}]\leq u^\kappa\Eb[Y^{2-\kappa}\ind_{Y\leq u}]$ with $u=\tau\lambda\beta^{-K}$ and $\kappa =2$ it follows that
\begin{flalign*}
 \sum_{K \geq 0}\beta^{(2-\gamma_\alpha )K}\beta^{\frac{\delta}{2}|K|}H_{\scriptscriptstyle{K}}^l(\tau\lambda\beta^{-K}) \; \leq \; C\tau^2\sum_{K\geq 0}\beta^{-(\gamma_\alpha-\delta/2)K} \; \leq \; C\tau^2
\end{flalign*}
where the constant $C$ depends on $\beta, \gamma_\alpha$ and $\delta$. Then, with $u=\tau\lambda\beta^{-K}, \; \kappa =2-\gamma_\alpha-2\delta/3$ we have that \[H_{\scriptscriptstyle{K}}^l(\tau\lambda\beta^{-K})\beta^{(2-\gamma_\alpha )K}\leq \beta^{\frac{2\delta}{3}K}(\tau\lambda)^{2-\gamma_\alpha -\frac{2\delta}{3}}\Er[Z_{sup}^{\gamma_\alpha +\frac{2\delta}{3}}]\] and therefore 
\begin{flalign*}
 \sum_{K \leq 0}\beta^{(2-\gamma_\alpha )K}\beta^{\frac{\delta}{2}|K|} H_{\scriptscriptstyle{K}}^l(\tau\lambda\beta^{-K}) \; \leq \; C\tau^{2-\gamma_\alpha -\frac{2\delta}{3}}\Er[Z_{sup}^{\gamma_\alpha +\frac{2\delta}{3}}]\sum_{K\geq 0}\beta^{\frac{\delta}{6}K} \; \leq \; C\tau^{2-\gamma_\alpha-\frac{2\delta}{3}}.
 \end{flalign*}
Since $\gamma_\alpha+\frac{2\delta}{3}<2$ we have that (\ref{sigdec}) holds.
 \end{proof}
\end{prp}

Combining Proposition \ref{iveeinfdiv} with Corollary \ref{uspdiff} and Lemma \ref{sumiidstab} with \[\lambda=\Gamma(2-\alpha)^{\frac{1}{\gamma_\alpha}}c_\mu^{\frac{1}{\gamma}}\beta^{\frac{\log(t)}{\log(\mu^{-1})}-\left\lfloor\frac{\log(t)}{\log(\mu^{-1})}\right\rfloor}\] proves Theorem \ref{infallthm}.

\section{Tightness}\label{tght}
We conclude the results for the walk on the subcritical tree with Theorem \ref{tghtthm} which is a tightness result for the process and a convergence result for the scaling exponent. We only prove the result in IVIE since the proof is standard (similar to that of Theorem 1.1 of \cite{arfrgaha}) and the other cases follow by the same method; however, we state the proof more generally. Recall that $r_n$ is $a_n$ in IVFE, $n^{1/\gamma}$ in FVIE, $a_n^{1/\gamma}$ in IVIE and $b_n:=\max\{m\geq 0:r_m\leq n\}$.

\begin{proof}[Proof of Theorem \ref{tghtthm}]
 As stated previously, we only prove the results in IVIE since the others follow by a similar calculation.
 
 For statement \ref{deltght} we want to show that $\lim_{t\rightarrow \infty}\limsup_{n\rightarrow \infty}\Pb\left(\Delta_n/r_n \notin [t^{-1},t]\right)=0$. Let $k$ be such that $a_{n_k(1)} \leq a_n < a_{n_{k+1}(1)}$ then
 \begin{flalign*}
  \Pb\left(\frac{\Delta_n}{a_n^{1/\gamma}} \notin[t^{-1},t]\right) & \leq \Pb\left(\frac{\Delta_{n_k(1)}}{a_{n_{k+1}(1)}^{1/\gamma}}<t^{-1}\right) + \Pb\left(\frac{\Delta_{n_{k+1}(1)}}{a_{n_k(1)}^{1/\gamma}}>t\right).
 \end{flalign*}
Since, for large enough $n$ we have that $(a_{n_{k+1}(1)}/a_{n_k(1)})^{1/\gamma}$ can be bounded above by some constant $c$, by continuity of the distribution of $R_1$ (which follows from $\lim_{x\rightarrow 0}\Lc(x)=-\infty$ and Theorem III.2 of \cite{pe}) \[\lim_{t\rightarrow \infty}\limsup_{n\rightarrow \infty}\Pb\left(\Delta_n/r_n \notin [t^{-1},t]\right) \leq \lim_{t\rightarrow \infty} \Pb\left(R_1 \notin [(tc)^{-1},tc]\right)=0.\]

For statement \ref{xtght} we want to show that $\lim_{t\rightarrow \infty}\limsup_{n\rightarrow \infty}\Pb\left(|X_n|/b_n \notin [t^{-1},t]\right)=0$. We want to compare $|X_n|$ with $\Delta_n$. In order to deal with the depth $X_n$ reaches into the traps we use a bound for the height of a trap; for any $\delta>0$ we have \[\Pb\left(\frac{|X_n|}{b_n}\geq t\right) \leq \Pb\left(\Delta_{\lfloor tb_n-b_n^\delta\rfloor}\leq n\right)+(tb_n-b_n^\delta)\Pb\left(\Hc(\Tc^{*-})\geq b_n^\delta\right).\]
By (\ref{numcrttrp}) we have that $(tb_n-b_n^\delta)\Pb\left(\Hc(\Tc^{*-})\geq b_n^\delta\right)\rightarrow 0$ as $\nin$. Using the definition of $b_n$ we have that
\[\Pb\left(\Delta_{\lfloor tb_n-b_n^\delta\rfloor}\leq n\right) \leq \Pb\left(\frac{\Delta_{\lfloor tb_n-b_n^\delta\rfloor}}{a_{tb_n-b_n^\delta}^{1/\gamma}}\leq \frac{a_{b_n+1}^{1/\gamma}}{a_{tb_n-b_n^\delta}^{1/\gamma}}\right).\]
Since $a_{b_n+1}^{1/\gamma}/a_{tb_n-b_n^\delta}^{1/\gamma}$ converges to $t^{-1/\gamma_\alpha}$ as $\nin$, by continuity of the distribution of $R_1$ and statement \ref{deltght} we have that $\lim_{t\rightarrow \infty}\limsup_{n\rightarrow \infty}\Pb\left(|X_n|/b_n >t\right)=0$. 

It remains to show that $\lim_{t\rightarrow \infty}\limsup_{n\rightarrow \infty}\Pb\left(|X_n|/b_n <t^{-1}\right)=0$. In this case we need to bound how far the walker backtracks after reaching a new furthest point in order to compare $|X_n|$ with $\Delta_n$. We have that \[\max_{i<j\leq n} (|X_i|-|X_j|) \leq \tau_1 \lor \max_{2\leq i\leq n} (\tau_i-\tau_{i-1}) + \max_{0\leq i \leq n}\Hc(\Tc^{*-}_{\rho_i})\] where $\tau_i$ are the regeneration times for $Y$. In particular, $(\tau_i-\tau_{i-1}), \tau_1$ and $\Hc(\Tc^{*-}_{\rho_i})$ have exponential moments for all $i$ therefore for any $\delta>0$
\[\limn\Pb\left(\max_{i<j\leq n} (|X_i|-|X_j|)>b_n^\delta\right)=0.\]
We then have that
\begin{flalign*}
 \Pb\left(|X_n|/b_n <t^{-1}\right) & \leq \Pb\left(\max_{i<j\leq n} |X_i|-|X_j|>b_n^\delta\right) + \Pb\left(\Delta_{\lfloor t^{-1}b_n+b_n^\delta\rfloor}>n\right) \\
 & \leq o(1) + \Pb\left(\frac{\Delta_{\lfloor 2t^{-1}b_n\rfloor}}{a_{2t^{-1}b_n}^{1/\gamma}}>\frac{a_{b_n}^{1/\gamma}}{a_{2t^{-1}b_n}^{1/\gamma}}\right).
\end{flalign*}
Then, since $a_{b_n}^{1/\gamma}/a_{2t^{-1}b_n}^{1/\gamma} \rightarrow (t/2)^{1/\gamma_\alpha}$ as $\nin$, by continuity of the distribution of $R_1$ and statement \ref{deltght} we indeed have that $\lim_{t\rightarrow \infty}\limsup_{n\rightarrow \infty}\Pb\left(|X_n|/b_n <t^{-1}\right)=0$.

For the final statement notice that \[\Pb\left(\limn \frac{\log|X_n|}{\log(n)}\neq \gamma(\alpha-1)\right)=\Pb\left(\limn \frac{\log|X_n|}{\log(b_n)}\frac{\log(b_n)}{\log(n)}\neq \gamma(\alpha-1)\right)\]
and since $b_n=n^{\gamma(\alpha-1)}\tilde{L}(n)$ for some slowly varying function $\tilde{L}$ we have that as $\nin$ $\log(b_n)/\log(n) \rightarrow \gamma(\alpha-1)$ thus it suffices to show that the following is equal to $0$
\begin{flalign*}
\Pb\left(\limn \frac{\log|X_n|}{\log(b_n)}\neq 1\right)  \leq \Pb\left(\limsup_{\nin} \frac{\log|X_n|}{\log(b_n)}>1\right) + \lim_{t\rightarrow \infty}\Pb\left(\liminf_{\nin} \frac{|X_n|}{b_n} \leq t^{-1}\right).
\end{flalign*}

By Fatou we can bound the second term above by $\lim\limits_{t \rightarrow \infty} \liminf\limits_{\nin} \; \Pb\left(|X_n|/b_n \leq t^{-1}\right)$ which is equal to $0$ by tightness of $(|X_n|/b_n)_{n\geq 0}$.

For the first term we have
\begin{flalign*}
 \Pb\left(\limsup_{\nin} \frac{\log|X_n|}{\log(b_n)}>1\right) & = \lim_{\varepsilon \rightarrow 0^+}\Pb\left(\limsup_{\nin}\frac{\log|X_n|}{\log(b_n)}\geq 1+\varepsilon\right) \\
 & \leq \lim_{\varepsilon \rightarrow 0^+}\Pb\left(\limn\frac{\sup_{k\leq n}|X_n|}{b_n^{1+\varepsilon}}\geq 1\right).
\end{flalign*}

Writing $D'(n):=\{\max\limits_{i=0,...,n}\Hc(\Tc^{*-}_{\rho_i})\leq 4\log(a_n)/\log(\mu^{-1})\}$ we have that $\Pb(D'(n)^c)=o(n^{-2})$ by (\ref{numcrttrp}) thus $\Pr(D'(n)^c \; i.o.)=0$. On $D'(n)$
\[\sup_{k\leq n}|X_k| \leq |X_{\kappa_n}|+ \kappa_{n+1}-\kappa_n+\frac{4\log(a_n)}{\log(\mu^{-1})}\]
where $\kappa_n$ is the last regeneration time of $Y$ before time $n$. Therefore, since $\kappa_{n+1}-\kappa_n$ have exponential moments we have that $\Pb(\limsup_{\nin}(\kappa_{n+1}-\kappa_n)\geq b_n) =0$; hence,
\[\Pb\left(\limn\frac{\sup_{k\leq n}|X_n|}{b_n^{1+\varepsilon}}\geq 1\right) \leq \Pb\left(\liminf_{\nin}\frac{|X_{\kappa_n}|}{b_n^{1+\varepsilon}}\geq 1-o(1)\right) \leq \lim_{t\rightarrow \infty}\liminf_{\nin}\Pb\left(\frac{|X_n|}{b_n}\geq t\right)\]
where the second inequality follows by Fatou's lemma. The result follows by tightness of $(|X_n|/b_n)_{n\geq 0}$.
\end{proof}

Theorem \ref{finexcthm} follows from Theorem \ref{tghtthm}, Proposition \ref{convinffee} and Corollary \ref{sumiidstab} with $\lambda =t$ since $nq_n\sim n^\varepsilon$. More specifically, since $R_{d_t,0,\Lc_t}$ is the infinitely divisible law with characteristic exponent
\begin{flalign*}
 id_1t + \int_0^\infty e^{itx}-1-\frac{itx}{1+x^2}\d\Lc_1(x)  = \int_0^\infty e^{itx}-1 d\Lc_1(x)  = t^{-(\alpha-1)}\int_0^\infty e^{ix}-1 d\Lc_1(x)
\end{flalign*}
by a simple change of variables calculation we have that the laws of the process $\left(\Delta_{nt}/a_n\right)_{t\geq 0}$ converge weakly as $\nin$ under $\Pb$ with respect to the Skorohod $J_1$ topology on $D([0,\infty),\Rb)$ to the law of the stable subordinator with characteristic function $\varphi(t)=e^{-C_\alpha t^{\alpha-1}}$ where $C_{\alpha,\beta,\mu}=-\int_0^\infty e^{ix}-1 d\Lc_1(x)$. A straightforward calculation then shows that the Laplace transform is of the form 
\begin{flalign*}
 \varphi_t(s)  = \Eb[e^{-sX_{d_t,0,\Lc_t}}] = e^{-ts^{\alpha-1} C_{\alpha,\beta,\mu}}
\end{flalign*}
where 
\begin{flalign}\label{Cval}C_{\alpha,\beta,\mu}=\frac{\pi(\alpha-1)}{\sin\left(\pi(\alpha-1)\right)}\cdot\left(\frac{\beta(1-\beta\mu)}{2(\beta-1)}\right)^{\alpha-1}.\end{flalign}

\section{Supercritical tree}\label{suptree}
As discussed in the introduction, the structures of the supercritical and subcritical trees are very similar and consist of some backbone structure $\Yc$ with subcritical GW-trees as leaves.
\begin{itemize}
 \item On the subcritical tree the backbone was a single infinite line of descent, represented by the solid line in Figure \ref{treediag} of Section \ref{numtrp}. On the supercritical tree the backbone is itself a random tree, represented by the solid line in Figure \ref{suptreediag}. In particular, it is a GW-tree without deaths whose law is determined by the generating function $g(s)=\left(f\left((1-q)s+q\right)-q\right)/(1-q)$ where $f$ is the generating function of the original offspring law and $q$ is the extinction probability. 
 \item Each backbone vertex has additional children (which we call buds) which are roots of subcritical GW-trees. On the subcritical tree, the number of buds had a size-biased law independent of the position on the backbone. On the supercritical tree, the distribution over the number of buds is more complicated since it depends on the backbone. Importantly, the expected number of buds can be bounded above by $\mu(1-q)^{-1}$ independently of higher moments of the offspring law which isn't the case for the subcritical tree. 
 \item In the subcritical case, the GW-trees forming the traps have the law of the original (unconditioned) offspring law. In the supercritical case, the law is defined by the p.g.f.\ $h(s)=f(qs)/q$ which has mean $f'(q)$.
\end{itemize}
 In Figure \ref{suptreediag}, the dashed lines represent the finite structures comprised of the buds and leaves. It will be convenient to refer to the traps at a site so for $x \in \Yc$ let $L_x$ denote the collection of traps adjacent to $x$, for example in Figure \ref{suptreediag} $L_\rho$ consists of the two tree rotted at $y,z$. We then write $\Tc^{*-}_{x}$ to be the branch at $x$, that is, the sub-tree formed by $x$, its buds and the associated traps. Let $\Tc$ denote the supercritical tree, $Z_n$ the size of the $n\th$ generation and $Z_n^*$ the number of vertices in the $n\th$ generation of the backbone. We then write $\Pr^*(\cdot):=\Pr(\cdot|Z_0^*>0)$ to be the GW law conditioned on survival. 

\begin{figure}[H]
\centering
 \includegraphics[scale=0.7]{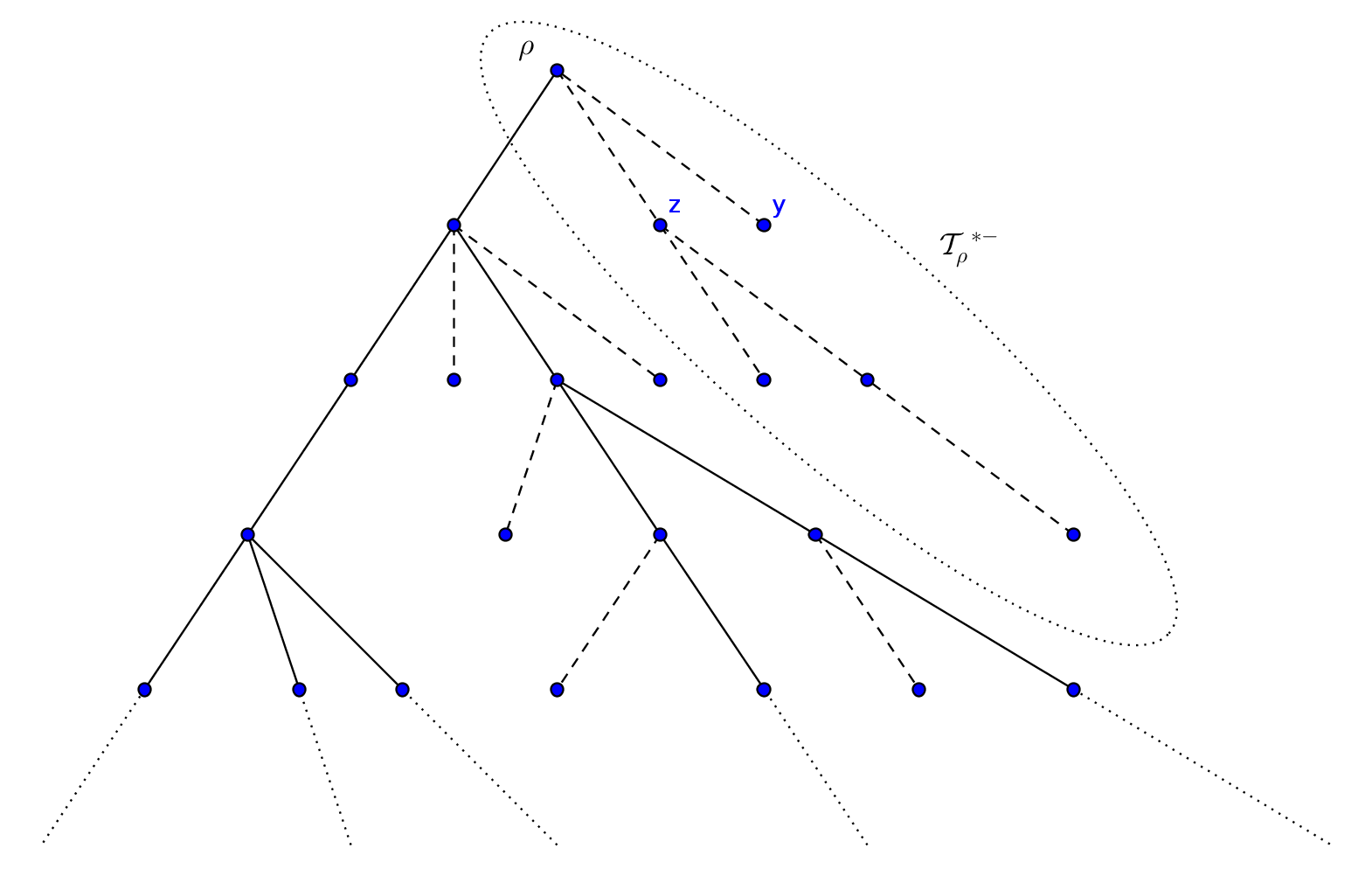} 
\caption{A sample supercritical tree.}\label{suptreediag}
\end{figure}

Let $n_k(t)=\lfloor tf'(q)^{-k}\rfloor$ and $\gamma$ be given as in (\ref{gamma}). In \cite{arfrgaha} it is shown that when $\mu>1, \Er[Z_1^2]<\infty$ and $\beta>f'(q)^{-1}$ we have that $\Delta_{n_k(t)}n_k(t)^{-\frac{1}{\gamma}}$ converges in distribution to an infinitely divisible law. In order to extend this result to prove Theorem \ref{suptreethm} it will suffice to prove Lemmas \ref{6.1}, \ref{7.4} and \ref{maxhgt} which we defer to the end of the section. 

In Lemma \ref{6.1} we show that $\Pr^*(\Hc(\Tc^{*-})>n) \sim C^*f'(q)^n$ for some constant $C^*$. This is the same as when $\Er[\xi^2]<\infty$ for the supercritical tree however, for the subcritical tree, the exponent is different. This is because the first moment of the bud distribution has a fundamental role and the change from finite to infinite variance changes this for the subcritical tree but not for the supercritical tree. Lemma \ref{6.1} is an extension of Lemma 6.1 of \cite{arfrgaha} which is proved using a Taylor expansion of the $f$ around $1$ up to second moments. We cannot take this approach because $f''(1)=\infty$; instead we use the form of the generating function determined in Lemma \ref{genform}. The expression is important because, as in FVIE, the expected time spent in a large branch is approximately $c(\mu\beta)^{\Hc(\Tc^{*-})}$ for some constant $c$. 

Lemma \ref{7.4} shows that, with high probability, no large branch contains more than one large trap. This is important because the number of large traps would affect the escape probability. That is, if there are many large traps in a branch then it is likely that the root has many offspring on the backbone since some geometric number of the offspring lie on the backbone. The analogue of this in \cite{arfrgaha} is proved using the bound $f'(1)-f'(1-\eta)\leq C\eta$ which follows because $f''(1)<\infty$. Similarly to Lemma \ref{6.1}, we use a more precise form of $f$ in order to obtain a similar bound.

Lemma \ref{maxhgt} shows that no branch visited by level $n$ is too large. This is important for the tightness result since we need to bound the deviation of $X$ from the furthest point reached along the backbone. The proof of this follows quite straightforwardly from Lemma \ref{6.1}.

To explain why these are needed, we recall the argument which follows a similar structure to the proof of Theorem \ref{finvarthm}. As was the case for the walk on the subcritical tree, the first part of the argument involves showing that, asymptotically, the time spent outside large branches is negligible. This follows by the same techniques as for the subcritical tree. 

One of the major difficulties with the walk on the supercritical tree is determining the distribution over the number of entrances into a large branch. The height of the branch from a backbone vertex $x$ will be correlated with the number of children $x$ has on the backbone. This affects the escape probability and therefore the number of excursions into the branch. It can be shown that the number of excursions into the first large trap converges in distribution to some non-trivial random variable $W_\infty$. In particular, it is shown in \cite{arfrgaha} that $W_\infty$ can be stochastically dominated by a geometric random variable and that there is some constant $c_W>0$ such that $\Pb(W_\infty>0)\geq c_W$. 

Similarly to Section \ref{trpapr}, it can be shown that asymptotically the large branches are independent in the sense that with high probability the walk won't reach one large branch and then return to a previously visited large branch. Using Lemmas \ref{6.1} and \ref{7.4} (among other results) it can then be shown that $\Delta_n$ can be approximated by the sum of i.i.d.\ random variables.

The remainder of the proof of the first part of Theorem \ref{suptreethm} involves decomposing the time spent in large branches, showing that the suitably scaled excursion times converge in distribution, proving the convergence results for sums of i.i.d.\ variables and concluding with standard tightness results similar to Section \ref{tght}. Since $\Pr(Z_1=k|Z_0^*=0)=p_kq^{k-1}$, the subcritical GW law over the traps has exponential moments. This means that these final parts of the proof follow by the results proven in \cite{arfrgaha} since, by Lemma \ref{6.1}, the scaling is the same as when $\Er[\xi^2]<\infty$.

Tightness of $(\Delta_nn^{-1/\gamma})_{n\geq 0}$ and $(X_nn^{-\gamma})_{n\geq 0}$ and almost sure convergence of $\log(|X_n|)/\log(n)$ then follow by the proof of Theorem 1.1 of \cite{arfrgaha} (with one slight adjustment) which is similar to the proof of Theorem \ref{tghtthm}. In order to bound the maximum distance between the walker's current position and the last regeneration point we used a bound on the maximum height of a trap seen up to $\Delta_n^Y$. In \cite{arfrgaha} it is shown that the probability a trap of height at least $4\log(n)/\log(f'(q)^{-1})$ is seen is at most order $n^{-2}$ by using finite variance of the offspring distribution to bound the variance of the number of traps in a branch. In Lemma \ref{maxhgt} we prove this using Lemma \ref{6.1}.

\begin{lem}\label{6.1}
 Under the assumptions of Theorem \ref{suptreethm} \[\Pr^*(\Hc(\Tc^{*-})> n) \sim C^* f'(q)^n\] where $C^*=q(\mu-f'(q))c_\mu/(1-q)$ and $c_\mu$ is such that $\Pr(\Hc(\Tc)\geq n|\Hc(\Tc)<\infty) \sim c_\mu f'(q)^n$.
 \begin{proof}
 Let $Z=Z_1, Z^*=Z^*_1$ and $s_n=\Pr(\Hc(\Tc)< n|\Hc(\Tc)<\infty)$, then
  \begin{flalign*}
  \Pr(\Hc(\Tc^{*-})>n|Z^*>0)  = 1-\frac{\Er\left[\Pr(\Hc(\Tc^{*-})\leq n, Z^*>0|Z,Z^*)\right]}{ \Pr(Z^*>0)} = 1-\frac{\Er\left[s_n^{Z-Z^*}\ind_{\{Z^*>0\}}\right]}{1-q}. 
  \end{flalign*}
 We then have that for any $t,s>0$
  \begin{flalign*}
   \Er\left[s^{Z-Z^*}t^{Z^*}\right] & = \Er\left[s^Z\Er[(t/s)^{Z^*}|Z\right] \\
   & = \Er\left[s^Z\sum_{k=0}^Z(t/s)^k\binom{Z}{k}q^{Z-k}(1-q)^k\right] \\
   & = \Er\left[(qs)^Z\left(1+\frac{t(1-q)}{qs}\right)^Z\right] \\
   & = f(sq+t(1-q)).
  \end{flalign*}
Furthermore, \[\Er\left[s^Z\ind_{\{Z^*=0\}}\right]=\Er\left[s^Z\Pr(Z^*=0|Z)\right]=\Er\left[(sq)^Z\right].\]
Therefore writing $t_n:= s_nq+1-q$ we have that $1-t_n=q(1-s_n)$ and \[\Pr^*(\Hc(\Tc^{*-})> n)=1-\frac{f(s_nq+1-q)}{1-q}+\frac{f(s_nq)}{1-q}=\frac{(1-f(t_n))-(q-f(s_nq))}{1-q}.\]
By Taylor we have that $\exists z \in [s_nq,q]$ such that $f(s_nq)=q+qf'(q)(s_n-1)+f''(z)q^2(s_n-1)^2/2$ and since $q<1$ we have that $f''(z)$ exists for all $z\leq q$ and is bounded above by $f''(q)<\infty$.

By Lemma \ref{genform} \[1-f(t_n)= \mu(1-t_n)+\frac{\Gamma(3-\alpha)}{\alpha(\alpha-1)}(1-t_n)^\alpha \overline{L}((1-t_n)^{-1})\] for a slowly varying function $\overline{L}$.
In particular, 
\begin{flalign}
 \frac{\Pr^*(\Hc(\Tc^{*-})> n)}{(1-s_n)} & = \frac{q}{1-q}\left(\frac{1-f(t_n)}{1-t_n}-f'(q)+f''(z)q(s_n-1)/2\right) \notag \\
 & = \frac{q}{1-q}\left(\mu-f'(q)+O((1-s_n)^{\alpha-1}\overline{L}((1-s_n)^{-1}))\right) \label{brnlrg}\\
 & \sim \frac{q(\mu-f'(q))}{1-q} \notag
\end{flalign}
which is the desired result.
 \end{proof}
\end{lem}

Let $B(n):=\bigcap_{i=0}^{\Delta_n^Y}\left\{|\{\Tc \in L_{Y_i}: \Hc(\Tc)\geq h_n\}|\leq 1\right\}$ be the event that any backbone vertex seen up to $\Delta_n$ has at most one $h_n$-trap (which is $C_3(n)$ of \cite{arfrgaha}).
\begin{lem}\label{7.4}
 $\Pb^*(B(n)^c)=o(1)$.
 \begin{proof}
  Using $C_1(n)$ we have that for some constant $C$
 \[   \Pb^*(B(n)^c)  \leq o(1) +Cn\left(\Pr^*(\Hc(\Tc^{*-})\geq h_n)-\Pr^*(|\{\Tc \in L_\rho: \Hc(\Tc)\geq h_n\}|= 1)\right).\]
 Recall $s_{h_n}=\Pr(\Hc(\Tc)< h_n|\Hc(\Tc)<\infty)$ and from (\ref{brnlrg}) we have that \[\Pr^*(\Hc(\Tc^{*-})> h_n)=(1-s_{h_n})\frac{q(\mu-f'(q))}{1-q}+O((1-s_{h_n})^\alpha \overline{L}((1-s_{h_n})^{-1})\]
 for some slowly varying function $\overline{L}$. 
 
 Similarly to the method used in Lemma \ref{6.1} we have that
 \begin{flalign*}
 \Pr^*(|\{\Tc \in L_\rho: \Hc(\Tc)\geq h_n\}|= 1) & = \sum_{k=1}^\infty\sum_{j=1}^{k-1}\Pr^*\left(Z_1=k,Z_1^*=k-j,|\{\Tc \in L_\rho: \Hc(\Tc)\geq h_n\}|= 1\right) \\
 & = \sum_{k=1}^\infty \Pr^*(Z_1=k)\sum_{j=1}^{k-1}\Pr^*(Z_1^*=k-j|Z_1=k)j(1-s_{h_n})s_{h_n}^{j-1} \\
 & =  \sum_{k=1}^\infty \frac{(1-q^k)p_k}{1-q}\sum_{j=1}^{k-1}\binom{k}{j}\frac{q^j(1-q)^{k-j}}{1-q^k}j(1-s_{h_n})s_{h_n}^{j-1} \\
 & =  \frac{q(1-s_{h_n})}{1-q}\sum_{k=1}^\infty kp_k\sum_{j=0}^{k-2}\frac{(k-1)!(qs_{h_n})^j(1-q)^{k-1-j}}{j!(k-1-j)!} \\
 & =  \frac{q(1-s_{h_n})}{1-q}\sum_{k=1}^\infty kp_k\left((qs_{h_n}+1-q)^{k-1}-(qs_{h_n})^{k-1}\right) \\
 & = \frac{q(1-s_{h_n})}{1-q}\left(f'(t_{h_n})-f'(qs_{h_n})\right)
  \end{flalign*}
where $t_{h_n}=qs_{h_n}+1-q$. By Taylor $f'(qs_n)=f'(q)+O(1-s_{h_n})$ as $\nin$. From Lemma \ref{genform} we have that $1-f(t_{h_n})=\mu(1-t_{h_n})+(1-t_{h_n})^\alpha \overline{L}((1-t_{h_n})^{-1})$ for some slowly varying function $\overline{L}$. Applying Theorem 2 of \cite{la} we have that $f'(t_{h_n})=\mu+O((1-t_{h_n})^{\alpha-1} \overline{L}((1-t_{h_n})^{-1}))$. In particular, 
\begin{flalign*}
\Pr^*(|\{\Tc \in L_\rho: \Hc(\Tc)\geq h_n\}|= 1)=\frac{q(1-s_{h_n})}{1-q}\left(\mu-f'(q)+O\left((1-t_{h_n})^{\alpha-1} \overline{L}\left((1-t_{h_n})^{-1}\right)\right)\right)
\end{flalign*}
since $\alpha<2$, thus
\[   \Pb^*(B(n)^c)  \leq o(1) +O(n(1-t_{h_n})^{\alpha} \overline{L}((1-t_{h_n})^{-1})).\]
There exists some constant $c$ such that $1-t_{h_n}\sim qc_\mu f'(q)^{h_n}\leq c n^{-(1-\varepsilon)}$ therefore since $\alpha>1$ we can choose $\varepsilon>0$ small enough (depending on $\alpha$) such that $\Pb^*(B(n)^c)=o(1)$.
 \end{proof}
\end{lem}

Let $D(n)=\left\{\max_{j\leq \Delta_n^Y}\Hc(\Tc^{*-}_{Y_j})\leq 4\log(n)/\log(f'(q)^{-1})\right\}$ be the event that all branches seen before reaching level $n$ are of height at most $4\log(n)/\log(f'(q)^{-1})$.
\begin{lem}\label{maxhgt}
 Under the assumptions of Theorem \ref{suptreethm} \[\Pb^*\left(D(n)^c \right)=O(n^{-2}).\]
 \begin{proof}
  By comparison with a biased random walk on $\Zb$, standard large deviations estimates yield that for $C$ sufficiently large $\Pb(\Delta_n^Y>C_1n)=O(n^{-2})$. Using Lemma \ref{6.1} we have that for independent $\Tc^{*-}_j$
  \begin{flalign*}
   \Pr^*\left(\bigcup_{j=1}^{C_1n}\Hc(\Tc^{*-}_j)>\frac{4\log(n)}{\log(f'(q)^{-1})}\right)  \leq C_1n \Pr^*\left(\Hc(\Tc^{*-})>\frac{4\log(n)}{\log(f'(q)^{-1})}\right)  \leq Cnf'(q)^{\frac{4\log(n)}{\log(f'(q)^{-1})}}  
    = Cn^{-3}.
  \end{flalign*}
 \end{proof}
\end{lem}

\section*{Acknowledgements}
I would like to thank my supervisor David Croydon for suggesting the problem, his support and many useful discussions. This work is supported by EPSRC as part of the MASDOC DTC at the University of Warwick. Grant No.\ EP/H023364/1.

\bibliographystyle{plainnat}
\bibliography{/home/masdoc/Dropbox/Work/Bibliography/ref}           
\end{document}